\documentclass[12pt]{amsart}

\usepackage[normalem]{ulem}
\usepackage{mathtools} 
\usepackage{extarrows} 
\usepackage{xcolor}
\usepackage{natbib}
\usepackage{comment}
\usepackage{enumitem}
\usepackage{tikz}
\usepackage{mathdots}
\usepackage{yhmath}
\usepackage{cancel}
\usepackage{color}
\usepackage{siunitx}
\usepackage{array}
\usepackage{multirow}
\usepackage{gensymb}
\usepackage{tabularx}
\usepackage{extarrows}
\usepackage{booktabs}
\usetikzlibrary{fadings}
\usetikzlibrary{patterns}
\usetikzlibrary{shadows.blur}
\usetikzlibrary{shapes}
\usepackage{caption}
\usepackage{array, caption, floatrow, tabularx, makecell, booktabs}
\usepackage{tikz-cd}
\usepackage{nicematrix}
\usepackage{graphicx}
\graphicspath{ {./images/} }
\usepackage{adjustbox}
\usepackage{longtable}
\usepackage{booktabs}
\usepackage{multirow}
\usepackage{rotating}
\usepackage{amssymb,amsmath,amsfonts,color, amsthm}
\usepackage{graphics}
\usepackage{subcaption}
\usepackage{niceframe}
\usepackage{threeparttable}
\usepackage{nicematrix}
\usepackage{makecell}
\usepackage{hyperref}



\newtheorem{theorem}{Theorem}[section]
\newtheorem{prop}[theorem]{Proposition}
\newtheorem{lemma}[theorem]{Lemma}
\newtheorem{corollary}[theorem]{Corollary}
\newtheorem{definition}[theorem]{Definition}
\newtheorem{remark}[theorem]{Remark}
\newtheorem{Construction}[theorem]{Construction}

\newtheorem{example}[theorem]{Example}
\newtheorem{notation}[theorem]{Notation}
\newtheorem{op}[theorem]{Open Problem}

\def\O{\mathcal{O}}

\newcommand{\rmv}[1]{}

\def\floor#1{\lfloor#1\rfloor}
\def\ceil#1{\lceil#1\rceil}

\newcommand{\nextcomb}[1]{\overline{\overline{#1}}}

\title{Cover-free families on graphs}

\author{Prangya Parida}
\address{
Department of Mathematics and Statistics,
University of Ottawa,
Ottawa, Canada
}

\author{Lucia Moura}
\address{
School of Electrical Engineering and Computer Science,
University of Ottawa,
Ottawa, Canada
}

\begin{document}

\maketitle

\begin{abstract}
A family of subsets of a $t$-set is a \emph{$d$-cover-free family} or $d$-CFF if no subset in the family is contained in the union of any $d$ other subsets. Let $t(d, n)$ denote the minimum $t$ for which there exists a $d$-CFF on a $t$-set with $n$ subsets. Since a $1$-CFF is the same as a Sperner family, using Sperner's theorem, we get $t(1, n) \sim \log_{2}(n)$ as $n$ grows. Erd\H{o}s, Frankl, and F\"uredi (JCTA, 1982) proved that $3.106\log_{2}(n) < t(2,n) < 5.512\log_{2}(n)$. This paper focuses on generalizing $1$-CFF and $2$-CFF using a graph $G$ where vertices correspond to subsets in the set system. A $G$-Sperner$(t, n)$ is a family of subsets of a $t$-set such that each edge of $G$ specifies a pair of subsets not contained in each other, whereas a $G$-CFF$(t, n)$ is a family of subsets of a $t$-set such that it is $G$-Sperner and the union of a pair of subsets corresponding to each edge of $G$ does not contain any other subset in the family. Let $t_s(G)$ and $t(G)$ denote the minimum $t$ for which there exist a $G$-Sperner$(t, n)$ and a $G$-CFF$(t, n)$, respectively. In this way, $t_s(K_n) = t(1, n)$ and $t(K_n) = t(2, n)$. In this paper, we prove that $t_s(G) = t(1, \chi(G))$ for any simple graph $G$ and provide various upper and lower bounds for $t(G)$. The \emph{trivial bounds}, $t(1, n) \leq t(G) \leq t(2, n)$, hold for any simple graph $G$ with no isolated vertex, with the lower bound being tight for an infinite family of star graphs and the upper bound being tight for complete graphs. We study when these bounds can be improved and give better constructive upper bounds for families of graphs such as stars, paths, cycles, wheels, and windmill graphs. In particular, a construction based on mixed-radix Gray codes yields the upper bound in $\log_{2}(n) \leq t(P_n) \leq t(C_n) \leq 1.893\log_{2}(n) + \O(1)$, where $P_n$ is a path and $C_n$ is a cycle with $n$ vertices.
\end{abstract}

\section{Introduction}

Cover-free families (CFFs) were first introduced by Kautz and Singleton \cite{KS} as \emph{superimposed codes}. Since then, it has been actively studied under different names, such as \emph{families of finite sets in which no set is covered by the union of $d$ others}\cite{erdos1982families, EFF} and \emph{disjunct matrices}\cite{DH}. Cover-free families have several applications to combinatorial group testing (first studied by Hwang and Sós \cite{HS}), cryptographic problems, and digital communications. Also, there have been various generalizations of CFFs, for these applications \cite{dalai2025efficient, gargano2018low, IM1, thaismourastructureaware, rescigno2023bounds, SW, SWZ, W}. For more information, we refer the reader to the survey \cite{IM2}.

A  \emph{set system} $\mathcal{F}$ is a pair $(\mathcal{X}, \mathcal{B})$ where $\mathcal{X}$ is a set of points and $\mathcal{B}$ is a set of blocks or subsets of $\mathcal{X}$. The \textit{incidence matrix} \( \mathcal{M} \) of a set system  $(\mathcal{X}, \mathcal{B})$  is a binary matrix with rows corresponding to elements of $\mathcal{X}$ and columns corresponding to elements of $\mathcal{B}$, such that $\mathcal{M}_{ij} = 1$ if and only if $x_i \in B_j$, for all   $x_i \in \mathcal{X}$  and $B_j \in \mathcal{B}$. Next, we define cover-free families. 

\begin{definition}
\label{def:1}
    For positive integers $d < t \leq n$, a set system $\mathcal{F = (\mathcal{X}, \mathcal{B})}$ with $|\mathcal{X}| = t$ and $|\mathcal{B}| = n$ is a \emph{$d$-cover-free family}, denoted by $d$-CFF$(t,n)$, if for any $d+1$ subsets $B_{i_0}, B_{i_1}, \ldots, B_{i_d}$,
\begin{equation}
\label{cff-eq}
    B_{i_0} \nsubseteq B_{i_1} \cup B_{i_2} \cup \cdots \cup B_{i_d}.
\end{equation}
\end{definition}
    
When $d = 1$, then a $1$-CFF$(t, n)$ is also called a \emph{Sperner family}, which is a set system where sets do not contain one another. We denote by $t(d,n)$ the minimum value $t$ for which there exists a $d$-CFF$(t,n)$, that is,
 $$t(d,n) = \min\{t : \exists \text{ a } d\text{-CFF}(t,n)\}.$$
 
We say that a $d$-CFF$(t, n)$ with $t = t(d, n)$ is \emph{optimal}. As a consequence of Sperner's Theorem, we get $t(1, n) = \min\left \{t : \binom{t}{\floor{\frac{t}{2}}}  \geq n \right \}$ and it is well-known that $t(1, n) \sim \log_{2}n$ as $n \rightarrow \infty$ (see \cite{parida2026cover} for a proof). We provide more relevant details of $t(1, n)$ in Section~\ref{more-1-CFFs}. 
In general, we have:
\begin{equation}
\label{simple-additive-bound}
    t(d, n+1) \leq t(d, n) + 1,
\end{equation}
since from any $d$-CFF$(t, n)$, $\mathcal{F} = ([1, t], \mathcal{B})$, we can build $\mathcal{F}' = ([1, t+1], \mathcal{B} \cup \{\{t+1\}\})$, which is a $d$-CFF$(t+1, n+1)$. 

The purpose of this article is to study generalized versions of $1$-CFFs and $2$-CFFs that extend these definitions to include a graph $G$, where vertices of $G$ correspond to sets in the set system. 
Such extensions have been introduced for general $d$ using hypergraphs \citep{IM1,thaismourastructureaware}, where they relate these objects to the application in community-aware group testing~\citep{NikolopoulosSGFD23}.
In the present paper, we do an in-depth study of the graph case, which in group testing corresponds to detecting ``defects" or ``infections" that must be contained in some edge of the graph.

A $G$-Sperner$(t, n)$ is a set system $([1, t], \mathcal{B})$ such that the sets corresponding to endpoints of an edge do not contain one another (Definition~\ref{def:G-sperner}). In this sense, a $K_n$-Sperner$(t, n)$ coincides with a $1$-CFF$(t, n)$ or a Sperner family. An \emph{edge cover-free family} on a graph $G$, denoted by $G$-ECFF$(t, n)$, is a set system $([1, t], \mathcal{B})$ such that for any $ e \in E(G)$, the union of the sets corresponding to vertices in $e$ does not contain any other set in the set system. A cover-free family on a graph $G$, denoted by $G$-CFF$(t, n)$, is a set system $([1, t], \mathcal{B})$ such that for any $ \emptyset \neq e' \subseteq e \in E(G)$, the union of the sets corresponding to vertices in $e'$ does not contain any other set in the set system. Thus, a $K_n$-CFF$(t, n)$ is equivalent to a $2$-CFF$(t, n)$. Furthermore, being $G$-CFF is equivalent to being both $G$-ECFF and $G$-Sperner.  We denote by $t_{s}(G)$, $t_e(G),$ and $t(G)$ the minimum value of $t$ for which there exists a $G$-Sperner$(t, n)$, a $G$-ECFF$(t, n)$, and a $G$-CFF$(t, n)$, respectively, where $n = |V(G)|$. Clearly, $t_s(K_n) = t(1, n)$ and $t(K_n) = t(2, n)$. It is also known that $t_e(K_n) = t(K_n) = t(2, n)$ \cite{IM1, thaismourastructureaware}, but in general, we may have $t_e(G) \neq t(G)$ for some $G$ (for example, $t_e(S_n) = t(S_n) - 1$, for a star graph $S_n$ [Corollary~\ref{star-optimum}]). 
Let $L_n$ be a graph with $n$ vertices and the edge set consisting of loops on all vertices. Then, $t(L_n) = t_e(L_n) = t(1, n)$.

 In this paper, we first prove the exact value of $t_{s}(G)$ for a simple graph $G$, that is, $t_{s}(G) = t(1, \chi(G))$ where $\chi(G)$ is the chromatic number of $G$ (Corollary~\ref{optimal-H-in}). We achieve this by providing a characterization of $G$-Sperner families via graph homomorphisms to a special graph, which we call a \emph{Sperner graph}. 

We further show that a $G$-CFF$(t, n)$ provides a rich combinatorial structure that lies between being a $1$-CFF and a $2$-CFF in terms of bounds, giving the trivial bounds: $t(1, n) \leq t(G) \leq t(2, n)$ (see Corollary~\ref{trivial-bound}), for any graph $G$ with no isolated vertices. We exemplify this claim in Figure~\ref{fig:GraphComparison} for a cycle of length $12$ $(C_{12})$ where the incidence matrix of a $C_{12}$-CFF$(7, 12)$ is given.
\begin{figure}
    \centering
    \includegraphics[width=0.8\linewidth]{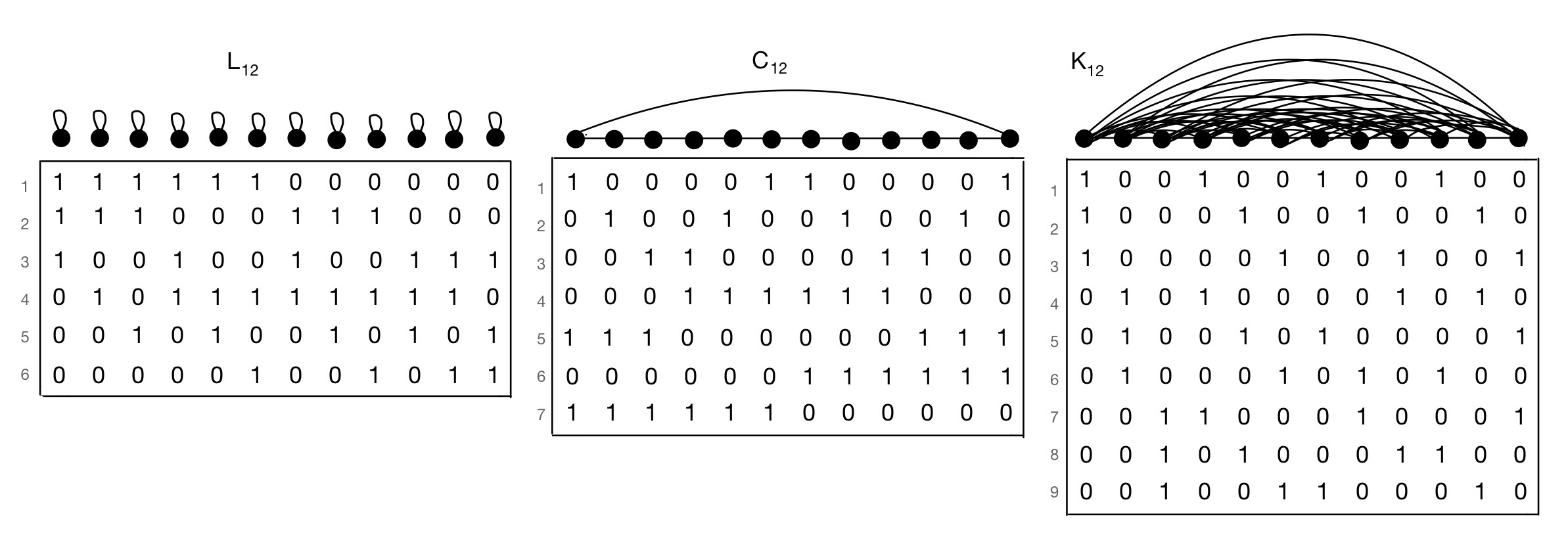}
    \caption{$t(L_{12}) = t(1, 12) = 6, 6 \leq t(C_{12}) \leq 7$, and $t(K_{12}) = t(2, 12) = 9$}
    \label{fig:GraphComparison}
\end{figure}
Using the known values of $t(1, 12) = 6$ and $t(2, 12) = 9$ \cite{LVW}, we obtain the trivial bounds $6 \leq t(C_{12}) \leq 9$, and the existence of a $C_{12}$-CFF$(7, 12)$ shows that the trivial upper bound is not always tight, and that $6 \leq t(C_{12}) \leq 7$. This observation motivates us to study conditions under which the trivial upper bound can be improved upon and finding which families of graphs meet either of these extreme bounds. We also find some families of graphs where $t(G)$ is close to the trivial lower bound. 

From the above definitions and $t_s(G) = t(1, \chi(G))$, it is easy to see that $t_e(G) \leq t(G) \leq t_e(G) + t_s(G) = t_e(G) + t(1, \chi(G))$. Thus, we further investigate under which conditions $t(G)$ is close to $t_e(G)$, even when $\chi(G)$ may be large. Surprisingly, for any graph $G$ with no isolated vertex, we show that $t_e(G) \leq t(G) \leq t_e(G) + 2$ (Theorem~\ref{GbarCFF=GCFF}).

Tables~\ref{tab:lowerbound} and \ref{tab:upperbound} summarize various lower and upper bounds on $t(G)$ proven in this paper, along with infinite families of graphs achieving these bounds. We assume that $G$ is a graph with no isolated vertex and $|V(G)| = n \geq 3$. We denote by $S_n$ the star graph with $n$ vertices (see Definition~\ref{def:star}), by $\mathcal{E}_n$ with $n$ even, the graph with $\frac{n}{2}$ single edge components, by $G + \{x\}$ the graph $G$ with the addition of a universal vertex $x$ (see Notation~\ref{G+x}), by $K_n \cup K_2$ the graph $K_n$ with a pendant vertex attached, and by $K_n \mathbin{\dot\cup} K_2$ the graph $K_n$ with an edge component (see Notation~\ref{not:graphK_n-operations}).  

\begin{table}[t]
\centering
\begin{threeparttable}
\begin{tabular}{|c|c|c|}
\hline
Lower bound for $t(G)$ & Extra assumptions for $G$ & \makecell{Lower bound is tight \\ infinitely often} \\ \hline
$t(1, n) \leq t(G)$ & --- & \makecell{ \\ 1. $S_n$ with $n = \binom{x}{\floor{\frac{x}{2}}} + 1$ \\ (Corollary~\ref{star-meeting-trivial-lower-bound}) \\ \\ 2. $\mathcal{E}_{2m}$ with $m > \floor{\tfrac{1}{2} \nextcomb{(\nextcomb{m} + 1)}}$ \\ (Proposition~\ref{matching}) \\} \\ \hline
\makecell{$t(1, n) + 1 \leq t(G)$ \\ (Corollary~\ref{graph-with-vertices-central-binomial})} & $n = \binom{x}{\floor{\frac{x}{2}}}$ & \makecell{\\ $\mathcal{E}_{\binom{x}{\floor{\frac{x}{2}}}}$ with $\binom{x}{\floor{\frac{x}{2}}}$ even  \\ (Theorem~\ref{matching-sperner})}\\ \hline
$t_{e}(G) \leq t(G)$ & \makecell{ --- } & \makecell{\\ 1. $\delta(G) \geq 2$ \\ (Theorem~\ref{GbarCFF=GCFF} part~\ref{delta=2}) \\ \\ 2. $K_n \cup K_2$ with $t(2, n) = t(2, n+1)$ \\ (Corollary~\ref{cor:trivial-t(G)=t_e(G)}) \\ \\ 3. $K_n \mathbin{\dot\cup} K_2$ with $t(2, n) = t(2, n+2)$ \\ (Corollary~\ref{cor:trivial-t(G)=t_e(G)})} \\ \hline
\makecell{$t_e(G) + 1 \leq t(G)$ \\ (Proposition~\ref{matching})} & $G = \mathcal{E}_n$ with $n$ even & \makecell{ \\ $\mathcal{E}_8$\tnote{$\dagger$} \\ (Proposition~\ref{example-E8})} \\ \hline
\end{tabular}
\caption{Lower bounds for $t(G)$}
\label{tab:lowerbound}
\begin{tablenotes}
\item[$\dagger$] The bound is tight for $\mathcal{E}_8$, not known if infinitely often. 
\end{tablenotes}
\end{threeparttable}
\end{table}

Table~\ref{tab:lowerbound} shows an infinite family of graphs that meets the trivial lower bound (first row) and an infinite family of graphs that does not meet the trivial lower bound (second row). It is also shown that $\delta(G) \geq 2$ is a sufficient condition for $t_e(G) = t(G)$ (third row) and that $\mathcal{E}_{n}$ is an infinite family of graphs not meeting the lower bound $t_e(G) \leq t(G)$ (fourth row). 

\begin{table}[t]
\centering
\begin{threeparttable}
\begin{tabular}{|c|c|c|}
\hline
Upper bound for $t(G)$ & Extra assumptions for $G$ & \makecell{Upper bound is tight \\ infinitely often} \\ \hline
\makecell{$t(G) \leq t(2, n)$ \\ (Corollary~\ref{trivial-bound})} & --- & $K_n$ \\ \hline

\makecell{$t(G) \leq t_e(G) + 2$ \\ (Theorem~\ref{GbarCFF=GCFF} part~\ref{t_e(G)+2})} & $\delta(G) = 1$ & \makecell{ \\ 1. $\mathcal{E}_{2m}$ with $m > \floor{\tfrac{1}{2} \nextcomb{(\nextcomb{m} + 1)}}$ \\ (Proposition~\ref{matching}) \\ \\ 2. $\mathcal{E}_{n}$ with $n = \binom{x}{\floor{\frac{x}{2}}}, n$ even \\ (Theorem~\ref{matching-sperner})} \\ \hline

\makecell{$t(G) \leq t_{e}(G) + 1$ \\ (Theorem~\ref{GbarCFF=GCFF} part~\ref{t_e(G)+1})} & \makecell{$\delta(G) = 1$ with no \\ single edge components} & \makecell{\\ $S_n$ (Corollary~\ref{star-optimum})}  \\ \hline

\makecell{$t(G) \leq t(H) + 1$ \\ (Theorem~\ref{star-cons-gen})} & \makecell{$G = H + \{x\}$ \\ with $\delta(H) \geq 1$} & \makecell{ \\ 1. $H$ with $t(H) = t(1, |V(H)|) + 1$ \\ (Corollary~\ref{star-gen-cons-special-case}) \\ \\ 2. $H = S_n$ (Corollary~\ref{Sn+x-optimum})} \\ \hline
\end{tabular}
\caption{Upper bounds for $t(G)$}
\label{tab:upperbound}
\end{threeparttable}
\end{table}

Table~\ref{tab:upperbound} shows various upper bounds for $t(G)$. We show that there are two infinite families of $\mathcal{E}_{n}$, meeting the upper bound of $t(G) \leq t_e(G) + 2$. If $\delta(G) = 1$ and $G$ has no connected component as $K_{2}$, then we have a tighter upper bound $t(G) \leq t_e(G) + 1$ and $S_n$ is an infinite family of graphs meeting this upper bound. Recall from Table~\ref{tab:lowerbound} that $t(G) = t_e(G)$ whenever $\delta(G) \geq 2$. The bound $t(H + \{x\}) \leq t(H) + 1$ generalizes the bound in Equation~\ref{simple-additive-bound} and is tight for any $H$ with $t(H) = t(1, |V(H)|) + 1$, as well as for $S_n$. 

Finally, we obtain bounds on $t(G)$ for specific classes of graphs $G$, as provided in Table~\ref{tab:graph bounds}. We denote by $P_n, C_n,$ and $W_n$, the path, cycle, and wheel graph of $n$ vertices, respectively. $\mathrm{Wd}(k, n)$ refers to the windmill graph defined in Definition~\ref{def:windmill}. For $G = S_n$ and $G = S_n + \{x\}$, the value of $t(G)$ is determined. Similarly, for $G = \mathcal{E}_{2m}, G = \mathrm{Wd}(3, n)$, and $G = W_{n+1}$, the value of $t(G)$ is determined up to a difference of $1$. For $G = \mathrm{Wd}(k, n)$ with $k \geq 4$, upper and lower bounds are provided. For $G = C_n$, an upper bound is provided in Theorem~\ref{maximal-path-cycle-thm} using a construction involving mixed-radix Gray codes \cite{knuth2011art}. This improves the upper bound $t(C_n) \leq 2\log_{2}(n)$, which is a consequence of a construction by Idalino and Moura \cite{thaismourastructureaware}, to $t(C_n) \leq 1.893\log_{2}(n) + \O(1)$ (see Corollary~\ref{asymptotics-t(Cn)}). It further results in having upper bounds for $t(P_n)$ and $t(C_n)$ that are relatively close to the trivial lower bound, that is, we have $\log_{2}(n) \leq t(P_n) \leq t(C_n) \leq 1.893\log_{2}(n) + \O(1)$ as $n$ grows. 

\begin{table}[t]
\centering
\begin{threeparttable}
\begin{tabular}{|c|c|}
\hline
$G$ & $t(G)$ \\ \hline
\makecell{$S_n$, $n \geq 3$ \\ (Corollary~\ref{star-optimum})} & $t(S_n) = t(1, n-1) + 1$ \\ \hline
\makecell{$S_n + \{x\}$, $n \geq 3$ \\ (Corollary~\ref{star-gen-cons-special-case})} & $t(S_n + \{x\}) = t(1, n-1) + 2$ \\ \hline
\makecell{$\mathcal{E}_{2m}$, $m \geq 2$ \\ (Proposition~\ref{matching})} & \makecell{\\
    $t(1, m) + 1 \leq t(\mathcal{E}_{2m}) \leq t(1, m) + 2$\\
     (the upper bound is tight for $m > \lfloor \tfrac{1}{2} \nextcomb{(\nextcomb{m} + 1)} \rfloor$)} \\ \hline
\makecell{$P_n, C_n, n\geq 3$ \\ (Theorem~\ref{maximal-path-cycle-thm})} & $t(P_n) \leq t(C_n) \leq \begin{cases}
    3k & \text{if } n \in (2 \cdot 3^{k-1}, 3^k] \text{ for some } k \geq 1, \\
    3k + 1 & \text{if } n \in (3^k, 4 \cdot 3^{k-1}] \text{ for some } k \geq 1, \\
    3k + 2 & \text{if } n \in (4 \cdot 3^{k-1}, 2 \cdot 3^k] \text{ for some } k \geq 1.
\end{cases}$ \\ \hline
\makecell{$\mathrm{Wd}(k, n)$, $n \geq 2, k \geq 4$ \\ (Theorem~\ref{windmill-cff-bounds})} & $t(1, (k-1)n) + 1 \leq t(\mathrm{Wd}(k, n)) \leq t(1, n) +  t(2, k-1) + 1$ \\ \hline

\makecell{$\mathrm{Wd}(3, n)$, $n \geq 2$ \\ (Corollary~\ref{friendship})} & \makecell{ \\ $t(1, n) + 2 \leq t(\mathrm{Wd}(3, n)) \leq t(1, n) +  3$ \\ (the upper bound is tight for $n > \left\lfloor \frac{1}{2} \nextcomb{(\nextcomb{n} + 1)} \right\rfloor$)} \\ \hline
\makecell{$W_{n+1}, n \geq 3$ \\ (Corollary~\ref{wheel})} & $\max\{t(C_n), t(C_{n+1})\} \leq t(W_{n+1}) \leq t(C_n) + 1$ \\ \hline
\end{tabular}
\caption{Bounds for $t(G)$ for specific families of graphs $G$}
\label{tab:graph bounds}
\end{threeparttable}
\end{table}

For small values of $n$, we are able to determine $t(P_n), t(C_n)$, and $t(W_n)$, as given in Proposition~\ref{small-cycles-234}, Theorem~\ref{small-cycles-678}, and Corollary~\ref{wheel-small}. In Table~\ref{tab:small-n-path-cycle-wheel}, we display the obtained exact values for $t(P_n), t(C_n)$, and $t(W_n)$ or their upper bounds for $n \leq 10$. We also display the exact values of $t(K_n) = t(2, n)$, to show how far the previous values are from the trivial upper bound.
\begin{table}[t]
\centering
\begin{threeparttable}
\begin{tabular}{|c|c|c|c|c|}
\hline
$n$ & $t(P_n)$ & $t(C_n)$ & $t(W_n)$ & $t(K_n)$\cite{LVW} \\ \hline
$3$ & $\mathbf{3}$ & $\mathbf{3}$ & $\mathbf{3}$ & $\mathbf{3}$  \\ \hline
$4$ & $\mathbf{4}$ & $\mathbf{4}$ & $\mathbf{4}$ & $\mathbf{4}$ \\ \hline
$5$ & $\mathbf{5}$ & $\mathbf{5}$ & $\mathbf{5}$ & $\mathbf{5}$ \\ \hline
$6$ & $\mathbf{5}$ & $\mathbf{5}$ & $\mathbf{6}$ & $\mathbf{6}$ \\ \hline
$7$ & $\mathbf{6}$ & $\mathbf{6}$ & $\mathbf{6}$ & $\mathbf{7}$ \\ \hline
$8$ & $\mathbf{6}$ & $\mathbf{6}$ & $\mathbf{7}$ & $\mathbf{8}$ \\ \hline
$9$ & $\mathbf{6}$ & $\mathbf{6}$ & $\mathbf{7}$ & $\mathbf{9}$ \\ \hline
$10$ & $\mathbf{6}$ & $7$ & $\mathbf{7}$ & $\mathbf{9}$ \\ \hline
\end{tabular}
\caption{Upper bounds for $t(P_n)$, $t(C_n)$, $t(W_n)$, and $t(K_n)$, where the bold entries indicate exact values.}
\label{tab:small-n-path-cycle-wheel}

\end{threeparttable}
\end{table}

 The structure of the paper is as follows. Section~\ref{bacground} provides the necessary background on Sperner families, cover-free families, and graph theory. Section~\ref{G-sperner} focuses on $G$-Sperner families and provides all the necessary technical results before proving its main result, that is, $t_s(G) = t(1, \chi(G))$. In Section~\ref{CFFonGraphs}, we provide formal definitions of cover-free families on graphs, along with lower and upper bounds on $t(G)$, as summarized in Tables~\ref{tab:lowerbound} and \ref{tab:upperbound}. In Section~\ref{Graphs-meeting-extreme-bounds}, we present infinite families of graphs $G$ for which $t(G)$ meets the extreme bounds proven in Section~\ref{CFFonGraphs}. In Section~\ref{cff-on-paths-cycles}, we provide an introduction to mixed-radix Gray codes and describe a construction of CFFs on paths and cycles using such Gray codes, which yields the upper bound in $\log_{2}(n) \leq t(P_n) \leq t(C_n) \leq 1.893 \log_{2} (n)+ \O(1)$. 
Section~\ref{star-gen-cons} generalizes the construction of CFFs on star graphs provided in Section~\ref{Graphs-meeting-extreme-bounds} to obtain bounds (tight bounds in some cases) for windmill graphs and friendship graphs. 
In Section~\ref{CFFonSmallPathsCycles}, we provide the results for exact values of $t(P_n), t(C_n),$ and $t(W_n)$ for small values of $n$, as provided in Table~\ref{tab:small-n-path-cycle-wheel} (in \textbf{bold}).
Conclusions and future work are provided in Section~\ref{future}.

\section{Background}
\label{bacground}

Throughout the paper, for $b > a$, $[a,b]$ denotes the set $\{a,a+1,\ldots,b\}$.

\subsection{Cover-free families}

Cover-free family has been defined in Definition~\ref{def:1}.
For $d < t \leq n$ positive integers, let $\mathcal{M}$ be a $t \times n$ binary matrix. It is easy to see that $\mathcal{M}$ is the incidence matrix of a $d$-CFF$(t,n)$ if and only if  any $t \times (d+1)$ submatrix contains a permutation submatrix of order $d+1$. 

The \emph{Boolean sum} of two binary vectors is the resultant vector obtained by taking their coordinate-wise OR, that is, an entry in the resultant vector is $1$ if and only if at least one of the corresponding entries in the two original vectors is $1$. For two binary vectors $\overrightarrow{v}$ and $\overrightarrow{w}$, we say that $\overrightarrow{v}$ \emph{covers} $\overrightarrow{w}$, denoted by $\overrightarrow{v} \supseteq \overrightarrow{w}$, if every position where $\overrightarrow{w}$ has a $1$, $\overrightarrow{v}$ also has a $1$.
Du and Hwang \cite{DH} define $d$-disjunct and $\overline{d}$-disjunct matrices. A binary $t \times n$ matrix $\mathcal{M}$ is called a \emph{$\overline{d}$-disjunct} matrix if the Boolean sum of any up to $d$ columns does not cover any other column. Similarly, a binary $t \times n$ matrix $\mathcal{M}$ is called a \emph{$d$-disjunct} matrix if the Boolean sum of any $d$ columns does not cover any other column. This definition also implies that the Boolean sum of up to any $d$ columns does not cover any other column, which tells us that a matrix is $d$-disjunct if and only if it is $\overline{d}$-disjunct. The following proposition is an immediate consequence of the above definitions.

\begin{prop}
   For $d < t \leq n$, a $t \times n$ binary matrix is $d$-disjunct if and only if it is the incidence matrix of a $d$-CFF$(t, n)$. 
\end{prop}

Cover-free families generalize the concept of Sperner families. A family of subsets $\mathcal{F}$ of an $n$-set has the \textit{Sperner property} if no subset in the family is contained in any other, such a family is called a $\textit{Sperner family}$. The upper bound on the maximum size of a Sperner family on a given $n$-set was proved by Sperner~\cite{S}.

\begin{theorem}\cite{S}
\label{sperner}
     If $\mathcal{F}$ is a Sperner family on  $[1, t]$, then $|\mathcal{F}| \leq \binom{t}{\floor{\frac{t}{2}}}$. The upper bound is only achieved by the following two families: the family of all $\floor{\frac{t}{2}}$-subsets and the family of all $\ceil{\frac{t}{2}}$-subsets of the $t$-set.
\end{theorem}

Thus, a $1$-CFF is equivalent to a Sperner family. Sperner's theorem \cite{S} gives an optimal construction for $1$-CFF which can be done by taking $n$ distinct sets out of all $ \floor{\frac{t}{2}}$-subsets of a $t$-set for the smallest possible $t$. Thus, for any $n \geq 2$, 
\begin{equation}
\label{max-sperner-family}
    t(1, n) = \min\left\{t : \binom{t}{\floor{\frac{t}{2}}}  \geq n \right\}.
\end{equation}
Equation~\ref{max-sperner-family} implies that $t(1, n) = t(1, n-1) + 1$ if $n = \binom{x}{\lfloor{\frac{x}{2}}\rfloor} + 1$ for some $x \in \mathbb{N}$, and $t(1, n) = t(1, n-1)$, otherwise. For $d \geq 2$, the best known lower bound and upper bound on $t$ is given by $$c_1 \frac{d^2}{\log_{2}(d)} \log_{2}(n) \le t(d, n) \le c_2 d^2 \log_{2}(n),$$ for some constants $c_1$ and $c_2$. The value of the constant $c_1$ is proven to be approximately $\frac{1}{2}$, $\frac{1}{4}$ and $\frac{1}{8}$ in \cite{DR}, \cite{F}, and \cite{R} respectively. Furthermore, the existence of $d$-CFFs based on the probabilistic method produce the best known upper bounds on $t(d, n)$ and derandomization techniques can be used to obtain polynomial time algorithms to construct a $d$-CFF$(t, n)$ with $t = \Theta(d^2 \log_{2}n)$ \cite{bshouty2015linear, gargano2018low, PR}. 
In this paper, we focus on generalizations of $2$-CFFs. Erd\H{o}s, Frankl, and F\"uredi \cite{erdos1982families} provided the following lower and upper bounds on $t(2, n)$:
    $$3.106\log_{2}(n) < t(2,n) < 5.512\log_{2}(n).$$
Table~\ref{tab:t2n-small} gives upper bounds on $t(2, n)$ for small values of $n$ \cite{IM2,LVW} and the bold entries indicate the bound is exact.

\begin{table}[h!]
\centering
\begin{tabular}{|c|c|c|c|c|c|c|c|c|c|c|c|c|c|c|c|c|}
\hline
$n$ & $\mathbf{3 \leq y \leq 8}$ & $\mathbf{12}$ & $\mathbf{13}$ & $\mathbf{17}$ & $20$ & $26$ & $28$ & $42$ & $48$ & $68$ & $69$ & $76$ & $90$ & $120$ & $176$ & $253$ \\
\hline
$t(2, n)$ & $\mathbf{y}$ & $\mathbf{9}$ & $\mathbf{10}$ & $\mathbf{11}$ & $12$ & $13$ & $14$ & $15$ & $16$ & $17$ & $18$ & $19$ & $20$ & $21$ & $22$ & $23$ \\
\hline
\end{tabular}
\caption{Upper bounds of $t(2, n)$ for small values of $n$}
\label{tab:t2n-small}
\end{table}

\subsection{More on 1-CFFs}
\label{more-1-CFFs}
 
This section provides auxiliary results on Sperner families for subsequent use in the paper. The following technical propositions are used later to prove Proposition~\ref{matching}, Theorem~\ref{matching-sperner}, and Corollary~\ref{friendship}. 
\begin{prop}
\label{sperner-technical}
    For $k \geq 2$, we have $t\left(1, \left\lfloor \frac{1}{2} \binom{2k+1}{k} \right\rfloor\right) = 2k$.
\end{prop}

\begin{proof}
Since $\binom{2k + 1}{k} = \binom{2k + 1}{k + 1} = \frac{2k + 1}{k + 1} \binom{2k}{k} < 2 \binom{2k}{k}$, we have $\left \lfloor \frac{1}{2} \binom{2k + 1}{k} \right \rfloor \leq \frac{1}{2} \binom{2k + 1}{k} < \binom{2k}{k}$. This implies $t \left(1, \left \lfloor \frac{1}{2} \binom{2k + 1}{k} \right \rfloor \right) \leq 2k$. It remains to show that $t \left(1, \left \lfloor \frac{1}{2} \binom{2k + 1}{k} \right \rfloor \right) > 2k-1$. Since $\binom{2k + 1}{k} = \binom{2k}{k} + \binom{2k}{k-1}$ and $\binom{2k}{k} = \frac{2k}{k} \binom{2k - 1}{k - 1} = 2 \binom{2k - 1}{k - 1}$, we have $\binom{2k + 1}{k} = 2\binom{2k - 1}{k - 1} + \binom{2k}{k - 1}$. In addition, since $k \geq 2$, this implies $2k > k - 1$ and $\binom{2k}{k - 1} > \binom{2k}{1} \geq 4$. Thus, $\left \lfloor \frac{1}{2} \binom{2k + 1}{k} \right \rfloor = \left \lfloor \frac{1}{2} \left( 2 \binom{2k - 1}{k - 1} + \binom{2k}{k - 1} \right) \right \rfloor = \left \lfloor  \binom{2k - 1}{k - 1} + \frac{1}{2} \binom{2k}{k - 1} \right \rfloor = \binom{2k - 1}{k - 1} + \left \lfloor \frac{1}{2} \binom{2k}{k - 1} \right \rfloor$. Since $\frac{1}{2} \binom{2k}{k - 1} \geq 2$, we have $\left \lfloor \frac{1}{2} \binom{2k + 1}{k} \right \rfloor \geq \binom{2k - 1}{k - 1} + 2$. Thus, we conclude that, by Equation~\ref{max-sperner-family}, $t \left(1, \left \lfloor \frac{1}{2} \binom{2k + 1}{k} \right \rfloor \right) > 2k-1$. 
\end{proof}

\begin{lemma}
\label{sperner-odd-odd+1}
    If $t(1,n) = 2k + 1$ for some $k \geq 1$, then $t(1, 2n) = 2k + 2$.
\end{lemma}

\begin{proof}
If $t(1, n) = 2k + 1$, then by Equation~\ref{max-sperner-family}, $\binom{2k}{k} < n \leq \binom{2k + 1}{k}$. This implies $2n \leq 2 \binom{2k + 1}{k} = \frac{2k + 2}{k + 1} \binom{2k + 1}{k} = \binom{2k + 2}{k + 1}$. Since $n \geq \binom{2k}{k} + 1$, we have $2n \geq 2 \binom{2k}{k} + 2 > \binom{2k}{k} + \binom{2k}{k - 1} + 2 = \binom{2k + 1}{k} + 2$. Thus, $\binom{2k + 1}{k} < 2n \leq \binom{2k + 2}{k + 1}$. This implies $t(1, 2n) = 2k + 2$.
\end{proof}

\begin{lemma}
\label{sperner-even}
   Suppose $t(1, n) = 2k$ for $k \geq 2$.
   \begin{enumerate}
       \item \label{sperner-even-+1} If $n \in \left[\binom{2k-1}{k} + 1, \left\lfloor \frac{1}{2} \binom{2k+1}{k} \right\rfloor \right]$, then $t(1, 2n) = t(1, n) + 1$.
       \item \label{sperner-even-+2} If $n \in \left[\left\lfloor \frac{1}{2} \binom{2k+1}{k} \right\rfloor + 1, \binom{2k}{k} \right]$, then $t(1, 2n) = t(1, n) + 2$.
   \end{enumerate}
\end{lemma}

\begin{proof}
If $t(1, n) = 2k$, then by Equation~\ref{max-sperner-family}, $n \in \left[\binom{2k-1}{k-1} + 1, \binom{2k}{k}\right]$. This implies $2 \binom{2k-1}{k-1} + 2 \leq 2n \leq 2 \binom{2k}{k}$. Thus, $2n \geq 2 \binom{2k-1}{k-1} + 2 = \frac{2k}{k} \binom{2k - 1}{k - 1} + 2 = \binom{2k}{k} + 2$. Hence, $t(1, 2n) > 2k$. Now we provide a proof for items~\ref{sperner-even-+1} and \ref{sperner-even-+2}. 

\begin{enumerate}
    \item It is enough to show that $t\left(1, 2\left\lfloor \frac{1}{2} \binom{2k+1}{k} \right\rfloor\right) = 2k+1$. Let $x = 1$ if $\binom{2k + 1}{k}$ is odd and $x = 0$ if $\binom{2k + 1}{k}$ is even. Then, $\left\lfloor \frac{1}{2} \binom{2k+1}{k} \right\rfloor = \frac{\binom{2k + 1}{k} - x}{2}$. Thus, $t\left(1, 2 \left\lfloor \frac{1}{2} \binom{2k+1}{k} \right\rfloor \right) = t\left(1, \binom{2k + 1}{k} - 1 \right)$ if $\binom{2k + 1}{k}$ is odd and $t\left(1, 2 \left\lfloor \frac{1}{2} \binom{2k+1}{k} \right\rfloor \right) = t\left(1, \binom{2k + 1}{k} \right)$ if $\binom{2k + 1}{k}$ is even. In both cases, it follows by Equation~\ref{max-sperner-family} that $t\left(1, 2 \left\lfloor \frac{1}{2} \binom{2k+1}{k} \right\rfloor \right) = 2k + 1$. 

    \item It is enough to show that $2k + 2 \leq t\left(1, 2\left(\left\lfloor \frac{1}{2} \binom{2k+1}{k} \right\rfloor + 1\right)\right) \leq t\left(1, 2 \binom{2k}{k}\right) = 2k + 2$. First, $\left\lfloor \frac{1}{2} \binom{2k+1}{k} \right\rfloor + 1 = \frac{\binom{2k + 1}{k} - 1}{2} + 1$ if $\binom{2k + 1}{k}$ is odd and $\left\lfloor \frac{1}{2} \binom{2k+1}{k} \right\rfloor + 1 = \frac{1}{2} \binom{2k + 1}{k} + 1$ if $\binom{2k + 1}{k}$ is even. Then, $t\left(1, 2\left\lfloor \frac{1}{2} \binom{2k+1}{k} \right\rfloor + 2 \right) = t\left(1, \binom{2k + 1}{k} + 1\right)$ if $\binom{2k + 1}{k}$ is odd and $t\left(1, 2\left\lfloor \frac{1}{2} \binom{2k+1}{k} \right\rfloor + 2 \right) = t\left(1, \binom{2k + 1}{k} + 2 \right)$ if $\binom{2k + 1}{k}$ is even. In both cases, it follows by Equation~\ref{max-sperner-family} that $t\left(1, 2 \left\lfloor \frac{1}{2} \binom{2k+1}{k} \right\rfloor + 2 \right) \geq 2k + 2$, establishing the first inequality we wanted to show. The final inequality $t\left(1, 2 \binom{2k}{k}\right) = 2k + 2$ holds, since $\binom{2k+2}{k+1} = 2\binom{2k+1}{k} = 2\binom{2k+1}{k+1} > 2\binom{2k}{k} > \frac{2k+1}{k+1}\binom{2k}{k} = \binom{2k+1}{k+1}$. 
\end{enumerate}
\end{proof}

Lemmas~\ref{sperner-odd-odd+1} and \ref{sperner-even} can be summarized as follows. 

\begin{prop}
\label{more-sperner-doubling}
  For $n \geq 2$, let $\nextcomb{n} = \min \{ \binom{x}{\floor{\frac{x}{2}}} : x \in \mathbb{N}, \binom{x}{\floor{\frac{x}{2}}} \geq n\}$. If $n \leq \lfloor \tfrac{1}{2} \nextcomb{(\nextcomb{n} + 1)} \rfloor$, then $t(1, 2n) = t(1, n) + 1$.
 Otherwise, $t(1, 2n) = t(1, n) + 2$.  
\end{prop}

\begin{proof}
 Let $x = t(1, n)$. Then, $n \in \left[\binom{x-1}{\floor{\frac{x-1}{2}}} + 1, \binom{x}{\floor{\frac{x}{2}}} \right],$ and $\nextcomb{n} = \binom{x}{\floor{\frac{x}{2}}}.$ This implies $\nextcomb{n} + 1 = \binom{x}{\floor{\frac{x}{2}}} + 1$, so $\nextcomb{\overline{n} + 1}  = \binom{x+1}{\floor{\frac{x+1}{2}}}$. We can assume $n \geq 3$, since for $n = 2$, the result follows as $t(1, 4) = 4$ and $t(1, 2) = 2$. We consider two different cases depending on the parity of $x$.
 
  \noindent \underline{\textbf{Case 1:}} $x = 2k + 1$ for some $k \geq 1$. In this case,
$n \leq \binom{x}{\floor{\frac{x}{2}}} = \binom{2k+1}{k} = \left\lfloor \frac{1}{2} \frac{2k+2}{k+1} \binom{2k+1}{k} \right\rfloor = \left\lfloor \frac{1}{2} \binom{2k+2}{k+1} \right\rfloor = \left\lfloor \frac{1}{2} \binom{x+1}{\floor{\frac{x+1}{2}}} \right\rfloor = \lfloor \tfrac{1}{2} \nextcomb{(\nextcomb{n} + 1)} \rfloor$. In addition, since $t(1, n) = 2k+1$, by Lemma~\ref{sperner-odd-odd+1}, we obtain $t(1, 2n) = t(1, n) + 1$.

    \noindent \underline{\textbf{Case 2:}} $x = 2k$ for some $k \geq 2$. In this case, $\lfloor \tfrac{1}{2} \nextcomb{(\nextcomb{n} + 1)} \rfloor = \left\lfloor \frac{1}{2} \binom{x+1}{\floor{\frac{x+1}{2}}} \right\rfloor = \left\lfloor \frac{1}{2} \binom{2k+1}{\floor{\frac{2k+1}{2}}} \right\rfloor = \left\lfloor \frac{1}{2} \binom{2k+1}{k} \right\rfloor$. The result then follows by Lemma~\ref{sperner-even}.
\end{proof}

\subsection{Graph Theory Terminology} A \emph{graph} $G = (V(G), E(G))$ consists of a non-empty finite set $V(G)$ of elements called \emph{vertices} and a set $E(G)$ of subsets of $V(G)$, each having size $1$ or $2$, called \emph{edges}. An edge $e$ with $|e| = 1$ is called a \emph{loop} and an edge $e$ with $|e| = 2$ is called a \emph{proper edge}. An edge $e$ is \emph{incident} to a vertex $v$ if $v \in e$. If $G$ has no loops, then $G$ is a \emph{simple graph}. A graph $G$ is \emph{complete} if $G$ is simple and $|E(G)| = \binom{|V(G)|}{2}$. A complete graph with $n$ vertices is denoted by $K_{n}$.

A graph $G' = (V', E')$ is a \emph{subgraph}  of a graph $G = (V, E)$ if  $V' \subseteq V,$ and $E' \subseteq E \cap \mathcal{P}(V')$. We write $G' \subseteq G$. The subgraph of $G$ \emph{induced}  by $V'$, denoted by $G[V']$, is the graph $(V', E')$ such that $E' = E \cap \mathcal{P}(V')$. For some $I \subset V$, we denote by $G-I$ the subgraph induced by $V(G) \setminus I$, that is, $G - I = G[V \setminus I]$. A graph is \emph{bipartite} if its vertex set  $V$ can be partitioned into two non-empty subsets $U$ and $W$ such that every edge has a vertex in $U$ and a vertex in $W$, and it is a \emph{complete bipartite} graph if, in addition, for every vertex $x$ in $U$ and $y$ in $W$, there exists an edge $e$ such that $e = \{x, y\}$. We denote by $K_{m, n}$ the complete bipartite graph with $|U| = m$ and $|W| = n$.

A graph with $n$ vertices is a \emph{path}, denoted by $P_n$, if its vertices can be ordered $(v_1, v_2, \ldots, v_n)$ where $E(P_n) = \{ \{v_i, v_{i+1}\} : i \in [1, n-1]\}.$ A graph with $n$ vertices, denoted by $C_n$, is a \emph{cycle} if its vertices can be ordered $(v_1, v_2, \ldots, v_n)$ where $E(C_n) = \{ \{v_i, v_{i+1}\} : i \in [1, n-1]\} \cup \{\{v_1, v_n\}\}$. A graph $G$ is \emph{Hamiltonian} if it has a cycle subgraph $C_n$ such that $V(G) = V(C_n)$. A graph is a \emph{tree} if between any pair of distinct vertices there exists a unique path. 

The \emph{degree} of a vertex $v$ in a graph $G$, denoted by $\deg(v)$, is the number of edges incident to $v$. The \emph{minimum degree} of $G$, denoted by $\delta(G)$, is the smallest degree of a vertex in $G$, and the \emph{maximum degree}, denoted by $\Delta(G)$, is the largest degree of a vertex in $G$. A vertex $v \in V(G)$ is \emph{isolated} if $deg(v) = 0$, otherwise it is \emph{non-isolated}. 

 A \emph{$k$-vertex coloring} is a function $f:V(G) \rightarrow \mathbb{Z}_k$ such that for every edge $\{a,b\}$, we have $f(a) \neq f(b)$. The \textit{chromatic number} of a graph $G$, denoted by $\chi(G)$, is the smallest integer $k$ such that a $k$-vertex coloring of $G$ exists. A \emph{clique} in a graph $G$ is a subset of vertices  $C \subseteq V(G)$  such that $G[C] = K_{|C|}$; a \emph{maximum clique} is a largest such clique and its size is called the \emph{clique number}, denoted by $\omega(G)$.

 A \emph{graph homomorphism} from $G$ to $H$ is a mapping $\phi: V(G) \to V(H)$ such that for every edge $\{u, v\} \in E(G)$, we have $\{ \phi(u), \phi(v) \} \in E(H)$. If there exists a homomorphism from $G$ to $H$, we write it as $G \rightarrow H$. If $\phi$ is a bijective homomorphism with $|E(G)| = |E(H)|$, then $\phi$ is called a \textit{graph isomorphism} and in this case, we write $G \cong H$. It is well-known that $K_{\omega(G)} \to G \to K_{\chi(G)}$ \cite{hell2026graphs}. 

 A graph $G$ is \emph{maximal} with respect to some property $X$ if $G$ satisfies $X$ and for every graph $H$ satisfying $X$ with $V(H) = V(G)$ and $G \subseteq H$, we have $G = H$.

   For graphs $G_1$ and $G_2$, the Cartesian product of $G_1$ and $G_2$ (denoted by $G_1 \; \square \; G_2$) is a graph with $V(G_1 \; \square \; G_2) = V(G_1) \times V(G_2)$ and
\[
\begin{aligned}
E(G_1 \; \square \; G_2)
  &= \Bigl\{
       \{(a,b_1),(a,b_2)\} :
       a \in V(G_1),\ 
       \{b_1,b_2\} \in E(G_2)
     \Bigr\} \\
  &\quad\cup
     \Bigl\{
       \{(a_1,b),(a_2,b)\} :
       \{a_1,a_2\} \in E(G_1),\ 
       b \in V(G_2)
     \Bigr\}.
\end{aligned}
\]

It is well-known that if $G_1$ and $G_2$ are Hamiltonian, then $G_1 \; \square \; G_2$ is also Hamiltonian \cite{imrich2008topics}. Similarly, for graphs $G_1, G_2, \ldots, G_k$, the graph $G_1 \; \square \; G_2 \; \square \; G_3 \; \square \; \cdots \; \square \; G_k$ is defined by $$V(G_1 \; \square \; G_2  \; \square \; G_3 \; \square \; \cdots \; \square \; G_k) = V(G_1) \times V(G_2) \times V(G_3) \times \cdots \times V(G_k)$$ and $\{(x_1, x_2, \ldots, x_k), (y_1, y_2, \ldots, y_k) \} \in E(G_1 \; \square \; G_2 \; \square \; G_3 \; \square \; \cdots \; \square \; G_k)$ if and only if there exists $i \in [1, k]$ such that $x_i \neq y_i$ and $\{x_i, y_i\} \in E(G_i)$ and for $j \in [1, n] \setminus \{i\}$, $x_j = y_j$.  

\section{$G$-Sperner set system}
\label{G-sperner}
Definition~\ref{def:G-sperner} is the definition of $G$-Sperner set systems for a graph $G$. Definition~\ref{def:G-Sperner-matrix} is an equivalent one, using the incidence matrix of the set system.

\begin{definition}\textbf{(via set system)}
\label{def:G-sperner}
  Let $n, t$ be positive integers with $n \geq t$, and let $G = ([1, n], E(G))$ be a graph. Let $\mathcal{F} = ([1,t], \mathcal{B})$ be a set system with $\mathcal{B} = \{B_1, B_2, \ldots, B_n\}$ and $|\mathcal{B}| = n$ where $B_i$ corresponds to the vertex $i$. Then, $\mathcal{F}$ is \emph{Sperner for a proper edge $ \{a, b\} \in E(G)$} if \begin{center}
        $|B_a \setminus B_b| \geq 1 \text{ and } |B_b \setminus B_a| \geq 1$.
    \end{center}  
The set system $\mathcal{F}$ is \emph{$G$-Sperner}, or a $G$-Sperner$(t, n)$, if it is Sperner for every proper edge in $E(G)$.
\end{definition}

\begin{definition}\textbf{(via incidence matrix)}
\label{def:G-Sperner-matrix}
     Let $G = ([1, n], E(G))$ be a graph, and let $A$ be a $t \times n$ binary matrix where each column corresponds to a vertex of $G$. We denote by $A[i, j]$ the entry of $A$ in row $i$ and the column corresponding to the vertex $j \in [1, n]$, and by $A^{j}$ the column vector of $A$ corresponding to the vertex $j$. Then, Matrix $A$ is \emph{Sperner for a proper edge $\{a, b\}$} if there exist row indices $l, m \in [1, t]$ such that $A[l, a] = 0$ and $A[l, b] = 1$, and $A[m, a] = 1$ and $A[m, b] = 0$. Matrix $A$ is \emph{$G$-Sperner} if it is Sperner for every proper edge in $E(G)$.
\end{definition}

It is easy to verify the equivalence given in the following proposition.

\begin{prop} 
Let $G = ([1, n], E(G))$ be a graph and $\mathcal{F} = ([1, t], \mathcal{B})$ be a set system where $|\mathcal{B}| = n$. Then, $\mathcal{F}$ is a $G$-Sperner$(t, n)$ if and only if its incidence matrix is $G$-Sperner.
\end{prop}

We denote by $t_s(G)$ the minimum $t$ such that there exists a $G$-Sperner$(t, |V(G)|)$ set system, that is,
    $$t_s(G) = \min\{ t : \exists \text{ a } G\text{-Sperner}(t, |V(G)|) \}.$$

In this section, we show that $t_{s}(G)$ is determined from its chromatic number for any simple graph $G$. Our method is to characterize it using a special family of graphs, which we call a \textit{Sperner graph}, defined in this section. We first prove a series of fundamental results before proving the main result. 

For the rest of the paper, we assume that $G$ is a simple graph. Therefore, $G$ has only proper edges.

\begin{prop}
\label{t-in-homo}
    Let $G$ and $H$ be graphs such that $G \rightarrow H$. Then $t_s(G) \leq t_s(H)$.
\end{prop}

\begin{proof}
    Let $f:V(G) \rightarrow V(H)$ be a homomorphism from $G$ to $H$, and $A$ be a $t^{*} \times |V(H)|$ \sloppy$H$-Sperner matrix with $t^* = t_s(H)$. 
    Now, from $A$ we construct a $t^{*} \times |V(G)|$ matrix $M$ with every column corresponding to a vertex in $G$ such that for every $v \in V(G)$, we set the column $M^v$ identical to $A^{f(v)}$. Let $\{x, y\} \in E(G)$. Since homomorphisms preserve edges, $\{f(x), f(y)\} \in E(H)$. Since $A$ is Sperner for the edge $\{f(x), f(y)\}$, then $M$ is also Sperner for $\{x, y\}$. Thus, $M$ is a $G$-Sperner matrix with $t^*$ rows and this implies $t_s(G) \leq t^* = t_s(H)$.
\end{proof}

\begin{prop}
\label{t-in-completegraph}
    Let $n \geq 2$. Then, $t_s(K_n) = t(1, n)$.
\end{prop}

\begin{proof}
It is easy to see that a set system being $K_n$-Sperner is equivalent to being Sperner, which is equivalent to being a $1$-CFF.
\end{proof}

\begin{definition}
    The \textit{Sperner graph of order $z$}, denoted $\mathcal{S}(z)$ is a graph $(V, E)$ with
    \begin{itemize}
        \item vertex set $V = \{0,1\}^z = \mathbb{Z}_{2}^{z}$, and 
        \item edge set $E$ such that $\{x, y\} \in E$ if and only if there exists $j, k \in [1, z]$ such that $x_j = 0, y_j = 1$ and $x_k = 1, y_k = 0$.
    \end{itemize}
where $x_j$ represents the $j^{th}$ position of the binary $z$-tuple $x$. 
\end{definition}

From the definition above, it is also clear that the vertices of a Sperner graph of order $z$ correspond to the elements in the power set $\mathcal{P}([1,z])$ and there is an edge between $x, y \subseteq [1,z]$ if and only if $x \not\subseteq y$ and $y \not\subseteq x$. Figure~\ref{fig:S3} shows the Sperner graph $\mathcal{S}(3)$.    
\begin{figure}
    \centering

\includegraphics[width = 0.3\textwidth]{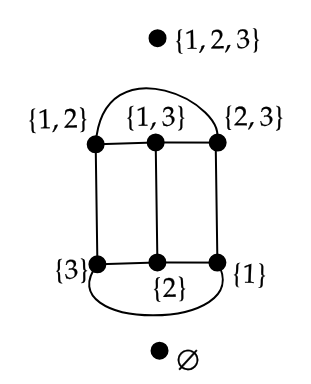}

    \caption{Sperner graph of order $3$}
    \label{fig:S3}
\end{figure}

A \emph{chain} in the partially ordered set of all subsets of $[1, n]$ ordered by inclusion, is a set of subsets in which for every pair $(A, B)$, either  $A \subseteq B$  or $B \subseteq A$. A \emph{chain decomposition} is a partition of $\mathcal{P}([1, n])$  into disjoint chains. The following theorem provides the chromatic number and the clique number of a Sperner graph of order $z$. The proof of this theorem is inspired by work on qualitative independent graphs by Meagher and Stevens~\cite{MS}. 

\begin{theorem}
\label{sperner-clique-coloring}
    $\omega(\mathcal{S}(z)) = \chi(\mathcal{S}(z)) = \binom{z}{\lfloor\frac{z}{2}\rfloor}$.
\end{theorem}

\begin{proof}
    By definition of $\mathcal{S}(z)$, only Sperner families on $[1,z]$ correspond precisely to the cliques in $\mathcal{S}(z)$. Hence, by Sperner's theorem, $\omega(\mathcal{S}(z)) = \binom{z}{\lfloor\frac{z}{2}\rfloor}$.

     Clearly, $\chi(\mathcal{S}(z)) \geq \omega(\mathcal{S}(z)) = \binom{z}{\lfloor\frac{z}{2}\rfloor}$. It is also well-known from Dilworth’s theorem \cite{dilworthAnnals, dilworth1987decomposition} that the poset of subsets of $[1,z]$ ordered by inclusion can be decomposed into $\binom{z}{\lfloor\frac{z}{2}\rfloor}$ disjoint chains such that each of the chains contains exactly one set of size $\lfloor \frac{z}{2} \rfloor$. Let $\mathcal{C}_1, \mathcal{C}_2, \ldots, \mathcal{C}_{\binom{z}{\lfloor\frac{z}{2}\rfloor}}$ be such chains. By the definition of Sperner's graph, it is clear that there cannot be an edge in $\mathcal{S}(z)$ between vertices corresponding to any two subsets in a chain. Thus, we can achieve a $\binom{z}{\lfloor\frac{z}{2}\rfloor}$-coloring of $\mathcal{S}(z)$ by assigning color $i$ to all vertices in chain $\mathcal{C}_i$, for $1 \leq i \leq \binom{z}{\lfloor\frac{z}{2}\rfloor}$. Thus, $\chi(\mathcal{S}(z)) \leq \binom{z}{\lfloor\frac{z}{2}\rfloor}.$ Hence, $\chi(\mathcal{S}(z)) = \binom{z}{\lfloor\frac{z}{2}\rfloor}$. 
\end{proof}

\begin{corollary}
\label{t-in-sperner-graph}
$t_s(\mathcal{S}(z)) = z$.    
\end{corollary}

\begin{proof}
    Since $K_{\omega(\mathcal{S}(z))} \rightarrow \mathcal{S}(z) \rightarrow K_{\chi(\mathcal{S}(z))}$, by Proposition~\ref{t-in-homo}, we have $t_s({K_{\omega(\mathcal{S}(z))}}) \leq t_s(\mathcal{S}(z)) \leq t_s({K_{\chi(\mathcal{S}(z))}})$. Then, by Theorem~\ref{sperner-clique-coloring}, Proposition~\ref{t-in-completegraph}, and Equation~\ref{max-sperner-family}, we have that $t_s(\mathcal{S}(z)) = t\left(1, \binom{z}{\lfloor\frac{z}{2}\rfloor}\right) = z$.
\end{proof}

In the next proposition, we give a characterization of $t_s (G)$ via a homomorphism from the graph $G$ to a Sperner graph.

\begin{prop}
\label{homo-in-sperner}
    $t_s(G) = \min_{l \in \mathbb{N}} \{l : G \rightarrow \mathcal{S}(l) \}$. 
\end{prop}

\begin{proof}
We claim that there exists an \( l \times |V(G)| \) \( G \)-Sperner matrix $A$ if and only if there exists a homomorphism \( f \) from \( G \) to \( \mathcal{S}(l) \). Indeed, if there exists such $A$, then $f : V(G) \to \mathbb{Z}_{2}^{l}$ with $f(v) = A^{v}$ is a homomorphism because for any edge $\{a,b\} \in E(G)$, there exist $j, k \in [1,l]$ such that $A[j, a] = 0$, $A[j, b] = 1$ and $A[k, a] = 1$, $A[k, b] = 0$, and thus $\{A^a, A^b\} \in E(\mathcal{S}(l))$. Now, for the converse, if there exists a homomorphism $f$ from $G$ to $\mathcal{S}(l)$, then for every edge $\{a, b\} \in E(G)$, we must have $\{f(a), f(b)\} \in E(\mathcal{S}(l))$. Let $A$ be an $l \times |V(G)|$ matrix where $A^{v} = f(v)$. Then, if $\{f(a), f(b)\} \in E(\mathcal{S}(l))$, there exist $j, k \in [1, l]$ such that $(f(a))_{j} = 0$, $(f(b))_{j} = 1$, $(f(a))_{k} = 1$, and $(f(b))_{k} = 0$. This also means that in rows $j$ and $k$ of $A$, we have $A[j, a] = 0$, $A[j, b] = 1$, $A[k, a] = 1$, and $A[k, b] = 0$. This claim implies $t_s(G) = \min_{l \in \mathbb{N}} \left\{ l : \exists \text{ a } G\text{-Sperner } (l, |V(G)|) \right\} 
= \min_{l \in \mathbb{N}} \left\{ l : G \to \mathcal{S}(l) \right\}$. 
\end{proof}

We have now provided all the foundational results needed to prove the main result of this section.

\begin{corollary}
\label{optimal-H-in}
    Let $G = (V, E)$ be a graph without any isolated vertex. Then, $t_s(G) = t(1, \chi(G))$. 
\end{corollary}

\begin{proof}
   Since $G \rightarrow K_{\chi(G)}
   $, from Propositions~\ref{t-in-homo} and \ref{t-in-completegraph}, we obtain $t_s(G) \leq t_s(K_{\chi(G)})$ and $t_s(K_{\chi(G)}) =  t(1, \chi(G))$. Thus, $t_s(G) \leq t(1, \chi(G))$.
  \par  Assume $t(1, \chi(G)) = k$ for some $k \in \mathbb{N}$. Suppose for the sake of contradiction that $t_s(G) < t(1, \chi(G))$, that is, $ t_s(G) \leq k-1$. Then, by Proposition~\ref{homo-in-sperner}, $G \rightarrow \mathcal{S}(k-1)$. By Theorem~\ref{sperner-clique-coloring}, $\chi(\mathcal{S}(k-1)) = \binom{k-1}{\lfloor \frac{k-1}{2} \rfloor}$. Since we know that $t(1, \chi(G)) = \min\{l : \binom{l}{\lfloor \frac{l}{2} \rfloor} \geq \chi(G)\}$ and $\chi(G) \leq \chi(\mathcal{S}(k-1)) = \binom{k-1}{\lfloor \frac{k-1}{2} \rfloor}$, the minimum $l$ occurs for some $l \leq k-1$. Thus, $t(1, \chi(G)) \leq k-1 < t(1, \chi(G))$, which is a contradiction. So, $t_s(G) = t(1, \chi(G))$.
\end{proof}

\section{Cover-free families on Graphs}
\label{CFFonGraphs}

We begin this section by giving two equivalent definitions of cover-free families on graphs. Definition~\ref{def:3} uses the set system and is a specialization to graphs from the more general cover-free families on hypergraphs defined in \cite{thaismourastructureaware}. Definition~\ref{def:3-matrix} is an equivalent one using the incidence matrix of the set system. 

 \begin{definition} \textbf{(via set system)}\label{def:3}
   Let $n, t$ be positive integers with $n \geq t$, and let $G = ([1, n], E(G))$ be a graph. Let $\mathcal{F} = ([1,t], \mathcal{B})$ be a set system with $\mathcal{B} = \{B_1, B_2, \ldots, B_n\}$ and $|\mathcal{B}| = n$ where $B_i$ corresponds to vertex $i$. The set system $\mathcal{F}$ is \emph{cover-free for an edge $\{a,b\}  \in E(G)$}, if for any $i_0 \in [1, n] \setminus \{a,b\}$, we have 

 \begin{center}
    $|B_{i_0} \setminus (B_a \cup B_b)| \geq 1$.
 \end{center}
Then, 
    \begin{enumerate}
        
        \item $\mathcal{F}$ is a $G$-\textit{ECFF}$(t, n)$ if $\mathcal{F}$ is cover-free for all edges in $E(G)$.

    \item $\mathcal{F}$ is a $G$-\textit{CFF}$(t, n)$ if $\mathcal{F}$ is both $G$-ECFF$(t, n)$ and $G$-Sperner$(t, n)$.
  
 \end{enumerate}
 \end{definition}
 We now present equivalent definitions of the above definitions via incidence matrices. 

\begin{definition} \textbf{(via incidence matrix)}
\label{def:3-matrix}
   Let $G = ([1, n], E(G))$ be a graph, and let $A$ be a $t \times n$ binary matrix where each column corresponds to a vertex of $G$. We denote by $A[i, j]$ the entry of $A$ in row $i$ and the column corresponding to the vertex $j \in [1, n]$, and by $A^{j}$ the column vector of $A$ corresponding to the vertex $j$. $A$ is \emph{cover-free for an edge $\{a,b\}  \in E(G)$}, if for any $v_0 \in [1, n] \setminus \{a,b\}$, there exists a row index $l \in [1, t]$ such that $A[l, a] = 0$, $A[l, b] = 0$, and $A[l, v_0] = 1$. Then,

    \begin{enumerate}
        \item Matrix $A$ is $G$-disjunct if it is cover-free for every edge in $E(G)$.

        \item Matrix $A$ is $\overline{G}$-disjunct if it is both $G$-\textit{disjunct} and $G$-Sperner.
      
    \end{enumerate}

\end{definition}
We remark that the notation $\overline{G}$ used here should not be confused with the standard notation for the complement of $G$. It is easy to prove the equivalences given in the following proposition.

\begin{prop} Let $G = ([1, n], E(G))$ be a graph and $\mathcal{F}$ be a set system. Then, 
    \begin{enumerate}
        \item $\mathcal{F}$ is a $G$-ECFF$(t, n)$ if and only if its incidence matrix is $G$-disjunct.
    
        \item $\mathcal{F}$ is a $G$-CFF$(t, n)$ if and only if its incidence matrix is $\overline{G}$-disjunct.
    \end{enumerate}
\end{prop}

Since a $2$-disjunct matrix is also a $\overline{2}$-disjunct matrix, being a $2$-CFF$(t, n)$ is equivalent to being a $K_n$-ECFF$(t, n)$ or being a $K_n$-CFF$(t, n)$. By Definition~\ref{def:3}, a $G$-CFF (or a $\overline{G}$-disjunct matrix) is a $G$-ECFF (or a $G$-disjunct matrix). However, in general, being a $G$-ECFF (or a $G$-disjunct matrix) is not equivalent to being a $G$-CFF (or a $\overline{G}$-disjunct matrix). Figure~\ref{GDisjunct-not-1CFF} shows a $G$-disjunct matrix that is not $\overline{G}$-disjunct, since the matrix is not Sperner for edges $\{1, 2\}$ and $\{3, 4\}$. 

 \begin{figure}
     \centering
     \includegraphics[width=0.25\linewidth]{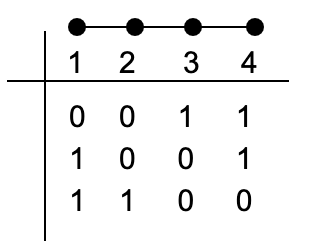}
     \caption{A $P_4$-disjunct matrix}
     \label{GDisjunct-not-1CFF}
 \end{figure}

 We denote by $t_e(G)$ and $t(G)$ the minimum $t$ such that there exist a $G$-ECFF$(t, n)$ and a $G$-CFF$(t, n)$; respectively:
     \begin{align*}
    t_{e}(G)  &= \min\{ t : \exists \text{ a } G\text{-ECFF}(t, |V(G)|) \}, \\
    t(G)      &= \min\{ t : \exists \text{ a } G\text{-CFF}(t, |V(G)|) \}.
\end{align*}

The following proposition is obtained directly from Definition~\ref{def:3}, by observing that we can vertically concatenate a $G$-disjunct matrix and a $G$-Sperner matrix to obtain a $\overline{G}$-disjunct matrix. 

\begin{prop} \cite{thaismourastructureaware}
\label{t(G)-simple-bound}
For a simple graph $G$, $t_e(G) \leq t(G) \leq t_{e}(G) + t_s(G)$.
\end{prop}

For the graph $L_n$ with $n$ loops, we have $t(L_n) = t_e(L_n) = t(1, n)$. If $G$ is a graph with a set of isolated vertices $I$ and at least $3$ non-isolated vertices, then $t_e(G - I) = t_e(G)$ and $t(G - I) = t(G)$; indeed, since isolated vertices are not part of any edges in the graph, to obtain a $G$-disjunct matrix from a $(G-I)$-disjunct matrix or a $\overline{G}$-disjunct matrix from a $\overline{(G - I)}$-disjunct matrix, it suffices to add columns of all $1$'s  as the columns corresponding to the isolated vertices. Therefore, throughout the rest of the paper, we assume that the graph has no isolated vertices.

Idalino and Moura~\cite{thaismourastructureaware} construct CFFs on hypergraphs using strong vertex coloring, which can be adapted to graphs as a construction using vertex coloring, given in Construction~\ref{explicit-construct1-color-1}, used to establish Theorem~\ref{bound-strongvertexcoloring}.

\begin{theorem} \cite{thaismourastructureaware}
\label{bound-strongvertexcoloring}
    Let $G = (V,E)$ be a graph and $\mathcal{N}_1, \mathcal{N}_2, \ldots, \mathcal{N}_k$ be the set of vertices in each color class of a $k$-vertex-coloring  of $G$. For each $i$, let $n_i = |\mathcal{N}_i|$ and $\sum_{i = 1}^{k} n_i = |V| = n$. Let $t_i = 1$ if $n_i = 1$ and $t_i = t(1, n_i)$, if $n_i \geq 2$, for $1 \leq i \leq k$. Then, there exists a  $G$-CFF$(t,n)$ with $t = \sum_{i = 1}^{k} t_i$. 
\end{theorem}

\begin{Construction}
 \label{explicit-construct1-color-1} 
Let $t_1, t_2, \ldots, t_k, n_1, n_2, \ldots, n_k$ be as in Theorem~\ref{bound-strongvertexcoloring}. Construct a matrix $A$ consisting of blocks $B_{ij}, 1 \leq i, j \leq k$, where block column $j$ corresponds to vertices having color $j$. Each $B_{ij}$ is a $t_i \times n_j$ matrix given as follows:
     \begin{center}
    $B_{i,j} = \left\{ \begin{array}{ll}  1\text{-CFF}(t_i, n_i), & \text{if } i = j \text{ and } n_i > 1, \\
   \mathbf{1}_{1 \times 1}, & \text{if } i = j \text{ and } n_i = 1, \\
      \mathbf{0}_{t_i \times n_j}, & \text{otherwise}, \\ \end{array} \right.$
\end{center}
as shown in Figure~\ref{constructcolor}.
\end{Construction}

\begin{figure}[h!]
  \centering
  \[
  \begin{bNiceMatrix}[first-col,first-row]
   & \textcolor{red}{\mathcal{N}_1} & \textcolor{green}{\mathcal{N}_2} & \textcolor{blue}{\mathcal{N}_3} & \cdots & \textcolor{violet}{\mathcal{N}_k} \\[10pt]
   & \fbox{\rule{0pt}{0.4cm}\makebox[2.3cm]{$1$-CFF$(t_1, n_1)$}} & \fbox{\rule{0pt}{0.4cm}\makebox[2.3cm]{$\mathbf{0}_{t_1, n_2}$}} & \fbox{\rule{0pt}{0.4cm}\makebox[2.3cm]{$\mathbf{0}_{t_1, n_3}$}} & \cdots & \fbox{\rule{0pt}{0.4cm}\makebox[2.3cm]{$\mathbf{0}_{t_1, n_k}$}} \\[10pt]
   & \fbox{\rule{0pt}{0.4cm}\makebox[2.3cm]{$\mathbf{0}_{t_2, n_1}$}} & \fbox{\rule{0pt}{0.4cm}\makebox[2.3cm]{$1$-CFF$(t_2, n_2)$}} & \fbox{\rule{0pt}{0.4cm}\makebox[2.3cm]{$\mathbf{0}_{t_2, n_3}$}} & \cdots & \fbox{\rule{0pt}{0.4cm}\makebox[2.3cm]{$\mathbf{0}_{t_2, n_k}$}} \\[10pt]
   & \fbox{\rule{0pt}{0.4cm}\makebox[2.3cm]{$\mathbf{0}_{t_3, n_1}$}} & \fbox{\rule{0pt}{0.4cm}\makebox[2.3cm]{$\mathbf{0}_{t_3, n_2}$}} & \fbox{\rule{0pt}{0.4cm}\makebox[2.3cm]{$1$-CFF$(t_3, n_3)$}} & \cdots & \fbox{\rule{0pt}{0.4cm}\makebox[2.3cm]{$\mathbf{0}_{t_3, n_k}$}} \\[10pt]
   & \vdots & \vdots & \vdots & \ddots & \vdots \\[5.83pt]
   & \fbox{\rule{0pt}{0.4cm}\makebox[2.3cm]{$\mathbf{0}_{t_k, n_1}$}} & \fbox{\rule{0pt}{0.4cm}\makebox[2.3cm]{$\mathbf{0}_{t_k, n_2}$}} & \fbox{\rule{0pt}{0.4cm}\makebox[2.3cm]{$\mathbf{0}_{t_k, n_3}$}} & \cdots & \fbox{\rule{0pt}{0.4cm}\makebox[2.3cm]{$1$-CFF$(t_k, n_k)$}}
  \end{bNiceMatrix}
  \]
  \caption{A $k$-vertex coloring construction of $G$-CFF}
  \label{constructcolor}
\end{figure}

The following corollary is a direct consequence of Theorem~\ref{bound-strongvertexcoloring}.

\begin{corollary}
\label{vertexcoloring-examples}
Let $P_n$, $C_n$, $T_n$ be respectively, a path, a cycle, a tree with $n \geq 3$ vertices, and let $K_{n_1,n_2}$ be a complete bipartite graph with $n_1, n_2 \geq 2$. Then,
\begin{enumerate}
    \item \[t(P_n) \leq 
\begin{cases} 
3 & \text{if } n = 3 \\
2t(1, \frac{n}{2}) & \text{if } n \text{ is even}, n \geq 4, \\
t(1, \left\lfloor \frac{n}{2} \right\rfloor) + t(1, \left\lceil \frac{n}{2} \right\rceil) & \text{if } n \text{ is odd}, n \geq 5 \\
\end{cases}\]
\item \[t(C_n) \leq 
\begin{cases} 
3 & \text{if } n = 3, \\
2t(1, \frac{n}{2}) & \text{if } n \text{ is even}, n \geq 4, \\
2t(1, \left\lfloor \frac{n}{2} \right\rfloor) + 1 & \text{if } n \text{ is odd}, n \geq 5.
\end{cases}\]
\item \[t(K_{n_{1},n_{2}}) \leq t(1, n_1) + t(1, n_2).\]
\item $t(T_n) \leq t(1, n_1) + t(1, n_2)$ where $n_1$ and $n_2$ are the number of vertices in each of the two color classes for a $2$-coloring of $T_n$. 
\end{enumerate}
\end{corollary}

\begin{proof}
The result follows from Theorem~\ref{bound-strongvertexcoloring} and the fact that $P_n, T_n$ are $2$-colorable and $C_n$ is $2$-colorable for $n$ even and $3$-colorable for $n$ odd. The size $n_i$ of each color class for a $2$-coloring is computed in the case of each type of graph:
  \begin{enumerate}
      \item For $P_n, n_1 = \floor{\frac{n}{2}}$ and $n_2 = \ceil{\frac{n}{2}}$, for $n > 3$. For $P_3$, see Proposition~\ref{small-cycles-234}.

      \item For $C_n, n_1 =\frac{n}{2}$ and $n_2 = \frac{n}{2}$ if $n$ is even. If $n \geq 5$ is odd, construct a $P_n$-disjunct matrix by Theorem~\ref{bound-strongvertexcoloring} and add a row with $0$'s in the columns corresponding to its endpoints and $1$'s for the remaining vertices. See Proposition~\ref{small-cycles-234}, for $C_3$.
      \item $K_{n_1, n_2}$ is $2$-colorable with color classes of size $n_1$ and $n_2$.

      \item The tree $T_n$ is $2$-colorable with color classes of size $n_1$ and $n_2$, as given. 

  \end{enumerate}
\end{proof}

For all cases of graph $G_n$ with $n$ vertices given in Corollary~\ref{vertexcoloring-examples} above, $t(G_n) \leq 2\log_{2} n + 1$ as $n$ grows. However, in Section~\ref{cff-on-paths-cycles} we improve this bound for $t(P_n)$ and $t(C_n)$ to $\frac{3}{\log_{2} 3} \log_{2} n + \O(1) \approx 1.893 \log_{2}(n) + \O(1)$.

\begin{prop}
\label{subgraph-CFF}
If $G$ is a subgraph of $H$, then $t(G) \leq t(H)$ and $t_e(G) \leq t_e(H)$.
\end{prop}

\begin{proof}
Let $A$ be the incidence matrix of an $H$-CFF$(t, |V(H)|)$. Let $M$ be the $t \times |V(G)|$ matrix where each column corresponding to a vertex $v_i \in V(G)$ is identical to the column corresponding to $v_i$ in $A$. It is easy to see that $M$ is $\overline{H[V(G)]}$-disjunct since deletion of vertices from $H$ to obtain $H[V(G)]$ does not affect the property that $M$ is cover-free and Sperner for all edges in $H[V(G)]$. Similarly, removing edges from $H[V(G)]$ to obtain the graph $G$ does not affect the property that $M$ is cover-free and Sperner for all edges in $G$. Thus, $M$ is $\overline{G}$-disjunct and $t(G) \leq t(H[V(G)]) \leq t(H)$. A similar argument yields $t_e(G) \leq t_e(H[V(G)]) \leq t_e(H)$.
\end{proof}

\begin{theorem}
\label{GbarCFF-1CFF}
     Let $G = (V(G), E(G))$ be a graph with no isolated vertices. Then any $G$-CFF$(t, |V(G)|)$ is a $1$-CFF$(t, |V(G)|)$. 
    \end{theorem}

\begin{proof}
   Let $n = |V(G)|$ and $A$ be a $\overline{G}$-disjunct matrix. We want to show that $A$ is $1$-disjunct. Let $u, v \in V(G)$. First, consider the case where $\{u,v\} \in E(G)$. Since $A$ is a $G$-Sperner matrix, there exists a row $j_1$ such that $A[j_1, u] = 0$ and $A[j_1, v] = 1$, and a row $j_2$ such that $A[j_2, u] = 1$ and $A[j_2, v] = 0$. Second, consider the case where $\{u,v\} \notin E(G)$. Let $u', v' \in V(G)$ such that $\{u, u'\} \in E(G)$ and $\{v, v'\} \in E(G)$, which must exist since $G$ has no isolated vertex. Thus, since $A$ is a $\overline{G}$-disjunct matrix, then there must exist a row $i_1$ with $A[i_1, u] = A[i_1, u'] = 0$ and $A[i_1, v] = 1$ and a row $i_2$ with $A[i_2, v] = A[i_2, v'] = 0$ and $A[i_2, u] = 1$. 
   From both cases considered, we conclude that for any $u, v \in V(G)$, there exists a row in $A$ that has a $0$ and a $1$ in the columns corresponding to vertices $u$ and $v$, respectively, and a row that has a $1$ and a $0$ in the columns corresponding to vertices $u$ and $v$, respectively.
   This implies $A$ is a $1$-CFF. 
\end{proof}

\begin{corollary}
\label{trivial-bound}
   Let $G = (V(G), E(G))$ be a graph on $n$ vertices, with no isolated vertices. Then,
   \begin{equation}
   \label{trivialbound-eq}
       t(1,n) \leq t(G) \leq t(2,n).
   \end{equation}

\end{corollary}

\begin{proof}
    Since $G$ is a subgraph of $K_n$, by Proposition~\ref{subgraph-CFF}, $t(G) \leq t(K_n) = t(2,n)$.  Since, by Theorem~\ref{GbarCFF-1CFF}, a $G$-CFF$(t, n)$ is a $1$-CFF$(t, n)$, we have $t(1, n) \leq t(G)$. 
\end{proof}

\begin{corollary}
\label{graph-with-vertices-central-binomial}
   Let $G$ be a graph with no isolated vertices and let $n = |V(G)|$ with $n = \binom{x}{\floor{\frac{x}{2}}}$, $x \geq 3$. Then,
   \begin{equation}
    \label{eq:t(G)>=x+1}
    t(G) \geq x + 1 = t(1, n) + 1.
   \end{equation}
\end{corollary}

\begin{proof}
  By Corollary~\ref{trivial-bound}, $t(G) \geq t(1, n)$. Then, by Equation~\ref{max-sperner-family}, $$t(1, n)  = t\!\left(1, \binom{x}{\left\lfloor \tfrac{x}{2} \right\rfloor}\right) = x.$$ Suppose there exists a $G$-CFF$\left(x, \binom{x}{\left\lfloor \tfrac{x}{2} \right\rfloor}\right)$. Since by Theorem~\ref{GbarCFF-1CFF}, a $G$-CFF is also a Sperner family with $\binom{x}{\left\lfloor \tfrac{x}{2} \right\rfloor}$ subsets; by Theorem~\ref{sperner}, this family must consist of every subset of size $\floor{\frac{x}{2}}$ or every subset of size $\ceil{\frac{x}{2}}$. Thus, the union of any two subsets in the family must contain at least one other subset in the family, a contradiction. Hence, it is impossible to construct a $G$-CFF$\left(x, \binom{x}{\left\lfloor \tfrac{x}{2} \right\rfloor}\right)$ on a ground set of size $x$, so $t(G) \geq x + 1 = t(1, n) + 1$.
\end{proof}

\subsection{Establishing that $\mathbf{t(G) \sim t_e(G)}$}
From Definition~\ref{def:3}, $t_e(G) \leq t(G)$. 
In this section, we show that, in many cases, $t(G) = t_e(G)$ and in all cases, their difference is at most $2$. First, we provide a necessary and sufficient condition for a graph $G$ such that any $G$-ECFF must be a $G$-CFF.

\begin{prop}
\label{GbarCFF=GCFF-nec-suf}
    For a graph $G = ([1, n], E(G))$, any $G$-disjunct matrix is $\overline{G}$-disjunct if and only if $\delta(G) \geq 2$.
\end{prop}

\begin{proof}
     First, suppose $\delta(G) \geq 2$. We need to show that every $G$-disjunct matrix is also $G$-Sperner. Let $A$ be a $G$-disjunct matrix.
    For an edge $\{u, v\}$,  let $x$ and $y$ be neighbors of $u$ and $v$, respectively, such that $x, y \notin \{u,v\}$. Then, for the edge $\{u, x\}$ and the vertex $v$, there must exist a row $i$ in $A$ such that $A[i, u] = 0, A[i, x] = 0,$ and $A[i, v] = 1$. Similarly, for the edge $\{v, y\}$ and the vertex $u$, there must exist a row $j$ in $A$ such that $A[j, v] = 0, A[j, y] = 0,$ and $A[j, u] = 1$. Thus, we have $A[i, u] = 0, A[i, v] = 1, A[j, v] = 0$, and $A[j, u] = 1$, which proves that $A$ is Sperner for the edge $\{u, v\}$. Thus, $A$ is $G$-Sperner, which implies $A$ is $\overline{G}$-disjunct. 

    Now, we show by contraposition that if $\delta(G) < 2$, there exists a $G$-disjunct matrix that is not $\overline{G}$-disjunct. Let $\{u, v\}$ be an edge with $\deg(u) = 1$. Let $A$ be a $t \times (n-1)$ $\overline{(G - \{u\})}$-disjunct matrix. Let $B$ be a $t \times n$ matrix constructed by adding a column to $A$, corresponding to $u$, which is equal to $A^{v}$.
   It is easy to see that $B$ is $G$-disjunct since using that $A$ is $1$-disjunct, we conclude that $B$ is cover-free for $\{u, v\}$.
    Note that $B$ is not $\overline{G}$-disjunct because $B$ is not Sperner for $\{u, v\}$. Hence, $B$ is not $G$-Sperner which means $B$ is  not $\overline{G}$-disjunct. 
\end{proof}

\begin{theorem}
\label{GbarCFF=GCFF}
   Let $G$ be a graph.
\begin{enumerate}

    \item
    \label{delta=2}
    If $\delta(G) \geq 2$, then \begin{equation}
        t(G) = t_{e}(G).
    \end{equation} 
 
    \item \label{t_e(G)+1} If $\delta(G) = 1$ and $G$ has no connected component that is a single edge, then \begin{equation} \label{item2-deltaG}
        t_{e}(G) \leq t(G) \leq t_{e}(G) + 1.
    \end{equation} 

    \item
    \label{t_e(G)+2}
    If $\delta(G) = 1$, then \begin{equation} \label{item3-deltaG}
      t_{e}(G) \leq t(G) \leq t_{e}(G) + 2.  
    \end{equation} 
   
\end{enumerate}

\end{theorem}

\begin{proof}
      By Proposition~\ref{t(G)-simple-bound}, we already know that $t_{e}(G) \leq t(G)$. Let $A$ be a $t_e(G) \times |V(G)|$  $G$-disjunct matrix. In the proof of Proposition~\ref{GbarCFF=GCFF-nec-suf}, we showed that $A$ is Sperner for an edge $\{u, v\}$ if $\deg(u), \deg(v) \geq 2$. Now, we prove each of the statements: 
\begin{enumerate}
    \item By Proposition~\ref{GbarCFF=GCFF-nec-suf}, if $\delta(G) \geq2$, every $G$-ECFF is also a $G$-CFF. This implies $t(G) \leq t_{e}(G)$, which concludes this case. 

    \item \label{deg1-no-single-component} Let $A'$ be a $(t_e(G) + 1) \times |V(G)|$ matrix created by appending a new row to $A$ with $0$'s in the columns corresponding to vertices of degree $1$, and $1$'s in the remaining columns. 
    It is enough to show that $A'$ is Sperner for any edge with an endpoint of degree $1$.
    Let $\{u, v\}$ be an edge such that $\deg(v) = 1$, and let $w$ be another neighbor of $u$. Such $w$ exists because no connected component of $G$ is a single edge. Then, for the edge $\{u, w\}$ and the vertex $v$, there must exist a row $i$ in $A$ such that $A[i, u] = 0, A[i, w] = 0$, and $A[i, v] = 1$. The last row of $A'$ contains a $0$ in the column corresponding to $v$ and a $1$ in the column corresponding to $u$. Thus, $A'$ is Sperner for $\{u, v\}$. This yields the bound $t(G) \leq t_e(G) + 1$.

    \item The only case we need to consider is when $\delta(G) = 1$ and $G$ has at least one connected component that is a single edge. Let $A''$ be a $(t_e(G) + 2) \times |V(G)|$ matrix obtained by appending two new rows to $A$ for dealing with vertices of degree $1$ as follows. First, copy an identity matrix $I_2$ under the columns corresponding to the vertices of each single-edge component. 
    Then, for the vertices that are not in a single-edge component, add an extra row to the corresponding columns of $A$, in the same way as done for $A'$ in Case~\ref{deg1-no-single-component}, and then add an arbitrary row to complete $t_{e}(G) + 2$ rows. Thus, $A''$ is $G$-Sperner, and since $A$ is $G$-disjunct, then $A''$ is $\overline{G}$-disjunct.
 This yields the bound $t(G) \leq t_e(G) + 2$.
\end{enumerate}
\end{proof}

\begin{notation}
\label{G+x}
Let $G$ be a graph and let $x \notin V(G)$. Then, $G + \{x\}$ is a graph where $V(G + \{x\}) = V(G) \cup \{x\}$ and $E(G + \{x\}) = E(G) \cup \left\{\{v, x\} : v \in V(G) \right\}$. 
\end{notation}

In other words, $G + \{x\}$ adds a universal vertex $x$ to $G$. A \emph{wheel} graph $W_n$ is a graph constructed by joining a single universal vertex to all vertices of a cycle of length $n-1$. Thus, $W_n = C_{n-1} + \{n\}$ and $ K_{n + 1} = K_{n} + \{n + 1\}$.

\begin{theorem}
\label{star-cons-gen}
 Let $n \geq 3$. Let $G = ([1, n], E(G)$ be a graph with no isolated vertices. Then, 
    \begin{equation}
    \label{star-cons-gen-eq}
        \max\{t(1, n) + 1, t(G)\} \leq t(G + \{0\}) \leq t(G) + 1.
    \end{equation}
 
\end{theorem}

\begin{proof}

Let $A$ be a $t \times n$ $\overline{G}$-disjunct matrix with $t = t(G)$. Let $M$ be a $(t+1) \times (n+1)$ matrix of the form $\begin{pmatrix}
\mathbf{0}_{t \times 1} & A \\
\mathbf{1}_{1 \times 1} & \mathbf{0}_{1 \times n}
\end{pmatrix}$, with columns labeled, in order, by $0, 1, \ldots, n$; and rows labeled as elements in $[1, t+1]$. We first show that $M$ is $\overline{G + \{0\}}$-disjunct. By Theorem~\ref{GbarCFF-1CFF}, $A$ is $1$-disjunct. Thus, for any vertices $y, z \in [1, n]$, there exists a row $r \in [1, t]$ of $A$ such that $A[r, y] = 0$ and $A[r, z] = 1$. Thus, for $\{0, y\} \in E(G + \{0\})$, we have $M[r, 0] = 0, M[r, y] = 0,$ and $M[r, z] = 1$. So, $M$ is cover-free for $\{0 ,y\}$. Furthermore, since $A$ is also $G$-disjunct, for each edge $\{y, w\} \in E(G)$ and any vertex $v \in [1, n] \setminus \{y, w\}$, there exists a row $r \in [1, t]$ of $A$ such that $A[r, y] = 0, A[r, w] = 0,$ and $A[r, v] = 1$. So, $M[r, y] = 0, M[r, w] = 0,$ and $M[r, v] = 1$. Hence, $M$ is $(G + \{0\})$-disjunct. Thus, $t_{e}(G + \{0\}) \leq t(G) + 1$. Since $\delta(G) \geq 1$, $\delta(G + \{0\}) \geq 2$. Then, by Theorem~\ref{GbarCFF=GCFF} part~\ref{delta=2}, $t(G + \{0\}) = t_{e}(G + \{0\}) \leq t(G) + 1$ (Using Proposition~\ref{GbarCFF=GCFF-nec-suf}, we also obtain that $M$ is $\overline{(G + \{0\}}$-disjunct). 

Let $M$ be a $\overline{(G + \{0\})}$-disjunct matrix. Let $I$ be the set of rows in $M$ such that $M[i, 0] = 0 $ for all $i \in I$. Then, the submatrix $M'$ of $M$ restricted to the columns $M^1, M^2, \ldots, M^{n}$ and rows in $I$ must be a $|I| \times n$ $1$-disjunct matrix. Therefore, $|I| \geq t(1, n)$. Since there must be at least one row $r$ such that $A[r, 0] = 1$ and $A[r, y] = 0$ for $y \in [1, n]$, then $t(G + \{0\}) \geq t(1, n) + 1$. Furthermore, since $G \subseteq G + \{0\}$; by Proposition~\ref{subgraph-CFF}, $t(G) \leq t(G + \{0\})$. Thus, we have the desired lower bound.
\end{proof}

Theorem~\ref{star-cons-gen} generalizes the bound $t(2,n) \leq t(2,n-1) + 1$ in Equation~\ref{simple-additive-bound}, which can also be written as $t(K_n) \leq t(K_{n-1}) + 1$. In addition, if $c = t(G) - t(1, n)$, then from Theorem~\ref{star-cons-gen}, it follows that $t(1, n) + c \leq t(G + \{0\}) \leq t(G) + 1 = t(1, n) + c + 1$. For the case of $c = 1$, the next corollary provides an exact bound of $t(G + \{0\})$ in terms of $t(G)$.

\begin{corollary}
\label{star-gen-cons-special-case}
    Let $G = ([1, n], E)$ be a graph without isolated vertices. If $t(G) - t(1, n) = 1$, then $t(G + \{0\}) = t(G) + 1$. 
\end{corollary}

\begin{proof}
    Since $t(G) = t(1, n) + 1$, by Theorem~\ref{star-cons-gen} we have $t(1, n) + 1 \leq t(G + \{0\}) \leq t(1, n) + 2$. Suppose, for the sake of contradiction, that $M$ is a $(t(1, n) + 1) \times (n+1)$ $\overline{(G + \{0\})}$-disjunct matrix. Let $I$ be the set of rows in $M$ such that $M[i, 0] = 0 $ for all $i \in I$. Then, the submatrix $M'$ of $M$ restricted to the columns $M^1, M^2, \ldots, M^{n}$ and rows in $I$ must be a $|I| \times n$ $1$-disjunct matrix. Therefore, $|I| \geq t(1, n)$. Since $M$ has $t(1, n) + 1$ rows, we have $|I| = t(1, n)$ and the remaining row $r \notin I$ must be the only row such that $M[r, 0] = 1$ and $M[r, y] = 0$ for all $y \in [1, n]$. This structure also implies $M'$ must be $\overline{G}$-disjunct, but $t(G) = t(1, n) + 1$ and $M'$ has $t(1, n)$ rows, which leads to a contradiction. Thus, $t(G + \{0\}) = t(1, n) + 2 = t(G) + 1$.
\end{proof}

\section{Tightness of various lower and upper bounds}
\label{Graphs-meeting-extreme-bounds}
 In this section, we provide families of graphs $G$ for which $t(G)$ achieves some of the extreme bounds presented in Section~\ref{CFFonGraphs}. In particular, we show that for a star graph $S_n$, $t(S_n)$ meets the upper bound of Equation~\ref{item2-deltaG} in Theorem~\ref{GbarCFF=GCFF}, and exhibit a subfamily of star graphs where $t(S_n)$ meets the lower bound of Equation~\ref{trivialbound-eq} in Corollary~\ref{trivial-bound}. Furthermore, $S_n$ also achieves the upper bound of Equation~\ref{star-cons-gen-eq} in Theorem~\ref{star-cons-gen}. We also provide infinite families of graphs achieving the upper bound of Equation~\ref{item3-deltaG} in Theorem~\ref{GbarCFF=GCFF} and one of them meets the lower bound of Equation~\ref{trivialbound-eq} in Corollary~\ref{trivial-bound}.

\begin{definition}
\label{def:star}
    A star $S_n$ is the complete bipartite graph $K_{1, n-1}$. 
\end{definition}
Thus, a star graph $S_n$ is $G + \{x\}$ where $G$ is a graph with $n-1$ vertices and no edges.

\begin{Construction}
\label{star-construction}
Let $A$ be a $t \times (n-1)$ binary matrix. Let $M_1$ be a matrix with columns labeled as $ 0, 1, \ldots, n-1 $ such that $M_1^0 = \mathbf{0}_{t \times 1}$ and $M_1^i = A^i$ for all $i \in [1, n-1]$. Let $M_2$ be the matrix obtained by appending a row to $M_1$ such that $M_2[t+1, 0] = 1$ and $M_2[t+1, i] = 0$ for all $i \in [1, n-1]$.
\end{Construction}

\begin{theorem}
\label{star-edgecff-optimal}
    Let $n \geq 3$ and $G = S_n$ be a star graph on $n$ vertices. Then $t_{e}(S_n) = t(1, n-1)$.
\end{theorem}

\begin{proof}
    We give label $0$ to the universal vertex and labels $1, \ldots, n-1$ to the other vertices. We then create a $t \times n$ matrix $M_1$ by following Construction~\ref{star-construction} with $A$ being a $t \times (n-1)$ $1$-disjunct matrix, with $t = t(1, n-1)$. For every edge $e = \{0, i\}$ with $i \in [1, n-1]$ and any vertex $j \in [1, n-1] \setminus \{i\}$, there exists a row $l \in [1, t(1, n-1)]$ such that $M_{1}[l, 0] = 0, M_{1}[l, i] = 0$ and $M_{1}[l, j] = 1$ since $M_{1}^i$ and $M_{1}^j$ are the $i^{th}$ and $j^{th}$ columns of $A$. Thus, $M_1$ is an $S_n$-disjunct matrix, which implies $t_{e}(S_n) \leq t(1, n-1)$. 
    
    Now, let $M$ be a $t^* \times n$ $S_n$-disjunct matrix. Let $i, j \in [1, n-1], i \neq j$. Then since $\{0, i\}$ is an edge, there must be a row $l \in [1, t^*]$ such that $M[l, 0] = 0, M[l, i] = 0,$ and $M[l, j] = 1$. Similarly, since $\{0, j\}$ is an edge, there must be a row $k \in [1, t^*]$ such that $M[k, 0] = 0, M[k, j] = 0,$ and $M[k, i] = 1$. Therefore, the submatrix $[M^1 \; M^2 \; M^3 \; \cdots \; M^{n-1}]$ is $1$-disjunct, which implies $t^* \geq t(1, n-1)$. Taking $t^* = t_e(S_n)$ implies $t_e(S_n)= t(1, n-1)$.      
\end{proof}

We show that $t(S_n)$ meets the upper bound of Equation~\ref{item2-deltaG} in Theorem~\ref{GbarCFF=GCFF}.
    
\begin{corollary}
\label{star-optimum}
    Let $n \geq 3$ and $G = S_n$ be a star graph on $n$ vertices. Then $t(S_n) = t_{e}(S_n) + 1 = t(1, n-1) + 1$.
\end{corollary}

\begin{proof}
  We construct the matrix $M_2$ using Construction~\ref{star-construction}, where $A$ is a $t \times (n-1)$ $1$-disjunct matrix where $t = t(1, n-1)$. In the proof of Theorem~\ref{star-edgecff-optimal}, we obtain that $M_1$ is $S_n$-disjunct. It guarantees the existence of a row $l$ such that $M_2[l, 0] = 0$ and $M_2[l, i] = 1$ for any $i \in [1, n-1]$. For the last row $l'$, we have $M_2[l', 0] = 1$ and $M_2[l', i] = 0$ for all $i \in [1, n-1]$. This implies $M_2$ is $S_n$-Sperner. Thus, $M_2$ is an $\overline{S_n}$-disjunct matrix, which implies $t(S_n) \leq t(1, n-1) + 1$.

 Let $M$ be a $t^* \times n$ $\overline{S_n}$-disjunct matrix with $M^0$ having weight $x$. Let $M'$ be the submatrix of $M$ corresponding to the rows of $M$ that contain a zero in column $M^0$ and to the columns of vertices $1, 2, \ldots, n-1$. We claim $M'$ is $1$-disjunct. Indeed, for any $i, j \in [1, n-1], i \neq j$, since $\{0, i\}$ and $\{0, j\}$ are edges and $M$ is $\overline{S_n}$-disjunct, there must be row indices $k, l$ such that $M[k, 0] = 0, M[k, i] = 0, M[k, j] = 1$ and $M[l, 0] = 0, M[l, j] = 0, M[l, i] = 1$. Thus, $M'[k, i] = 0, M'[k, j] = 1$ and $M'[l, j] = 0, M'[l, i] = 1$. This implies that $M'$ is $1$-disjunct. So, $t^* - x \geq t(1, n-1)$. Also, $x \geq 1$ since there must exist a row that has a $0$ in the column $M^i, i \in [1, n-1]$ and a $1$ in the column $M^0$. Hence, $t^* = t^* - x + x \geq t(1, n-1) +1$. Thus, $t(S_n) = t(1, n-1) + 1$ and by Theorem~\ref{star-edgecff-optimal}, $t(1, n-1) = t_{e}(S_n)$.
\end{proof}

We provide an example of an optimal $\overline{S_n}$-disjunct matrix for $n = 9$ in Figure~\ref{fig:S9-CFF}. 

\begin{figure}[ht]
    \centering
    \begin{subfigure}[b]{0.3\textwidth}
        \centering
 \tikzset{every picture/.style={line width=1pt}}       

\begin{tikzpicture}[scale=1]

\tikzset{
    vertex/.style={circle, fill=black, minimum size=10pt, inner sep=0pt, text=white},
    centervertex/.style={circle, fill=black, minimum size=10pt, inner sep=0pt, text=white}
}

\node[centervertex] (v0) at (0,0) {0};

\node[vertex] (v1) at (-1.5,-1.5) {1};
\node[vertex] (v2) at (0,-2) {2};
\node[vertex] (v3) at (1.5,-1.5) {3};
\node[vertex] (v4) at (2,0) {4};
\node[vertex] (v5) at (1.5,1.5) {5};
\node[vertex] (v6) at (0,2) {6};
\node[vertex] (v7) at (-1.5,1.5) {7};
\node[vertex] (v8) at (-2,0) {8};

\foreach \i in {1,...,8}
    \draw (v0) -- (v\i);

\end{tikzpicture}
   
       \caption{Star graph $S_9$}
    \end{subfigure}
    \hfill
    \begin{subfigure}[b]{0.5\textwidth}
    \vspace{-30cm}
  \centering
\raisebox{1.9cm}{$
\begin{bNiceMatrix}[first-col,first-row]
  & \{0\} & V(S_9) \setminus \{0\} \\[3pt]
  & \fbox{\rule{0pt}{0.2cm}\makebox[2.3cm]{$\mathbf{0}_{5\times 1}$}}
  & \fbox{\rule{0pt}{0.2cm}\makebox[2.3cm]{$1$-CFF$(5,8)$}} \\[3.7pt]
  & \fbox{\rule{0pt}{0.2cm}\makebox[2.3cm]{$\mathbf{1}_{1 \times 1}$}}
  & \fbox{\rule{0pt}{0.2cm}\makebox[2.3cm]{$\mathbf{0}_{1\times 8}$}}
\end{bNiceMatrix}
$}
        \caption{A $6 \times 9$ $\overline{S_9}$-disjunct matrix}
    \end{subfigure}
    \caption{An optimal $\overline{S_9}$-disjunct matrix}
    \label{fig:S9-CFF}
\end{figure}

The next corollary provides an infinite family of star graphs $S_n$ such that $t(S_n) = t(1, n)$, showing that the lower bound in Equation~\ref{trivialbound-eq} is tight. 

\begin{corollary}
\label{star-meeting-trivial-lower-bound}
   If $n = \binom{x}{\lfloor{\frac{x}{2}}\rfloor} + 1$ for some $x \in \mathbb{N}$, then $t(S_n) = t(1,n)$. 
\end{corollary}

\begin{proof}
    If $n = \binom{x}{\lfloor{\frac{x}{2}}\rfloor} + 1$, by Equation~\ref{max-sperner-family}, $t(1, n-1) = t\!\left(1, \binom{x}{\left\lfloor \tfrac{x}{2} \right\rfloor}\right) = x$.  Then, by Corollary~\ref{star-optimum}, $t(S_n) = t(1, n-1) + 1 = x + 1 = t(1, n)$. 
\end{proof}

We show next that $S_n$ is an infinite family of graphs satisfying the upper bound of Equation~\ref{star-cons-gen-eq}.

\begin{corollary}
\label{Sn+x-optimum}
   For $n \geq 3$, $t(S_n + \{x\}) = t(S_n) + 1$. 
\end{corollary}

\begin{proof}
   By Theorem~\ref{star-cons-gen}, $t(1, n) + 1 \leq t(S_n + \{x\}) \leq t(S_n) + 1$. By Corollary~\ref{star-optimum}, $t(S_n) = t(1, n-1) + 1$. If $n = \binom{x}{\lfloor{\frac{x}{2}}\rfloor} + 1$ for some $x \in \mathbb{N}$, then by Corollary~\ref{star-meeting-trivial-lower-bound}, $t(S_n) = t(1, n)$ and it implies $t(S_n + \{x\}) = t(S_n) + 1$. For other values of $n$, we have $t(1, n) = t(1, n-1)$. It further implies $t(S_n) -t(1, n) = 1$. Then by Corollary~\ref{star-gen-cons-special-case}, $t(S_n + \{x\}) = t(S_n) + 1$.
\end{proof}

We now show that there exists an infinite family of graphs meeting the upper bound of Equation~\ref{item3-deltaG} in Theorem~\ref{GbarCFF=GCFF}, which also meets the lower bound of Equation~\ref{trivialbound-eq} in Corollary~\ref{trivial-bound}.
 
\begin{notation}
\label{En}
     Let $\mathcal{E}_n$, for $n$ even, be a graph with $n$ vertices and $\frac{n}{2}$ disjoint edges.
\end{notation}
 In \cite{IM1, thaismourastructureaware}, it is proven that for hypergraphs $H$ with disjoint edges, $t_{e}(H) = t(1, |E(H)|)$. Thus, we state the following proposition specialized to graph $\mathcal{E}_n$.

\begin{prop} \cite{thaismourastructureaware}
\label{matchinggraph-optimal-edge-CFF}
  For even $n \geq 4, t_e(\mathcal{E}_n) = t(1, \frac{n}{2})$.
\end{prop}

 \begin{prop}
\label{matching}
      If $m \leq \lfloor \tfrac{1}{2} \nextcomb{(\nextcomb{m} + 1)} \rfloor$, then $t_e(\mathcal{E}_{2m}) + 1 \leq t(\mathcal{E}_{2m}) \leq t_e(\mathcal{E}_{2m}) + 2$. Otherwise, $t(\mathcal{E}_{2m}) = t_e(\mathcal{E}_{2m}) + 2= t(1, 2m)$.
\end{prop}

\begin{proof}
    By Proposition~\ref{matchinggraph-optimal-edge-CFF}, $t_e(\mathcal{E}_{2m}) = t(1, m)$ and by Theorem~\ref{GbarCFF=GCFF} item~\ref{t_e(G)+2}, $t(\mathcal{E}_{2m}) \leq t_e(\mathcal{E}_{2m})) + 2 = t(1, m) + 2$. By Theorem~\ref{trivial-bound}, it follows that $t(1, 2m) \leq t(\mathcal{E}_{2m}) \leq t(1, m) + 2$. If $m \leq \lfloor \tfrac{1}{2} \nextcomb{(\nextcomb{m} + 1)} \rfloor$, then by Proposition~\ref{more-sperner-doubling}, $t(1, 2m) = t(1, m) + 1$ and it yields $t_e(\mathcal{E}_{2m}) + 1 \leq t(\mathcal{E}_{2m}) \leq t_e(\mathcal{E}_{2m}) + 2$. Otherwise, $t(1, 2m) = t(1, m) + 2$ and we obtain $t(\mathcal{E}_{2m}) = t(1, m) + 2 = t(1, 2m)$. 
\end{proof}

Notice that Proposition~\ref{matching} provides an infinite family of $\mathcal{E}_{n}$ which meets the upper bound of $t(G)$ given in Equation~\ref{item3-deltaG} of Theorem~\ref{GbarCFF=GCFF}, that is, if $n = 2m$ with $m > \lfloor \tfrac{1}{2} \nextcomb{(\nextcomb{m} + 1)} \rfloor$, then $t(\mathcal{E}_{2m}) = t_e(\mathcal{E}_{2m}) + 2$. The next theorem provides another infinite family of $\mathcal{E}_{n}$ which also meets the upper bound of $t(G)$ given in Equation~\ref{item3-deltaG} of Theorem~\ref{GbarCFF=GCFF}, as well as, the lower bound of $t(G)$ given in Equation~\ref{eq:t(G)>=x+1} of Corollary~\ref{graph-with-vertices-central-binomial}.

 \begin{theorem}
\label{matching-sperner}
    Let $x \geq 4$ be an integer such that $\binom{x}{\lfloor \frac{x}{2} \rfloor}$ is even. Then, $$t\!\left(\mathcal{E}_{\binom{x}{\left\lfloor \tfrac{x}{2} \right\rfloor}}\right) = x + 1 = t_e\!\left(\mathcal{E}_{\binom{x}{\left\lfloor \tfrac{x}{2} \right\rfloor}}\right) + 2.$$
\end{theorem}

\begin{proof}
By Proposition~\ref{matching}, $t\!\left(\mathcal{E}_{\binom{x}{\left\lfloor \tfrac{x}{2} \right\rfloor}}\right)
 \leq t_e\!\left(\mathcal{E}_{\binom{x}{\left\lfloor \tfrac{x}{2} \right\rfloor}}\right) + 2$ and by Proposition~\ref{matchinggraph-optimal-edge-CFF}, we have $t_e\!\left(\mathcal{E}_{\binom{x}{\left\lfloor \tfrac{x}{2} \right\rfloor}}\right)
 = t\!\left(1, \tfrac{1}{2} \binom{x}{\left\lfloor \tfrac{x}{2} \right\rfloor}\right)$. Since 
$\left| V\!\left(\mathcal{E}_{\binom{x}{\left\lfloor \tfrac{x}{2} \right\rfloor}}\right) \right|
= \binom{x}{\left\lfloor \tfrac{x}{2} \right\rfloor}$, by Corollary~\ref{graph-with-vertices-central-binomial}, $t\!\left(\mathcal{E}_{\binom{x}{\left\lfloor \tfrac{x}{2} \right\rfloor}}\right) \geq x + 1$.
Therefore, $x + 1 \leq t\!\left(\mathcal{E}_{\binom{x}{\left\lfloor \tfrac{x}{2} \right\rfloor}}\right) \leq t_e\!\left(\mathcal{E}_{\binom{x}{\left\lfloor \tfrac{x}{2} \right\rfloor}}\right) + 2 = t\!\left(1, \tfrac{1}{2} \binom{x}{\left\lfloor \tfrac{x}{2} \right\rfloor}\right) + 2$. To complete the proof, it is enough to show that $t\!\left(1, \tfrac{1}{2} \binom{x}{\left\lfloor \tfrac{x}{2} \right\rfloor}\right) = x-1$. To prove this, we consider two cases:
 
  \noindent \underline{\textbf{Case 1:}} $x = 2k + 1$ for some $k \geq 2$. Then, $t\!\left(1, \tfrac{1}{2} \binom{x}{\left\lfloor \tfrac{x}{2} \right\rfloor}\right) = t\!\left(1, \tfrac{1}{2} \binom {2k + 1}{k}\right)$ and by Proposition~\ref{sperner-technical}, we have $t\!\left(1, \tfrac{1}{2} \binom {2k + 1}{k}\right) = 2k$, which is equal to $x - 1$.

 \noindent \underline{\textbf{Case 2:}} $x = 2k$ for some $k \geq 2$. Then, $t\!\left(1, \tfrac{1}{2} \binom{x}{\left\lfloor \tfrac{x}{2} \right\rfloor}\right) = t\!\left(1, \tfrac{1}{2} \binom {2k}{k}\right) = t\!\left(1, \tfrac{2k}{2k} \binom {2k-1}{k-1}\right) = t\!\left(1, \binom {2k-1}{k-1}\right) = 2k - 1 = x - 1$. 

Thus, $t\!\left(\mathcal{E}_{\binom{x}{\left\lfloor \tfrac{x}{2} \right\rfloor}}\right) = x + 1 = t_{e}\!\left(\mathcal{E}_{\binom{x}{\left\lfloor \tfrac{x}{2} \right\rfloor}}\right) + 2$.
\end{proof}

\begin{prop}
\label{example-E8}
$t(\mathcal{E}_8) = 5 = t_{e}(\mathcal{E}_8) + 1$. 
\end{prop}

\begin{proof}
Since $4 \leq \lfloor \tfrac{1}{2} \nextcomb{(\nextcomb{4} + 1)} \rfloor = 5$, by Proposition~\ref{matching}, we have $5 \leq t(\mathcal{E}_8) \leq 6$. Let $V(\mathcal{E}_8) = \{v_i : i \in [1 ,8]\}$ and $$E(\mathcal{E}_8) = \left\{ \{v_1, v_2\}, \{v_3, v_4\}, \{v_5, v_6\}, \{v_7, v_8\} \right\}.$$
With the ground set $[1, 5]$, let $\mathcal{F} = \{B_{v_i} : i \in [1, 8]\}$ be a set system where 
\[
\begin{array}{llll}
B_{v_1} = \{1,2\} & B_{v_3} = \{1,4\} & B_{v_5} = \{2,4\} & B_{v_7} = \{3,4\} \\
B_{v_2} = \{1,3\} & B_{v_4} = \{1,5\} & B_{v_6} = \{2,5\} & B_{v_8} = \{3,5\}
\end{array}
\]
It is easy to verify that $\mathcal{F}$ is an $\mathcal{E}_8$-CFF implying $t(\mathcal{E}_8) = 5$. Also, by Proposition~\ref{matchinggraph-optimal-edge-CFF}, $t_e(\mathcal{E}_8) = t(1, 4) = 4$.
\end{proof}

Propositions~\ref{matching} and \ref{example-E8} lead us to the following open problems, which are worth investigating. 
\begin{op}
\label{op-E2m-existence}
    Is there an infinite family of $\mathcal{E}_{2m}$, where $t(\mathcal{E}_{2m}) = t_{e}(\mathcal{E}_{2m}) + 1$?
\end{op}

\begin{op}
\label{op-E2m-categorize}
   Determine the values of $m$ such that $t(\mathcal{E}_{2m}) = t_e(\mathcal{E}_{2m}) + 1$ and $t(\mathcal{E}_{2m}) = t_e(\mathcal{E}_{2m}) + 2$. 
\end{op}

Now, we show that there are infinite families of graphs $G$ with $\deg(G) = 1$ meeting the lower bounds provided in Equations~\ref{item2-deltaG} and \ref{item3-deltaG}.

\begin{notation}
\phantomsection
\label{not:graphK_n-operations}
\begin{itemize}
    \item A \emph{complete graph $K_n$ with a pendant vertex attached} is a graph $G$ obtained from $K_n$ by adding a new vertex $x$ and an edge joining $x$ to exactly one vertex of $K_n$. In other words, $V(G)=V(K_n) \mathbin{\dot\cup} \{x\}$ and $E(G) = E(K_n) \cup \{\{{x,v}\}\}$ for some $v\in V(K_n)$. We denote this graph by $K_n \cup K_2$.
    \item A \emph{complete graph $K_n$ with an edge component} is a graph $G$ obtained by adding two new vertices $x, y$ to $K_n$ and an edge $\{x, y\}$. In other words, $V(G)=V(K_n) \mathbin{\dot\cup} \{x, y\}$ and $E(G) = E(K_n) \cup \{\{{x,y}\}\}$.  We denote this graph by $K_n \mathbin{\dot\cup} K_2$.
\end{itemize}
\end{notation}

\begin{prop}
\label{trivial-t(G)=t_e(G)}
    \begin{enumerate}
        \item Let $J_{n+1}$ be a simple graph with $n+1$ vertices such that $K_n \subset J_{n+1}$. If $t(2, n+ 1) = t(2, n)$, then $t(J_{n+1}) = t_e(J_{n+1})$.
        \item Let $Y_{n+2}$ be a simple graph with $n+2$ vertices such that $K_n \subset Y_{n+2}$, If $t(2, n+ 2) = t(2, n)$, then  $t(Y_{n+2}) = t_e(Y_{n+2})$.
    \end{enumerate}
\end{prop}

\begin{proof}
    \begin{enumerate}
        \item Since $K_n \subseteq J_{n+1} \subseteq K_{n+1}$, by Propositions~\ref{t(G)-simple-bound} and \ref{subgraph-CFF}, we have $t(2, n) = t_e(K_n) \leq t_e(J_{n+1}) \leq t(J_{n+1}) \leq t(K_{n+1}) = t(2, n+1)$. Since $t(2, n) = t(2, n+1)$, we have $t(J_{n+1}) = t_e(J_{n+1})$.
        \item Since $K_n \subseteq Y_{n+2} \subseteq K_{n+2}$, by Propositions~\ref{t(G)-simple-bound} and \ref{subgraph-CFF}, we have $t(2, n) = t_e(K_n) \leq t_e(Y_{n+2}) \leq t(Y_{n+2}) \leq t(K_{n+2}) = t(2, n+2)$. Since $t(2, n) = t(2, n+2)$, we have $t(Y_{n+2}) = t_e(Y_{n+2})$.
    \end{enumerate}
\end{proof}

\begin{corollary}
\label{cor:trivial-t(G)=t_e(G)}
    \begin{enumerate}
        \item If $t(2, n+ 1) = t(2, n)$, then $t(K_n \cup K_2) = t_e(K_n \cup K_2)$.
        \item If $t(2, n+ 2) = t(2, n)$, then  $t(K_n \mathbin{\dot\cup} K_2) = t_e(K_n \mathbin{\dot\cup} K_2)$.
    \end{enumerate}
\end{corollary}

\begin{proof}
    We get the desired results by applying Proposition~\ref{trivial-t(G)=t_e(G)} with $J_{n+1} = K_n \cup K_2$ and $Y_{n+2} = K_n \mathbin{\dot\cup} K_2$.
\end{proof}

\begin{remark}
    Since $t(2, n) \leq 5.512\log_{2}(n)$ \cite{erdos1982families}, there must exist infinite families of $n_i$ and $n'_i$ with $t(2, n_i) = t(2, n_i + 1)$ and $t(2, n'_i) = t(2, n'_i + 2)$, respectively.
\end{remark}

\begin{example}

We provide an example of $K_{11} \cup K_2$ and a $K_{10} \mathbin{\dot\cup} K_2$ in Figure~\ref{fig:K12-subgraphs}. From Table~\ref{tab:t2n-small}, we have $t(2, 10) = t(2, 11) = t(2, 12) = 9$. By Proposition~\ref{trivial-t(G)=t_e(G)}, in both cases, we have $t(G) = t_e(G)$.
\end{example}

\begin{figure}[ht]
    \centering

    \begin{minipage}{0.48\textwidth}
        \centering
     \begin{tikzpicture}[
    scale=0.8,
    every node/.style={
        circle,
        fill=black,
        inner sep=1.2pt
    }
]

\foreach \i in {1,...,11}{
    \node (v\i) at ({90-(\i-1)*360/11}:2.4) {};
}

\foreach \i in {1,...,11}{
    \foreach \j in {\i,...,11}{
        \ifnum\i<\j
            \draw[thin] (v\i)--(v\j);
        \fi
    }
}

\node (leaf) at ($(v4)+(1.5,0)$) {};
\draw[thick] (v4)--(leaf);

\end{tikzpicture}
    \end{minipage}
    \hfill
    \begin{minipage}{0.48\textwidth}
        \centering
        \begin{tikzpicture}[
    scale=0.85,
    every node/.style={
        circle,
        fill=black,
        inner sep=1.2pt
    }
]
\foreach \i in {1,...,10}{
    \node (v\i) at ({90-(\i-1)*360/10}:2.2) {};
}

\foreach \i in {1,...,10}{
    \foreach \j in {\i,...,10}{
        \ifnum\i<\j
            \draw[thin] (v\i)--(v\j);
        \fi
    }
}
\node (a) at (3,0.4) {};
\node (b) at (4,0.4) {};
\draw[thick] (a)--(b);

\end{tikzpicture}

    \end{minipage}
    \caption{ A $K_{11} \cup K_2$ (left) and a $K_{10} \mathbin{\dot\cup} K_2$ (right)}
    \label{fig:K12-subgraphs}
\end{figure}

\section{Cover-free families on paths and cycles}
\label{cff-on-paths-cycles}

By Corollary~\ref{vertexcoloring-examples} and Proposition~\ref{subgraph-CFF}, $t(P_n) \leq t(C_n) \leq 2t(1, \lfloor n/2 \rfloor) + (n \bmod 2)$. Using a mixed-radix Gray code, we provide a construction that improves this upper bound to $1.893 \log_{2} n + \mathcal{O} (1)$. Section~\ref{Gray-intro} provides the necessary introduction to Gray codes and their generalizations. 

\subsection{Mixed-radix Gray code}
\label{Gray-intro}
The \emph{Hamming distance} between two $n$-tuples is the number of positions at which the corresponding co-ordinates differ. A \emph{binary Gray code} $G_n$ is an ordering of the $2^n$ binary $n$-tuples such that any two consecutive tuples have Hamming distance exactly one. Analogously, a mixed-radix Gray code $G_{n}^{(m_1, m_2, \ldots, m_n)}$ deals with ordering $n$-tuples $(a_1, a_2, \ldots, a_n) \in \mathbb{Z}_{m_1} \times \mathbb{Z}_{m_2} \times \cdots \times \mathbb{Z}_{m_n}$\footnote{We use $\mathbb{Z}_{m_1} \times \cdots \times \mathbb{Z}_{m_n}$ to denote the Cartesian product of the sets $\mathbb{Z}_{m_1}, \ldots, \mathbb{Z}_{m_n}$; not the direct product.} such that any two consecutive tuples have Hamming distance exactly one. Hence, a binary Gray code $G_n$ is $G_{n}^{(2, 2, \ldots, 2)}$. A Gray code is \emph{cyclic} if the Hamming distance between the first and last codewords is $1$. Knuth discusses \textit{mixed-radix reflected Gray codes} and \textit{mixed-radix modular Gray codes} in \cite{knuth2011art}, whose definitions and constructions are discussed next. Throughout the rest of this paper, the adjective ``mixed-radix" will be dropped from the mixed-radix Gray code, and we will simply refer to the Gray code $G_{n}^{(m_1, m_2, \ldots, m_n)}$ and simply specify when it is binary.

\begin{notation} For a Gray code $\mathcal{D} = [D_0, D_1,\ldots, D_{m-1}, D_m]$, the reflection of $\mathcal{D}$ is denoted $\mathcal{D}^R=[D_{m}, D_{m-1},\ldots, D_1, D_0]$.
For an $n$-tuple $D = (a_1, a_2, \ldots, a_n)$, $x \cdot D$ denotes
$(x, a_1, a_2, \ldots, a_n)$ and $x \cdot \mathcal{D}$ denotes $[x \cdot D_0, x \cdot D_1,\ldots, x \cdot D_{m-1}, x \cdot D_m]$. Similarly, $D \cdot x$ denotes
$(a_1, a_2, \ldots, a_n, x)$ and $\mathcal{D} \cdot x$ denotes $[D_0 \cdot x,  D_1 \cdot x, \ldots, D_{m-1} \cdot x, D_m \cdot x]$. 
\end{notation}

The \textit{reflected Gray code} $R_{n}^{(m_1, m_2, \ldots, m_n)}$ is defined recursively in Construction~\ref{mixed-radix-gray-reflected-code-recursive}.

\begin{Construction}(Reflected Gray code)
\label{mixed-radix-gray-reflected-code-recursive}
\begin{align*}
    R_{1}^{(m_n)} = \mathcal{R}_1 &:= [(0),(1),(2), \ldots, (m_{n}-1)]. \text{ For } i \geq 2, \\
     R_{i+1}^{(m_{n-i}, m_{n-i+1}, \ldots, m_{n-1}, m_n)} = \mathcal{R}_{i+1} &:= 
    \left\{
    \begin{array}{ll}
        [0 \cdot \mathcal{R}_{i}, 1 \cdot \mathcal{R}_{i}^{R}, 2 \cdot \mathcal{R}_{i}, \ldots, m_{n-i}-1 \cdot \mathcal{R}_{i}^{R}], & \text{if } m_{n-i} \text{ is even} \\[8pt]
        [0 \cdot \mathcal{R}_{i}, 1 \cdot \mathcal{R}_{i}^{R}, 2 \cdot \mathcal{R}_{i}, \ldots, m_{n-i}-1 \cdot \mathcal{R}_{i}], & \text{if } m_{n-i} \text{ is odd}
    \end{array} 
    \right.
\end{align*}
\end{Construction}

 The proof of the next Proposition can be found in Appendix~\ref{cyclic-even-m1-proof}.

\begin{prop}
\label{cyclic-even-m1}
     $R_{n}^{(m_1, m_2, \ldots, m_n)}$ is cyclic if and only if $m_1$ is even.
\end{prop}

 \emph{Modular Gray codes} $M_{n}^{(m_1, m_2, \ldots, m_n)}$ \cite{knuth2011art} are, however, different from reflected ones in that each successor is obtained from its predecessor by incrementing exactly one coordinate (say at the $i^{\text{th}}$ position) modulo \(m_i\). The term `modular' indicates that the ordering is generated by `cyclically wrapping around'. In general, transitions for reflected Gray code are $0 \rightarrow 1 \rightarrow 2 \rightarrow \cdots \rightarrow m_{i} - 2 \rightarrow m_{i} - 1 \rightarrow m_{i} - 2 \rightarrow \cdots \rightarrow 2 \rightarrow 1 \rightarrow 0$ while transitions for modular Gray code are $0 \rightarrow 1 \rightarrow 2 \rightarrow \cdots m_{i}-2 \rightarrow m_{i}-1 \rightarrow 0 \rightarrow 1 \rightarrow 2 \rightarrow \cdots$. 

Let $M_{n}^{q}$ be an ordered list of tuples of length $n$ where $m_1 = m_2 = \cdots = m_n = q$. For $n \geq 2$, we express a construction of $M_{n}^{q}$ recursively by referring to the individual codewords in $M_{n-1}^{q}$. 
\begin{notation}
    For $D \in \mathbb{Z}_{q}^{n-1}$ and $j \in \mathbb{Z}_{q}$, denote $D \cdot j^+ = [D \cdot (j \mod q), D \cdot ((j+1) \mod q), \ldots, D \cdot ((j+q-1) \mod q)]$.
\end{notation}

\begin{Construction} (Modular Gray code)
\label{loc-rec-fixedradix} 
Let $q>1$, $n\geq 1$. Define 
$M_{1}^{q} = [(0), (1), \ldots, (q-1)]$, and if $n\geq 2$ and $M_{n-1}^{q} = [D_0, D_1, \ldots, D_{q^{n-1}-1}]$; then
\begin{center}
    $ M_{n}^{q} = [D_{0} \cdot 0^{+}, D_{1} \cdot (q-1)^{+}, D_{2} \cdot (q-2)^{+}, \ldots, D_{i} \cdot (q-i)^{+}, \ldots, D_{q^{n-1}-1} \cdot 1^{+} ]$.
\end{center}
\end{Construction}

The proof of the next Proposition is provided in Appendix~\ref{loc-rec-fixed-proof-exp}. 

\begin{prop}
\label{loc-rec-fixed-proof}
    Let $q > 1$ and $n \geq 1$ be integers. The modular Gray code $M_{n}^{q}$ given in Construction~\ref{loc-rec-fixedradix} is cyclic.
\end{prop}

\subsection{Constructions of cover-free families on paths and cycles using Gray codes}

Let $n = \prod_{i = 1}^{k} m_i, t = \sum_{i = 1}^{k} m_i$, and $f(i, d) = \sum_{l = 1}^{i-1} m_l + d + 1$ where $m_i, d \in \mathbb{Z}^{+}$. Let $\mathcal{P}_i = \{f(i, 0), f(i, 1), \ldots, f(i, m_{i}-1)\}$. Thus, $[1, t] = \mathcal{P}_1 \; \mathbin{\dot\cup} \; \mathcal{P}_2 \; \mathbin{\dot\cup} \;  \cdots \; \mathbin{\dot\cup} \; \mathcal{P}_k$. We define a bijection $F$ between tuples $D = (d_1, d_2, \dots, d_k)$ in  $\mathbb{Z}_{m_1} \times \mathbb{Z}_{m_2} \times \cdots \times \mathbb{Z}_{m_k}$ and $k$-subsets of $[1, t]$ that are transversals of $\{\mathcal{P}_1, \mathcal{P}_2, \ldots, \mathcal{P}_k\}$:

\begin{equation}
\label{eq:gray}
F(D) = \{f(1, d_1), f(2, d_2), \ldots, f(k, d_k)\}.
\end{equation}

\begin{Construction}
\label{CFFonPaths-cons}

Let $G_{k}^{(m_1, m_2, \ldots, m_k)} = [D_0, D_1, D_2, \ldots, D_{n-1}]$ where $t = \sum_{i = 1}^{k} m_i, n = \prod_{i = 1}^{k} m_i$ and $\mathcal{B}\left(G_{k}^{(m_1, m_2, \ldots, m_k)}\right) = \{ F(D_0), F(D_1) ,\ldots, F(D_{n-1})\}$ with $B_i = F(D_i)$ for $i \in [0, n-1]$ and take $$\mathcal{F}\left(G_{k}^{(m_1, m_2, \ldots, m_k)}\right) = \left([1, t], \mathcal{B}\left(G_{k}^{(m_1, m_2, \ldots, m_k)}\right)\right).$$

\end{Construction}

Given two subsets $A$ and $B$, $A \triangle B$ denotes the symmetric difference of $A$ and $B$. We prove the following lemma.

\begin{lemma}
\label{hamming-distance-1}
    Let $F$ be the function defined in Equation~\ref{eq:gray}. Let $D_{z_1}, D_{z_2},$ and $D_{z_3}$ be distinct tuples in $\mathbb{Z}_{m_1} \times \mathbb{Z}_{m_2} \times \cdots \times \mathbb{Z}_{m_k}$. If the Hamming distance between $D_{z_1}$ and $D_{z_2}$ is $1$, then $F(D_{z_3}) \nsubseteq F(D_{z_1}) \cup F(D_{z_2})$.
\end{lemma}

\begin{proof} 
Assume $D_{z_1}$ and $D_{z_2}$ have Hamming distance $1$.
Let $l$ be the unique coordinate in which $D_{z_1}$ and $D_{z_2}$ differ, that is, $D_{z_1} = (d_1, d_2, \ldots, d_l, \ldots, d_k)$ and $D_{z_2} = (d_1, d_2, \ldots, d'_l, \ldots, d_k)$ with $d_j \in \mathbb{Z}_{m_j}$, $d_l, d'_l \in \mathbb{Z}_{m_l}$, and $d_l \neq d'_l$. Let $A = \{f(1, d_1), \ldots, f(l-1, d_{l-1}), f(l+1, d_{l+1}), \ldots, f(k, d_k)\}, \\ f_{1} = f(l, d_l),$ and $f_{2} = f(l, d'_l)$. Thus, $F(D_{z_1}) = A \cup \{f_1\}$ and $F(D_{z_2}) = A \cup \{f_2\}$. If $A \subseteq F(D_{z_3})$, then $F(D_{z_3}) = A \cup \{f_3\}$ where $f_3 \notin \{f_1, f_2\}$. So, $f_3 \notin F(D_{z_1}) \cup F(D_{z_2})$. If $A \nsubseteq F(D_{z_3})$, then there exists $x \in [1, k] \setminus \{l\}$ such that the $x^{th}$ coordinate of $D_{z_3}$, $y_x \neq d_x$. In this case, $f(x, y_x) \in F(D_{z_3})$ and $f(x, y_x) \notin F(D_{z_1}) \cup F(D_{z_2})$. Therefore, in both cases, $F(D_{z_3}) \nsubseteq F(D_{z_1}) \cup F({D_{z_2}})$. 
\end{proof}

\begin{theorem}
\label{mixed-radix-to-CFF-paths}
Let $n = \prod_{i = 1}^{k} m_i$ and $t = \sum_{i = 1}^{k} m_i$. Let $G_k = G_{k}^{(m_1, m_2, \ldots, m_k)}$ be a Gray code and define $\mathcal{F}(G_k)$ as in Construction~\ref{CFFonPaths-cons}. Then, $\mathcal{F}(G_k)$ is a $P_n$-CFF$(t, n)$.  If $G_k$ is cyclic, then $\mathcal{F}(G_k)$ is a $C_n$-CFF$(t, n)$.
\end{theorem}

\begin{proof} Since the sets in $\mathcal{B}(G_k)$ are pairwise distinct and of size $k$, $\mathcal{F}(G_k)$ is a Sperner family. Since any two consecutive codewords $D_i$ and $D_{i+1}$ with $i \in [0, n-2]$ have Hamming distance $1$, by Lemma~\ref{hamming-distance-1},  $F(D_i) \cup F(D_{i+1})$ does not contain any other subset of $\mathcal{F}(G_k)$. Thus, $\mathcal{F}(G_k)$ is a $P_n$-CFF$(t, n)$. 

If $G_{k}$ is cyclic, then the first and last codeword $D_0$ and $D_{n-1}$ have Hamming distance $1$. Thus, by Lemma~\ref{hamming-distance-1}, $F(D_0) \cup F(D_{n-1})$ does not contain any other subset of $\mathcal{F}(G_{k})$. Thus, $\mathcal{F}(G_{k})$ is a $C_n$-CFF$(t, n)$.
\end{proof}

\begin{example}
We present an example of the binary reflected Gray code of length $3$ and construct the corresponding set system as described in Construction~\ref{CFFonPaths-cons}. By Theorem~\ref{mixed-radix-to-CFF-paths} and Proposition~\ref{cyclic-even-m1}, $\mathcal{F}$ is a $C_8$-CFF$(6, 8)$. See Figure~\ref{fig:BRGC3-C8CFF}. 

\begin{figure}[htbp]
  \centering

   {\large $[1, 6] = \{1, 2\} \; \mathbin{\mathbin{\dot\cup}} \; \{3, 4\} \; \mathbin{\dot\cup} \; \{5, 6\}.$}

   \vspace{0.5cm}

  \begin{minipage}[t]{0.45\textwidth}
    \centering

    \begin{tikzpicture}
        \node[draw, minimum width=0.5cm, minimum height=3cm, align=left] (box1) at (0, 0) {
            \begin{minipage}{5cm}
                \begin{enumerate}
                    \item $000 \rightarrow \{1,3,5\}$ 
                    \item $001 \rightarrow \{1,3,6\}$
                    \item $011 \rightarrow \{1,4,6\}$
                    \item $010 \rightarrow \{1,4,5\}$
                    \item $110 \rightarrow \{2,4,5\}$
                    \item $111 \rightarrow \{2,4,6\}$ 
                    \item $101 \rightarrow \{2,3,6\}$ 
                    \item $100 \rightarrow \{2,3,5\}$
                \end{enumerate}
            \end{minipage}
            \hfill
        };
    \end{tikzpicture}
  \end{minipage}
  \hfill
  \begin{minipage}[t]{0.45\textwidth}
    \centering
    \includegraphics[width=0.8\linewidth]{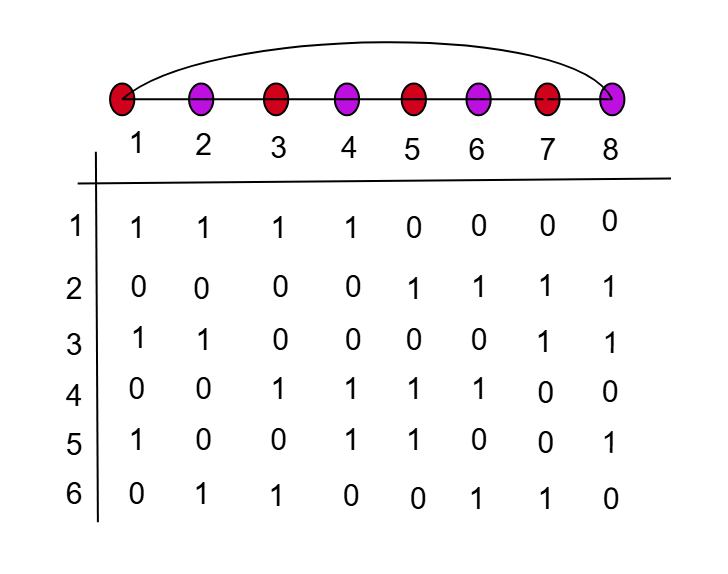}
  \end{minipage}

  \caption{$C_8$-CFF$(6, 8)$ constructed from a binary reflected Gray code of length $3$ (left) and the corresponding $ 6 \times 8 \;\; \overline{C_8}$-disjunct matrix (right)}
    \label{fig:BRGC3-C8CFF}
\end{figure}
\end{example}

The ability to construct $C_n$-CFFs out of Gray codes ultimately leads us to the problem of how a set of size $t$ should be partitioned into distinct parts such that the product $n$ of the sizes of the parts is maximum. This is a well-studied problem, which has been addressed in competitions \cite{scholes1976, scholes1979}. Its solution appears as sequence A000792 in the On-Line Encyclopedia of
Integer Sequences \cite{sloane}, given in the next theorem.

\begin{theorem} {\cite{sloane}}
\label{maximum-product-partitions}
Let $a(m)$ be the function that gives the maximum product of the parts of $m$, that is, $$a(m) = \max\left\{\prod_{i = 1}^{l} a_i : a_1 + a_2 + \cdots + a_l = m,\; a_i \in \mathbb{Z}^+, l \in \mathbb{Z}^+ \right\}.$$

 Then,

    \[
a(m) =
\begin{cases}
    3^k & \text{if } m = 3k, \\
    4 \cdot 3^{k-1} & \text{if } m = 3k + 1, \\
    2 \cdot 3^k & \text{if } m = 3k + 2.
\end{cases}
\]
\end{theorem}

\begin{theorem}
\label{maximal-path-cycle-thm}
        \[
t(P_n) \leq t(C_n) \leq
\begin{cases}
    3k & \text{if } n \in (2 \cdot 3^{k-1}, 3^k] \text{ for some } k \geq 1, \\
    3k + 1 & \text{if } n \in (3^k, 4 \cdot 3^{k-1}] \text{ for some } k \geq 1, \\
    3k + 2 & \text{if } n \in (4 \cdot 3^{k-1}, 2 \cdot 3^k] \text{ for some } k \geq 1.
\end{cases}
\]
\end{theorem}

\begin{proof} We split the proof into three different cases depending on the parity of $n$ modulo $3$.

\noindent \underline{\textbf{Case 1:}}
    $n \in (2 \cdot 3^{k-1}, 3^k] \text{ for some } k \geq 1$. \\
For $k = 1$, use the identity matrix of order $3$ as $P_3$-CFF or $C_3$-CFF. For $k \geq 2$, use Construction~\ref{CFFonPaths-cons} applied to the Gray code $M_{k}^{3}$ to obtain $\mathcal{F}(M_{k}^{3})$. By Proposition~\ref{loc-rec-fixed-proof}, $M_{k}^{3}$ is cyclic. By Theorem~\ref{mixed-radix-to-CFF-paths}, $\mathcal{F}(M_{k}^{3})$ is a $C_{3^k}$-CFF$(3k, 3^k)$. So, it remains to address the case $2 \cdot 3^{k-1} < n < 3^k$. By Construction~\ref{loc-rec-fixedradix}, we observe that each triplet of codewords of $M_{k}^{3}$ appears in the form $D \cdot i$, $D \cdot (i + 1)$, $D \cdot (i + 2)$ with $i \in \mathbb{Z}_3$, where $D$ is a codeword of $M_{k-1}^{3}$ and there are $3^{k-1}$ such triplets. We delete $a = 3^k - n$ codewords $D_j$ with $j \equiv 1 \pmod{3}$, where $1 \leq a \leq 3^{k-1} - 1$ to obtain the resultant ordering of the remaining codewords $(M_{k}^{3})^a$, where consecutive codewords also have Hamming distance $1$. By similar argument as in the proof of Lemma~\ref{hamming-distance-1} and Theorem~\ref{mixed-radix-to-CFF-paths}, $\mathcal{F}((M_{k}^{3})^a)$ is a $C_n$-CFF$(3k, n)$.

\noindent \underline{\textbf{Case 2:}} $n \in (3^k, 4 \cdot 3^{k-1}] \text{ for some } k \geq 1$. \\
Use Construction~\ref{mixed-radix-gray-reflected-code-recursive} to build the reflected Gray code $R_{k + 1} = R_{k+1}^{(2, 2, 3, 3, \ldots, 3)}$ and by Proposition~\ref{cyclic-even-m1}, $R_{k + 1}$ is cyclic. By Theorem~\ref{mixed-radix-to-CFF-paths}, $\mathcal{F}(R_{k + 1})$ is a $C_{4 \cdot 3^{k-1}}$-CFF$(3k+ 1, 4 \cdot 3^{k-1})$, which solves the case $n = 4 \cdot 3^{k-1}$. For $3^k < n < 4 \cdot 3^{k-1}$, we employ a similar technique as in Case 1. Note that each triplet of codewords of $R_{k+1}^{(2, 2, 3, 3, \ldots, 3)}$ appears in the form $D \cdot 0$, $D \cdot 1$, $D \cdot 2$ or $D \cdot 2$, $D \cdot 1$, $D \cdot 0$, where $D$ is a codeword of $R_{k}^{(2, 2, 3, 3, \ldots, 3)}$. Remove $a = 4 \cdot 3^{k-1} - n$ codewords of type $D \cdot 1$. This is always possible since $1 \leq a \leq 3^{k-1}-1$ and there are $4 \cdot 3^{k-2} > 3^{k-1}$ such codewords. Thus, we obtain a $C_n$-CFF$(3k + 1, n)$. 

\noindent \underline{\textbf{Case 3:}} $n \in (4 \cdot 3^{k-1}, 2 \cdot 3^k] \text{ for some } k \geq 1$. \\
Use Construction~\ref{mixed-radix-gray-reflected-code-recursive} to build the reflected Gray code $R_{k + 1} = R_{k+1}^{(2, 3, 3, \ldots, 3)}$ and by Proposition~\ref{cyclic-even-m1}, $R_{k + 1}$ is cyclic. By Theorem~\ref{mixed-radix-to-CFF-paths}, $\mathcal{F}(R_{k + 1})$ is a $C_{2 \cdot 3^{k}}$-CFF$(3k+ 2, 2 \cdot 3^{k})$, which solves the case $n = 2 \cdot 3^{k}$. For $4 \cdot 3^{k-1} < n < 2 \cdot 3^k$, we employ a similar technique as in Case 2, and remove $a = 2 \cdot 3^{k} - n$ codewords of type $D \cdot 1$. This is always possible since $ 1 \leq a \leq 2 \cdot 3^{k-1}-1$ and there are $2 \cdot 3^{k-1}$ such codewords. Thus, we obtain a $C_n$-CFF$(3k + 2, n)$.    
\end{proof}

The following corollary gives asymptotic bounds for $t(P_n)$ and $t(C_n)$.

 \begin{corollary}
 \label{asymptotics-t(Cn)}
   For \( n \geq 4 \), we have $t(C_n) \leq 3 \log_{3} n + 2 \leq 1.893 \log_2 n + \mathcal{O} (1) $.
\end{corollary}

\begin{proof}
    By Theorem~\ref{maximal-path-cycle-thm}, $t(C_n) \leq 3k$ when $2 \cdot 3^{k-1} < n \leq 3^k$. Taking logarithms base $3$ and multiplying through by $3$ from the inequality $2 \cdot 3^{k-1} < n $, we get $3k < 3\log_{3} n + 3 - 3\log_{3} 2$. Theorem~\ref{maximal-path-cycle-thm} implies $t(C_n) \leq 3k + 1$ for the case $n \in (3^k, 4 \cdot 3^{k-1}]$ for some $k \geq 1$. Going through similar calculations as above, from $3^k < n$ we obtain $3k + 1 <  3\log_{3} n + 1.$ Similarly, in the case $t(C_n) \leq 3k + 2$, when $n \in (4 \cdot 3^{k-1}, 2 \cdot 3^{k}]$ for some $k \geq 1$, we obtain $3k + 2 < 3\log_{3} n + 5 - 3\log_{3} 4.$ Hence, from the above three cases, we have $t(C_n) \leq 3 \log_{3} n + 2.$ Since $3\log_{3} n = \frac{3}{\log_{2}3} \log_{2} n \leq 1.893 \log_{2}(n)$, we get the desired result.
\end{proof}

\begin{example}
\label{P36-P54}
An example of a $C_{36}$-CFF and a $C_{54}$-CFF built on the ground sets $[1, 10]$ and $[1, 11]$, respectively, is given in Figure~\ref{fig:P36-P54}. They are also $P_{36}$-CFF and $P_{54}$-CFF, respectively. We also provide the reflected Gray code $R_{4}^{(2,2,3,3)}$ used to construct the $C_{36}$-CFF in Figure~\ref{fig:P_36-CFF(10,36)}. In contrast, we present the $C_{54}$-CFF$(11,54)$ in Figure~\ref{fig:P36-P54} without including the corresponding reflected Gray code $R_{4}^{(2,3,3,3)}$. We also provide an example of a $P_{27}$-CFF (using $R_{3}^{(3, 3, 3)})$ and a $C_{27}$-CFF (using $M_{3}^{3})$ built on $[1, 9]$ in Figure~\ref{fig:927}. Figure~\ref{fig:P27-P19} provides an example of a $C_{19}$-CFF$(9, 19)$ by shortening $M_{3}^{3}$.

\begin{figure}[h!]
    \centering
    \begin{subfigure}[t]{0.45\textwidth}
        \centering
        {\tiny $[1, 10] = \{1, 2\} \; \mathbin{\dot\cup} \; \{3, 4\} \; \mathbin{\dot\cup} \; \{5, 6, 7\} \; \mathbin{\dot\cup} \; \{8, 9, 10\}.$}
        \includegraphics[
    width=\textwidth,
    trim={0 6cm 6cm 3.5cm},
    clip
]{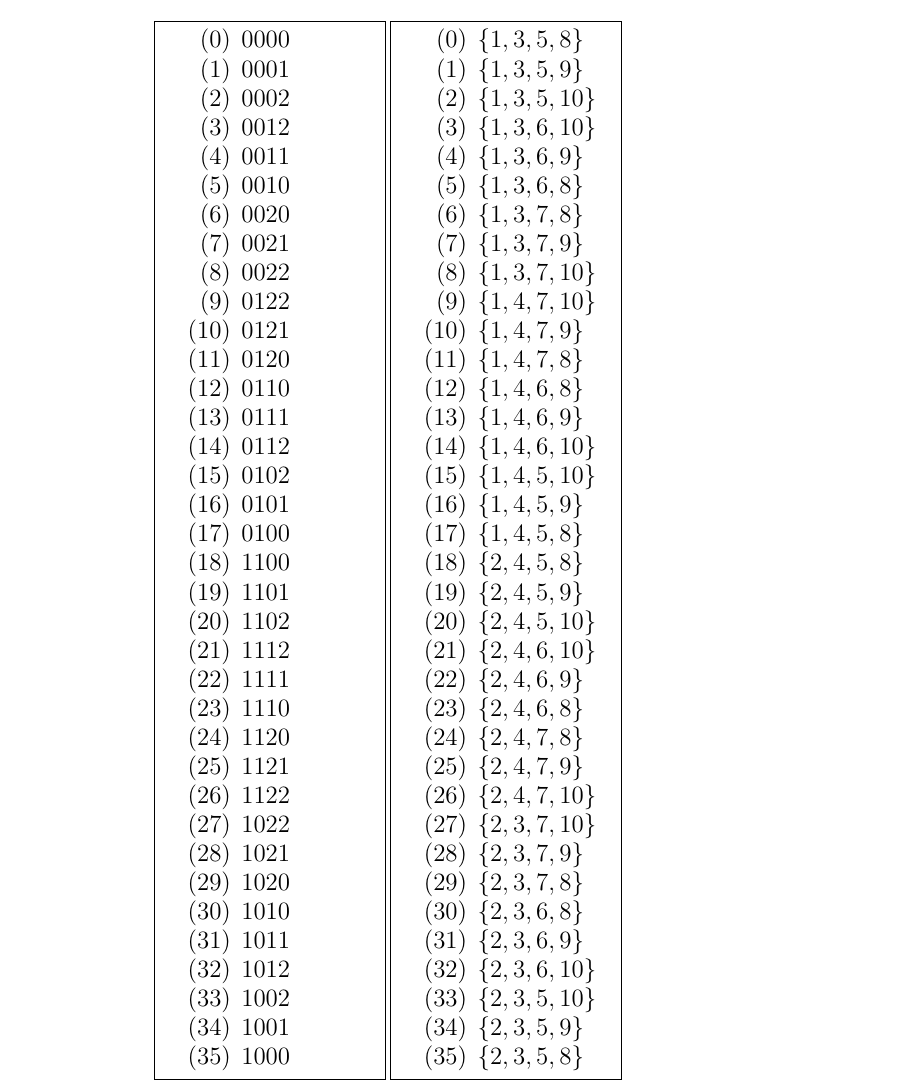} 
        \caption{$R_{4}^{(2, 2, 3, 3)}$ and the corresponding $C_{36}$-CFF$(10, 36)$ built on $[1,10]$}
        \label{fig:P_36-CFF(10,36)}
    \end{subfigure}
    \begin{subfigure}[t]{0.54\textwidth}
        \centering
         {\tiny $[1, 11] = \{1, 2\} \; \mathbin{\dot\cup} \; \{3, 4, 5\} \; \mathbin{\dot\cup} \; \{6, 7, 8\} \; \mathbin{\dot\cup} \; \{9, 10, 11\}.$}
         \includegraphics[
    trim=0cm 0cm 0cm 1.5cm,
    clip,
    width=0.8\textwidth
]{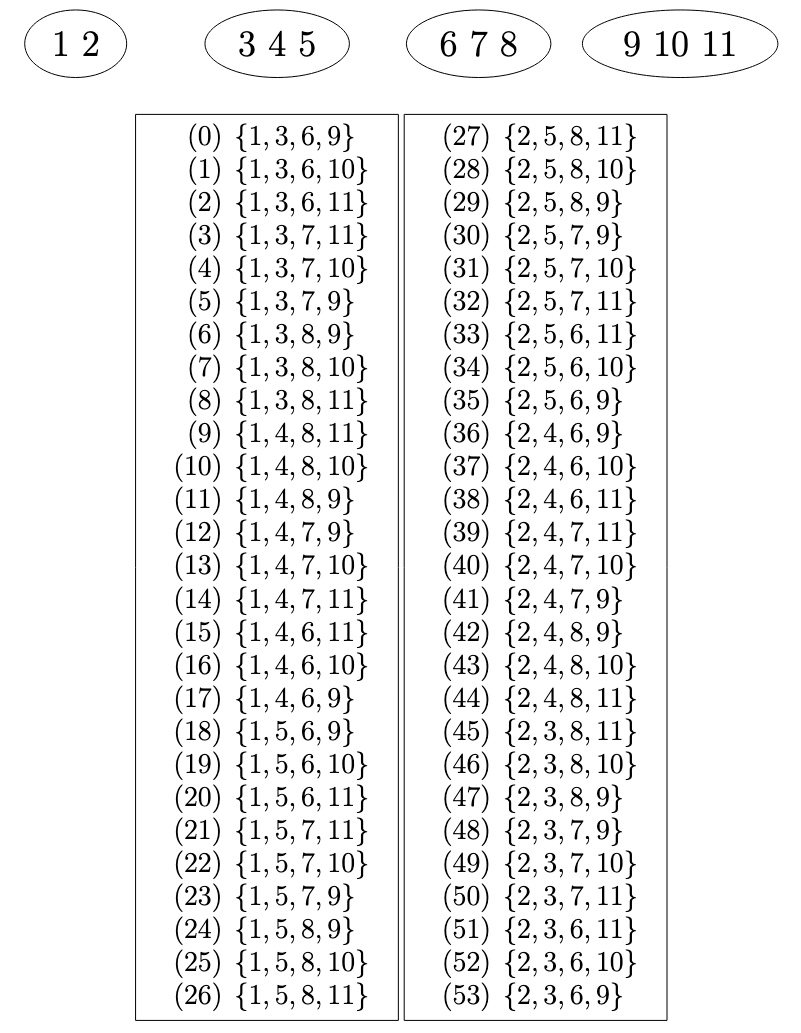} 
        \caption{A $C_{54}$-CFF$(11, 54)$ built on $[1,11]$ using $R_{4}^{(2, 3, 3, 3)}$}
        \label{fig:C_54-CFF(11,54)}
    \end{subfigure}
    \caption{Examples of a $C_{36}$-CFF and a $C_{54}$-CFF}
    \label{fig:P36-P54}
\end{figure}

\begin{figure}[htbp]
    \centering

     {\footnotesize $[1, 9] = \{1, 2, 3\} \; \mathbin{\dot\cup} \; \{4, 5, 6\} \; \mathbin{\dot\cup} \; \{7, 8, 9\}.$}
    
  \begin{subfigure}[t]{0.35\textwidth}
        \centering
        \includegraphics[width=\textwidth]{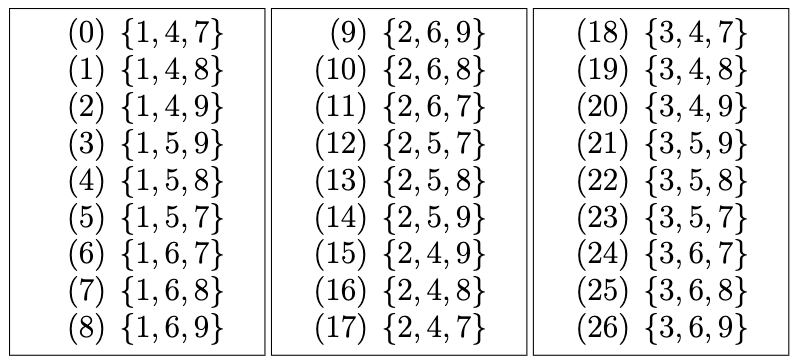} 
        \caption{A $P_{27}$-CFF$(9, 27)$ built on $[1,9]$ using $R_{3}^{(3, 3, 3)}$}
        \label{fig:P27}
    \end{subfigure}
    \hfill
    \begin{subfigure}[t]{0.35\textwidth}
        \centering
        \includegraphics[width=\textwidth]{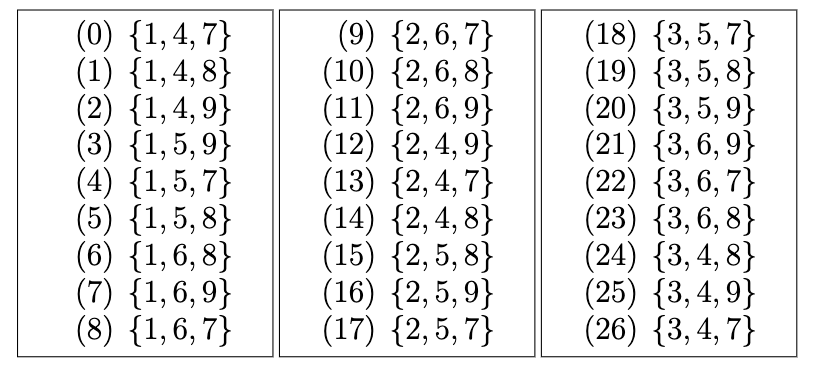} 
        \caption{A $C_{27}$-CFF$(9, 27)$ and a $P_{27}$-CFF$(9, 27)$ built on $[1,9]$ using $M_{3}^{3}$}
        \label{fig:C27}
    \end{subfigure}
\caption{Examples of a $P_{27}$-CFF and a $C_{27}$-CFF}   
\label{fig:927}
\end{figure} 
    \end{example}

\begin{figure}[h!]
    \centering
    {\small $[1, 9] = \{1, 2, 3\} \; \mathbin{\dot\cup} \; \{4, 5, 6\} \; \mathbin{\dot\cup} \; \{7, 8, 9\}.$}

    \begin{subfigure}[t]{0.54\textwidth}
        \centering
        \includegraphics[
    trim=0cm 0cm 0cm 1.3cm,
    clip,
    width=0.7\textwidth
]{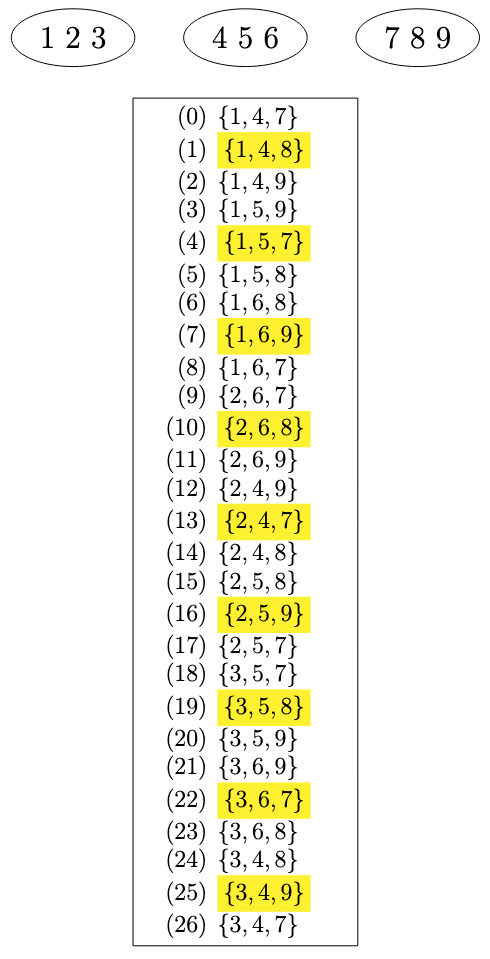}
        \caption{A $C_{27}$-CFF$(9, 27)$ built on $[1,9]$ using $M_{3}^{3}$}
        \label{fig:C_27-CFF(9,27)-highlighted}
    \end{subfigure}
    \begin{subfigure}[t]{0.45\textwidth}
        \centering
      \includegraphics[
    trim=0cm 12cm 5cm 3cm,
    clip,
    width=1.65\textwidth
]{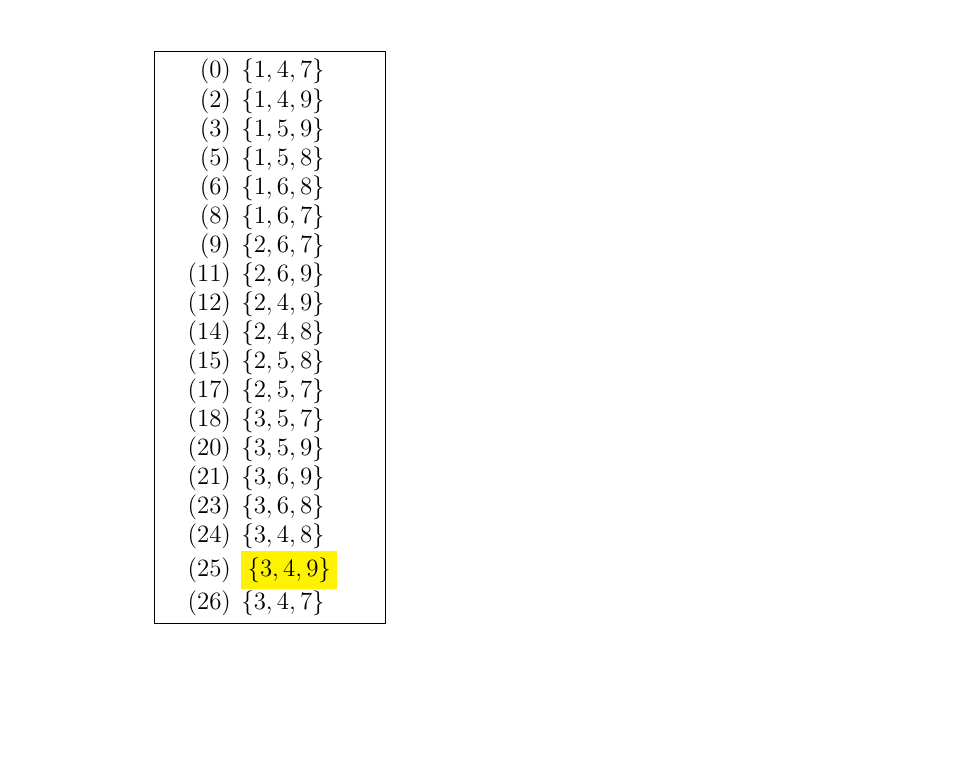}
    
        \caption{A $C_{19}$-CFF$(9, 19)$ built on $[1,9]$ using $(M_{3}^{3})^8$. $(M_{3}^{3})^8$ is constructed by deleting the first $8$ highlighted codewords in \ref{fig:C_27-CFF(9,27)-highlighted}.}
        \label{fig:C_19-CFF(9,19)}
    \end{subfigure}
    \caption{A $C_{27}$-CFF$(9, 27)$ and a $C_{19}$-CFF$(9, 19)$}
    \label{fig:P27-P19}
\end{figure}

Recently, Viana and Idalino \cite{viana2026constructing} used SAT-solvers for finding CFFs, including cyclic consecutive CFFs, which for the case $d=2$ corresponds to graphs that are cycles.
Table~\ref{tab:t(Cn)-comparison} compares the upper bounds on $t(C_n)$ given in Theorem~\ref{maximal-path-cycle-thm} with the computational upper bounds provided in \cite{viana2026constructing}.
In their work, they computationally obtain exact values for the same values we prove to be exact in Proposition~\ref{small-cycles-234} and Theorem~\ref{small-cycles-678}. They improve the bound in Theorem~\ref{maximal-path-cycle-thm} by at most $2$ for  $21\leq n\leq 308$, while the upper bounds match for $302 <n \leq 323$, after which their computation ends.

\begin{table}[t]
    \centering   \caption{Comparison of the upper bounds on $t(C_n)$ from Theorem~\ref{maximal-path-cycle-thm} with those obtained computationally in \cite{viana2026constructing}}
    \label{tab:t(Cn)-comparison}
    \begin{tabular}{|c|c|c|}
    \hline
       $n$  & \makecell{$t(C_n)$ \\ (Theorem~\ref{maximal-path-cycle-thm})} &  \makecell{$t(C_n)$ \\ \citep{viana2026constructing}}  \\
        \hline
        $3$ & $\mathbf{3}$ & $\mathbf{3}$ \\
        \hline
        $4$ & $\mathbf{4}$ & $\mathbf{4}$ \\
        \hline
        $6$ & $\mathbf{5}$ & $\mathbf{5}$ \\
        \hline
        $9$ & $\mathbf{6}$ & $\mathbf{6}$ \\
        \hline
        $21$ & $9$ & $7$ \\
        \hline
        $28$ & $10$ & $8$ \\
        \hline
        $45$ & $11$  & $9$ \\
        \hline
        $62$ & $12$ & $10$ \\
        \hline
        $91$ & $13$ & $11$ \\
        \hline
        $128$ & $14$ & $12$ \\
        \hline
        $176$ & $15$ & $13$ \\
        \hline
        $231$ & $15$ & $14$ \\
        \hline
        $302$ & $16$ & $15$ \\
        \hline
        $323$ & $16$ & $16$ \\
        \hline
    \end{tabular}
\end{table}

We further show that the set system produced in Construction~\ref{CFFonPaths-cons} is also a CFF on a generalized Hamming graph. Let $H = K_{a_1} \; \square \; K_{a_2} \; \square \; \cdots \; \square \; K_{a_d}$ with $V(K_{a_l}) = \{0, 1, \ldots, a_{l} - 1\}$ for $l \in [1, d]$. Then, $V(H)$ is a set of all $d$-tuples in $V(K_{a_1}) \times V(K_{a_2}) \times \cdots \times V(K_{a_d})$ and, for a pair of vertices \sloppy$(x_1, x_2, \ldots, x_d), (y_1, y_2, \ldots, y_d) \in V(H)$, we have that $\{(x_1, x_2, \ldots, x_d), (y_1, y_2, \ldots, y_d)\} \in E(H) $ if and only if for some $i \in [1, d]$, we have $x_i \neq y_i$ with $\{x_i, y_i\} \in E(K_{a_i})$, and for each $j \in [1, d] \setminus \{i\}$, we have $x_j = y_j$. In other words, $K_{a_1} \; \square \; K_{a_2} \; \square \; \cdots \; \square \; K_{a_d}$ has the vertex set consisting of codewords of $G_{d}^{(a_1, a_2, \ldots, a_d)}$ and $\{(x_1, x_2, \ldots, x_d), (y_1, y_2, \ldots, y_d)\} \in E(K_{a_1} \; \square \; K_{a_2} \; \square \; \cdots \; \square \; K_{a_d}) $ if and only if the Hamming distance between $(x_1, x_2, \ldots, x_d)$ and $(y_1, y_2, \ldots, y_d) $ is $1$. 

\begin{definition}
   Let $\mathcal{F}$ be a set system consisting of $n \geq 3$ subsets of a $t$-set. A graph $G$ is the \emph{CFF-maximal graph for} $\mathcal{F}$ if it is the maximal graph $G$ on $n$ vertices such that $\mathcal{F}$ is a $G$-CFF$(t, n)$. More specifically, $G = ([1, n], E)$ where 
    $$E = \left\{e \in \binom{[1, n]}{2} : \mathcal{F} \text{ is cover-free and in-cover-free for } e\right\}.$$
\end{definition}

 \begin{prop}
\label{Gray-maximal-hamming}
Let $G_{k}^{(m_1, m_2, \ldots, m_k)}$ be a Gray code and let $t, n$, and $\mathcal{F}\left(G_{k}^{(m_1, m_2, \ldots, m_k)}\right)$ be defined as in Construction~\ref{CFFonPaths-cons}. The CFF-maximal graph $\mathcal{H}$ for $\mathcal{F}\left(G_{k}^{(m_1, m_2, \ldots, m_k)}\right)$ is the Hamming graph $\mathcal{H} = H_{m_1, m_2, \ldots, m_k}$. In addition, $t(C_n) \leq t\left(H_{m_1, m_2, \ldots, m_k}\right) \leq \sum_{i = 1}^{k} m_i$.
\end{prop}

\begin{proof} 
 Let $\mathcal{H}$ be the Hamming graph $H_{m_1, m_2, \ldots, m_k}$ with $V(H_{m_1, m_2, \ldots, m_k}) = \mathbb{Z}_{m_1} \times \cdots \times \mathbb{Z}_{m_k}$. Let $e \in \binom{[1, n]}{2}$, where $n = \prod_{i = 1}^{k} m_i$. If $e = \{D_i, D_j\} \in E(H_{m_1, m_2, \ldots, m_k})$, then the Hamming distance between $D_i$ and $D_j$ is exactly $1$, and by Lemma~\ref{hamming-distance-1}, for any $D_l \neq D_i, D_j$, we have $F(D_l) \nsubseteq F(D_i) \cup F(D_j)$. Also, $F(D_i) \nsubseteq F(D_j)$ and $F(D_j) \nsubseteq F(D_i)$ because $|F(D_i)| = |F(D_j)|$. Thus, $\mathcal{F}$ is cover-free for $e$ and in-cover-free for $e$. On the other hand, if $e' = \{D_i, D_j\} \notin E(H_{m_1, m_2, \ldots, m_k})$, then the Hamming distance between $D_i$ and $D_j$ is strictly greater than $1$. Without loss of generality, let us assume that $D_i$ and $D_j$ differ in the $1^{st}$ and $2^{nd}$ co-ordinates. Thus, $D_i$ and $D_j$ must be of the forms $D_i = (d_1, d_2, d_3, \ldots, d_k)$ and $D_j = (d'_1, d'_2, d'_3, \ldots, d'_k)$ with $d_l, d'_l \in \mathbb{Z}_{m_l}, l \in [1, k]$, and $d_1 \neq d'_{1}, d_2 \neq d'_{2}$. Let $D_y = (d_1, d'_2, d_3, \ldots, d_n)$. Then, $F(D_y) \subseteq F(D_i) \cup F(D_j)$, which implies that $\mathcal{F}$ is not cover-free for $e'$. This shows that $H_{m_1, m_2, \ldots, m_k}$ is the CFF-maximal graph for $\mathcal{F}$, and that $t\left(H_{m_1, m_2, \ldots, m_k}\right) \leq \sum_{i = 1}^{k} m_i$. Since $H_{m_1, m_2, \ldots, m_k}$ is Hamiltonian, by Proposition~\ref{subgraph-CFF}, $t(C_n) \leq t(H_{m_1, m_2, \ldots, m_k})$.
  \end{proof}

Next, we present Construction~\ref{cycle-path-doubling-construction}, which provides the bounds in Proposition~\ref{cycle-path-doubling-bound}.

   \begin{Construction}
\label{cycle-path-doubling-construction}
    Let $A$ be a $t_1 \times n$ binary matrix and $I_2$ be the $2 \times 2$ identity matrix. Let $B$ be a matrix with columns of $A$ in the reverse order. Using $A$ and $I_2$, we create a $(t_1 + 2) \times 2n$ binary matrix $C$ by where columns are labeled by $1, 2, \ldots, n, n+1, n+2, \ldots, 2n$. The construction is as follows. We copy the matrices $A$ and $B$ under vertices $1, \ldots, n$ and  $n+1, \ldots, 2n$, respectively and call the resultant matrix $C_1$.  We copy the first column and second column of $I_2$ under vertices labeled $1, \ldots, n$ and $n+1, \ldots, 2n$, respectively and call the resultant matrix $C_2$. Let $C$ be the vertical concatenation of $C_1$ and $C_2$. We label the rows of $C$ as the elements of $[1, t_1 + 2]$ with rows labeled $[1, t_1]$ corresponding to the rows of $C_1$ and rows labeled $[t_1 + 1, t_1 + 2]$ corresponding to the rows of $C_2$.
\end{Construction}

\begin{prop}
\label{cycle-path-doubling-bound}
    $t(C_{2n}) \leq t(C_n) + 2$ and $t(P_{2n}) \leq t(P_n) + 2$.  
\end{prop}

 \begin{proof}
     Let $A$ be a $t_1 \times n$ $\overline{C_n}$-disjunct matrix and $B$ be the matrix with columns of $A$ in the reverse order. We claim that the matrix $C$, obtained using Construction~\ref{cycle-path-doubling-construction}, is a $(t_1 + 2) \times 2n$ $\overline{C_{2n}}$-disjunct matrix. Indeed, since $A$ is cover-free for an edge $\{i, i+1\}$ with $i \in [1, n-1]$ and furthermore, $C[t_1 + 2, i] = C[t_1 + 2, i+1] = 0$ and $C[t_1 + 2, j] = 1$ for all $j \in [n+1, 2n]$, this implies that $C$ is cover-free for $\{i, i+1\}$. Using a similar argument, we can conclude that $C$ is cover-free for the edge $\{j , j+1\}$ for all $j \in [n+1, 2n-1]$. Since by Theorem~\ref{GbarCFF-1CFF}, $A$ is $\overline{1}$-disjunct and for the edge $\{n, n+1\}$, we have $C_1^{n} = C_1^{n+1} = A^{n}$. Then for any $j \in [1, 2n] \setminus \{n, n+1\}$, there exists a row $l_1 \in [1, t_1]$ in $C_1$ such that $C_1[l_1, n] = C_1[l_1, n+1] = 0$ and $C_1[l_1, j] = 1$. Thus, $C$ is cover-free for the edge $\{n, n+1\}$. Using a similar argument, we can also show that $C$ is cover-free for the edge $\{1, 2n\}$. Thus, $C$ is $C_{2n}$-disjunct. Since $\delta(C_{2n}) \geq 2$, by Proposition~\ref{GbarCFF=GCFF-nec-suf}, $C$ is $\overline{C_{2n}}$-disjunct. Thus, $t(C_{2n}) \leq t(C_n) + 2$. A similar construction and proof can be applied to a $t_1 \times n$ $\overline{P_n}$-disjunct matrix $A$ to obtain $t(P_{2n}) \leq t(P_n) + 2$.
\end{proof}

It would be interesting to investigate when $t(C_{2n})-t(C_{n})$ is $1$ or $2$
and when $t(P_{2n})-t(P_{n})$ is $1$ or $2$. We conjecture both of these quantities are positive. In Table~\ref{tab:small-n-path-cycle-wheel}, we have examples of $t(C_{2n})-t(C_{n})=2$ and $t(P_{2n})-t(P_{n})=2$, when $n=3,4$, and an example of $t(P_{2n})-t(P_{n})=1$, when $n = 5$. 

Next, using $t(C_n)$, we provide an upper and lower bound of $t(W_n)$. 

\begin{corollary}
\label{wheel}
For $n \geq 3$, $\max\{t(1, n) + 1, t(C_n), t(C_{n+1})\} \leq t(W_{n+1}) \leq t(C_n) + 1$. 
\end{corollary}

\begin{proof}
Since $W_{n+1}$ is Hamiltonian, by Proposition~\ref{subgraph-CFF}, $t(C_{n+1}) \leq t(W_{n+1})$, and since $W_{n+1} = C_n + \{0\}$, by Theorem~\ref{star-cons-gen}, $t(1, n) + 1 \leq t(W_{n+1}) \leq t(C_{n}) + 1$. Since $C_n \subseteq W_{n+1}$, by Proposition~\ref{subgraph-CFF}, we have $t(C_n) \leq t(W_{n+1})$. Thus, we have the desired bound. 
\end{proof}

\section{Generalizing the construction of CFFs on star graphs}
\label{star-gen-cons}

We generalize the construction of CFFs on star graphs to obtain CFFs for other families of graphs, such as windmill graphs, friendship graphs, and wheel graphs.

\begin{definition}
\label{def:windmill}
    Let $k, n \geq 2$. The \emph{windmill graph} $\mathrm{Wd}(k,n)$ is an undirected graph constructed by joining $n$ copies of the complete graph $K_k$ at a shared universal vertex. Thus, $\mathrm{Wd}(k,n)$ has $n(k-1) + 1$ vertices. The \emph{friendship graph} is $\mathrm{Wd}(3, n)$ which has $2n+1$ vertices and is denoted by $F_{2n + 1}$.
\end{definition}

Thus, the star graph $S_{n+1}$ is $\mathrm{Wd}(2,n)$. 

\begin{example}
We give an example of a friendship graph and a windmill graph in Figures~\ref{fig:Wd34} and \ref{fig:Wd54}.
    \begin{figure}[ht]
    \centering
    \begin{subfigure}[b]{0.45\textwidth}
        \centering
        \begin{tikzpicture}[scale=1.3]

            \node[draw, circle, fill=black, inner sep=1pt, label=below:] (a) at (0,0) {};

            \def\n{4}

            \foreach \i in {0, 1, 2, 3} {
                \foreach \j in {1, 2} {
                    \pgfmathsetmacro{\angle}{90 + \i*360/\n + \j*120/\n}
                    \node[draw, circle, fill=black, inner sep=1pt] (b\i\j) at (\angle:2) {};
                    \draw (a) -- (b\i\j);
                }
                \draw (b\i1) -- (b\i2);
                \draw (b\i2) -- (b\i1);
            }
        \end{tikzpicture}
        \caption{Friendship Graph: $F_{9} = \mathrm{Wd}(3,4)$}
        \label{fig:Wd34}
    \end{subfigure}\hfill
    \begin{subfigure}[b]{0.45\textwidth}
        \centering
        \begin{tikzpicture}[scale=1.3]

            \node[draw, circle, fill=black, inner sep=1pt, label=below:] (a) at (0,0) {};

            \foreach \i in {0, 90, 180, 270} {
                \foreach \j in {1, 2, 3, 4} {
                    \node[draw, circle, fill=black, inner sep=1pt] (b\i\j) at (\i+\j*18:2) {};
                    \draw (a) -- (b\i\j);
                }
                \foreach \j in {1, 2, 3, 4} {
                    \foreach \k in {1, 2, 3, 4} {
                        \ifnum\j<\k
                          \draw (b\i\j) -- (b\i\k);
                        \fi
                    }
                }
            }
        \end{tikzpicture}
        \caption{Windmill graph: $\mathrm{Wd}(5,4)$}
        \label{fig:Wd54}
    \end{subfigure}
\caption{Examples of a friendship graph and a windmill graph}
\end{figure}
\end{example}

Below we give a construction for a $\overline{\mathrm{Wd}(k, n)}$-disjunct matrix. 

\begin{Construction}
\label{windmil-cff-cons}
Let $\mathrm{Wd}(k, n)$ be a windmill graph with $k \geq 3$ and $n \geq 2$. Let $v_0$ be the universal vertex and for $1 \leq i \leq n$, let the remaining vertices of the $i^{th}$ subgraph $K_{k-1}$ be $ v_{i, 1}, v_{i, 2}, \ldots, v_{i, k-1}$. Let $A$ be a $t(1, n) \times n$ $1$-disjunct matrix and let $M$ be a $t(2, k-1) \times (k-1)$ $2$-disjunct matrix if $k \geq 4$, or a $2 \times 2$ identity matrix if $k = 3$. Let $t_2$ be $t(2, k-1)$ if $k \geq 4$, and $2$ if $k = 3$. 

Let $\mathcal{C}_1, \mathcal{C}_2,$ and $\mathcal{C}_3$  be matrices where each column is labeled with the vertices of the graph $\mathrm{Wd}(k, n)$ as described next. In $\mathcal{C}_1$, $\mathcal{C}_{1}^{v_0} = \mathbf{0}_{t(1, n) \times 1}$ and $\mathcal{C}_{1}^{v_{i, j}} = A^i$ for all $j \in [1, k-1]$. In $\mathcal{C}_2$, $\mathcal{C}_{2}^{v_0} = \mathbf{0}_{t_2 \times 1}$ and for all $i \in [1, n]$, $\mathcal{C}_{2}^{v_{i, j}} = M^j$ with $j \in [1, k-1]$. $\mathcal{C}_3$ is the row matrix such that $\mathcal{C}_3[1, v_0] = 1$ and $\mathcal{C}_3[1, v_{i, j}] = 0$ for all $i \in [1, n]$ and $j \in [1, k-1]$. Let $\mathcal{C}$ be the vertical concatenation of $\mathcal{C}_1, \mathcal{C}_2,$ and $\mathcal{C}_3$, as shown in Figure~\ref{Wd}.

    \begin{figure}
        \centering
        \includegraphics[width=0.7\linewidth]{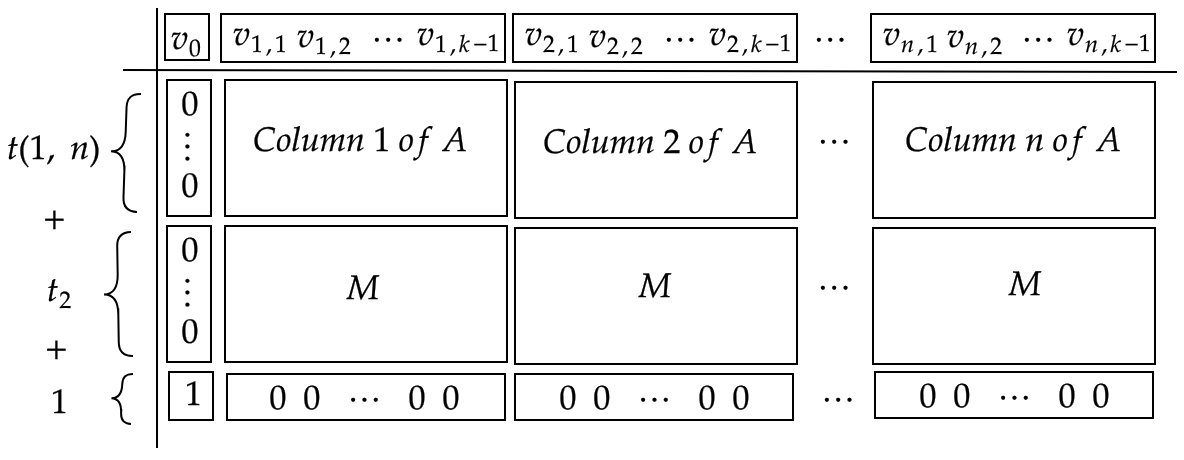}
        \caption{A $\overline{\mathrm{Wd}(k, n)}$-disjunct}
        \label{Wd}
    \end{figure}

\end{Construction}

\begin{theorem}
\label{windmill-cff-bounds}
 Let $n \geq 2$. For $k \geq 4$, 
     $$t(1, (k-1)n) + 1 \leq t(\mathrm{Wd}(k, n)) \leq t(1, n) +  t(2, k-1) + 1.$$
For $k = 3$,
     $$t(1, 2n) + 1 \leq t(\mathrm{Wd}(3, n)) \leq t(1, n) + 3.$$
\end{theorem}

 \begin{proof}
     We show that the matrix $\mathcal{C}$ created in Construction~\ref{windmil-cff-cons} is $\overline{\mathrm{Wd}(k, n)}$-disjunct. Since $\delta(\mathrm{Wd}(k, n)) \geq 2$, by Theorem~\ref{GbarCFF=GCFF} part~\ref{delta=2}, it is enough to prove that $\mathcal{C}$ is $\mathrm{Wd}(k, n)$-disjunct. Let $\{u, v\}$ be an edge in $\mathrm{Wd}(k, n)$ for $k \geq 3$. We need to show that for any other vertex $w$, there exists a row $r$ with $\mathcal{C}[r, u] = \mathcal{C}[r, v] = 0$ and $\mathcal{C}[r, w] = 1$. 

    \noindent \underline{\textbf{Case 1:}} $\{u, v\} = \{v_0, v_{i,j}\}$ for $i \in [1,n]$ and $j \in [1,k-1]$, 

    \noindent \textbf{\textit{Case 1.1:}}  $w = v_{l,m}$, where $l \in [1,n], m \in [1,k-1]$ and $l \neq i$. According to Construction~\ref{windmil-cff-cons}, $\mathcal{C}_{1}^{v_0} = \mathbf{0}_{t(1, n) \times 1}, \mathcal{C}_{1}^{v_{i,j}} = A^i,$ and $\mathcal{C}_{1}^{v_{l,m}} = A^l$. Since $A$ is $1$-disjunct, there is a row $r$ in $\mathcal{C}_1$ (thus, in $\mathcal{C}$) such that $\mathcal{C}_{1}[r, v_0] = \mathcal{C}_{1}[r ,v_{i,j}] = 0$ and $\mathcal{C}_{1}[r, v_{l,m}] = 1$.

        \noindent   \textbf{\textit{Case 1.2:}} $w = v_{i,m}$, where $m \in [1,k-1]$ and $m \neq j$. According to Construction~\ref{windmil-cff-cons}, $\mathcal{C}_{2}^{v_0} = \mathbf{0}_{t_2 \times 1}, \mathcal{C}_{2}^{v_{i, j}} = M^{j},$ and $\mathcal{C}_2^{v_{i, m}} = M^{m}$. Since $M$ is also $1$-disjunct, there is a row $r$ in $\mathcal{C}_2$ such that $\mathcal{C}_{2}[r, v_0] = \mathcal{C}_{2}[r ,v_{i,j}] = 0$ and $\mathcal{C}_{2}[r, v_{i,m}] = 1$. 

       \noindent \underline{\textbf{Case 2:}} $\{u, v\} = \{v_{i,j}, v_{i,m}\}$ for $i \in [1,n]$ and $j, m \in [1,k-1]$ with $j \neq m$. 
       
      \noindent   \textbf{\textit{Case 2.1:}} $w = v_0$. It is easy to see that $\mathcal{C}_{3}[1, v_{i,j}] =  \mathcal{C}_3[1, v_{i,m}] = 0$ and $\mathcal{C}_3[1, v_0] = 1$.
           
           \noindent   \textbf{\textit{Case 2.2:}} 
           $w = v_{i, p}$ with $p \in [1, k-1]$ and $p \neq j, m$. Note that in this case $k \geq 4$. Since $\mathcal{C}_{2}^{v_{i,j}} = M^j, \mathcal{C}_{2}^{v_{i,m}} = M^m,$ and $\mathcal{C}_{2}^{v_{i, p}} = M^p$ and $M$ is $2$-disjunct, there exists a row $r$ in $\mathcal{C}_2$ such that $\mathcal{C}_{2}[r, v_{i, j}] = \mathcal{C}_{2}[r ,v_{i,m}] = 0$ and $\mathcal{C}_2[r, v_{i,p}] = 1$. 
           
           \noindent   \textbf{\textit{Case 2.3:}} $w = v_{l,p}$ for $l \in [1, n], p \in [1, k-1] $ with $l \neq i$. Since $\mathcal{C}_{1}^{v_{i,j}} = \mathcal{C}_{1}^{v_{i,m}} = A^i,$ and $\mathcal{C}_{1}^{v_{l,p}} = A^l$ and $A$ is $1$-disjunct, there is a row $r$ in $\mathcal{C}_1$ (hence, in $\mathcal{C}$) such that $\mathcal{C}_{1}[r, v_{i, j}] = \mathcal{C}_{1}[r ,v_{i,m}] = 0$ and $\mathcal{C}_{1}[r, v_{l,p}] = 1$.

    Thus, we have $t(\mathrm{Wd}(k, n)) \leq t(1, n) +  t(2, k-1) + 1$. Since $S_{(k-1)n+1} \subseteq \mathrm{Wd}(k,n)$, by Proposition~\ref{subgraph-CFF} and $t(S_{(k-1)n+1}) = t(1, (k-1)n) + 1$, by Corollary~\ref{star-optimum}, we have $t(1, (k-1)n) + 1 \leq t(\mathrm{Wd}(k, n))$.
 \end{proof}
 
 \begin{corollary}
\label{friendship}
If $n > \lfloor \tfrac{1}{2} \nextcomb{(\nextcomb{n} + 1)} \rfloor$, then $t(F_{2n+1}) = t(1, n) +  3$. Otherwise, $t(1, n) + 2 \leq t(F_{2n+1}) \leq t(1, n) +  3$.
\end{corollary}

\begin{proof}
  Since $F_{2n+1}$ is a $\mathrm{Wd}(3, n)$, by Theorem~\ref{windmill-cff-bounds}, we obtain $t(1, 2n) + 1 \leq t(F_{2n+1}) \leq t(1, n) +  3$. If $n > \lfloor \tfrac{1}{2} \nextcomb{(\nextcomb{n} + 1)} \rfloor$, then by Proposition~\ref{more-sperner-doubling}, $t(1, 2n) = t(1, n) + 2$ and we obtain $t(F_{2n+1}) = t(1, n) +  3$. If $n \leq \lfloor \tfrac{1}{2} \nextcomb{(\nextcomb{n} + 1)} \rfloor$, then by Proposition~\ref{more-sperner-doubling}, $t(1, 2n) = t(1, n) + 1$ and we obtain $t(1, n) + 2 \leq t(F_{2n+1}) \leq t(1, n) +  3$.
\end{proof}

\section{Values of $t(P_n), t(C_n),$ and $t(W_n)$ for small $n$}
\label{CFFonSmallPathsCycles}

In this section, we determine $t(P_n)$, $t(C_n)$, and $t(W_n)$ for some small values of $n$, as summarized in Table~\ref{tab:small-n-path-cycle-wheel}.

\begin{prop}
\label{small-cycles-234}
If $n = 3, 4$; then $t(P_n) = t(C_n) = n$. 
\end{prop}

\begin{proof}
    For $n = 3, 4$; $t(1, n) = n$, and also, $t(2, n) = n$ by \cite{LVW}. Hence, by Theorem~\ref{trivial-bound}, $t(P_n) = n = t(C_n)$.
\end{proof}

To obtain the exact values of $t(P_n)$ and $t(C_n)$ for some larger values of $n$, we provide Lemmas~\ref{base4-max4} and \ref{max-5-6}. We prove these two lemmas by case analysis. In the case where all subsets in a set system have size $2$, we associate a graph with the set system, as explained in the next remark.

\begin{remark}
\label{G^A}
Let $\mathcal{F} = (\mathcal{X}, \mathcal{B})$ be a set system where $|\mathcal{X}| = t$ and for every $B \in \mathcal{B}$, we have $|B| = 2$. Let $\mathcal{M}$ be the incidence matrix associated with $\mathcal{F}$. Such $\mathcal{F}$ can be depicted as a graph $G$ where $V(G) = \mathcal{X}$ and $E(G) = \mathcal{B}$. 

For a graph $H$, let $\mathcal{F}$ be $H$-CFF. For $\{u, v\} \in E(H)$, let $B_{u} = \{a, b\}$ and $B_v = \{c, d\}$. For $\{x, y\} \subseteq B_u \cup B_v$ and $\{x, y\} \neq B_u, B_v$, $\{x, y\}$ is said to be \emph{covered by sets $B_u$ and $B_v$}. If $\{x, y\}$ is covered by $B_u$ and $B_v$, then by definition of $H$-CFF, $\{x, y\} \notin \mathcal{B}$. Hence, $\{x, y\} \notin E(G)$. The pair $(B_u, B_v)$ can be one of two possible types:

\noindent \underline{\textbf{Case 1:}} \textbf{(Type 1 pair)} $b = d$ and $a \neq b \neq c$. Then, $\{a, c\} \notin E(G)$. Figure~\ref{fig:type1} shows the type $1$ pair $\{a, b\}$ and $\{b, c\}$ as black edges and the covered set $\{a, c\}$ as a non-edge, since $\{a, c\}$ is prohibited in $\mathcal{B}$.

\noindent \underline{\textbf{Case 2:}} \textbf{(Type 2 pair)} $a, b, c, d$ are pairwise distinct. Then, none of the sets $\{a,c\}, \{a,d\}, \{b,c\}, \{b,d\}$ belong to $E(G)$. Figure~\ref{fig:type2} shows the type $2$ pair $(\{a, b\}, \{c, d\})$ as black edges, and the covered subsets that are prohibited in $\mathcal{B}$ as the non-edges.
\end{remark}

\begin{figure}[htbp]
    \centering

    \begin{subfigure}{0.45\textwidth}
        \centering
        \includegraphics[width=0.45\linewidth]{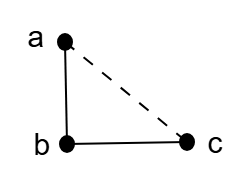}
        \caption{Type $1$ pair: $\{a, b\}$ and $\{b, c\}$}
        \label{fig:type1}
    \end{subfigure}
    \hfill
    \begin{subfigure}{0.45\textwidth}
        \centering
        \includegraphics[width=0.45\linewidth]{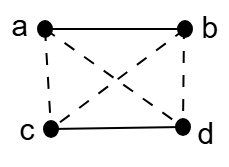}
        \caption{Type $2$ pair: $\{a, b\}$ and $\{c, d\}$}
        \label{fig:type2}
    \end{subfigure}

    \caption{Types of pair $(B_u, B_v)$ with $\{u, v\} \in E(H)$ as described in Remark~\ref{G^A}}
    \label{fig:type1-2pair}
\end{figure}

\begin{lemma}
    Let $\mathcal{F} = (\mathcal{X}, \mathcal{B})$ be a $P_n$-CFF with $|\mathcal{X}| = 4$ and $n \geq 3$. Then $n \leq 4$.
    \label{base4-max4}
\end{lemma}

 \begin{proof}
Let $\mathcal{X} = \{1, 2, 3, 4\}$ and let $V(P_n) = \{v_i: i \in [1, n]\}$ and $E(P_n) = \left \{ \{v_i, v_{i+1}\} : 1 \leq i \leq n-1 \right \}$. We split the proof into the following three cases:
 
\noindent \underline{\textbf{Case 1:}} $\mathcal{B}$ has a set $B$ with $|B| = 3$. Without loss of generality, let $B = \{1, 2, 3\}$ and let $v_{i}$ be the vertex such that $B = B_{v_i}$ for some $i \in [1, n]$. Since there is an edge $\{v_{i}, v_{j}\} \in E(P_n)$ for some $j \in \{i-1, i+1\}$, we have $4 \in B_{v_j}$. Thus, $B_{v_i} \cup B_{v_j} = \mathcal{X}$, which contains any other $B_{v_l}$ with $l \in [1, n] \setminus \{i, j\}$ leading to a contradiction. 

   \noindent \underline{\textbf{Case 2:}} $\mathcal{B}$ has a set $B$ with $|B| = 1$. Without loss of generality, let $B = \{1\}$. By Theorem~\ref{GbarCFF-1CFF}, $\mathcal{F}$ is a $1$-CFF. Since $1$ cannot be an element of other subsets in $\mathcal{F}$, $(\mathcal{X} \setminus \{1\}, \mathcal{B} \setminus \{\{1\}\})$ is also a $1$-CFF. Let $\mathcal{X}' = \mathcal{X} \setminus \{1\}$ and $\mathcal{B}' = \mathcal{B} \setminus \{\{1\}\}$. Since $|\mathcal{X}'| = 3$, then by Theorem~\ref{sperner}, $|\mathcal{B}'| \leq 3$. This implies $n = |\mathcal{B}| \leq 4$.

   \noindent \underline{\textbf{Case 3:}} 
   $|B| = 2$ for all $B \in \mathcal{B}$. Then $\mathcal{B} \subseteq \binom{\mathcal{X}}{2}$. 
   Suppose, for the sake of contradiction, that $n \geq 5$. Let $\mathcal{F}' = (\mathcal{X}, \mathcal{B}')$ where $\mathcal{B}' = \{B_{v_1}, B_{v_2}, B_{v_3}, B_{v_4}, B_{v_5} \} \subseteq \mathcal{B}$. Then, $\mathcal{F}'$ is $P_5$-CFF. Suppose, without loss of generality, that $\{1, 2\} \in \binom{\mathcal{X}}{2} \setminus \mathcal{B}'$ and so, $\mathcal{B}' = \binom{\mathcal{X}}{2} \setminus \{\{1, 2\}\}$. Let $B_{v_i} = \{3, 4\}$ for some $i \in [1, 5]$. It is easy to verify that for any $B \in \mathcal{B}' \setminus \{B_{v_i}\}$, there exists $B' \in \mathcal{B}' \setminus \{B, B_{v_i}\}$ such that $B' \subseteq B_{v_i} \cup B$. Since $v_i$ has at least one neighbor, there must exist one such $B$. Then by Remark~\ref{G^A}, we have $B' \notin \mathcal{B'}$, which further implies that $n < 5$, a contradiction.
\end{proof}

\begin{lemma}
\label{max-5-6}
Let $\mathcal{F} = (\mathcal{X}, \mathcal{B})$ be a $P_n$-CFF with $|\mathcal{X}| = 5$ and $n \geq 3$. Then $n \leq 6$.
\end{lemma}

\begin{proof}
Suppose, by contradiction, $\mathcal{F}$ is a $P_7$-CFF$(5, 7)$ and let $A$ be the
 $5 \times 7$ $\overline{P_7}$-disjunct matrix corresponding to $\mathcal{F}$. Let $V(P_7) = \{v_i: i \in [1, 7]\}$ and $E(P_7) = \left \{ \{v_i, v_{i+1}\} : 1 \leq i \leq 6 \right \}$\footnote{In all the figures in Lemma~\ref{max-5-6}, the vertices are labeled $i$ instead of $v_i$ for $i \in [1, 7]$, for convenience and to save space in the figures.}. We split the proof into the following four cases:
 
\noindent \underline{\textbf{Case 1:}} $\mathcal{B}$ has a set $B$, with $|B| = 4$. Without loss of generality, let $ B = \{1,2,3,4\} = B_{v_i}$ for some $i \in [1, n]$. Then, $5 \in B_{v_j}$ for $j \in \{i-1, i+1\}$. Thus, $B_{v_i} \cup B_{v_j} = [1,5]$, which contains any other $B_{v_l}$ with $l \in [1, n] \setminus \{i, j\}$ leading to a contradiction.

\noindent \underline{\textbf{Case 2:}} $\mathcal{B}$ has a set $B$, with $|B| = 3$ and no other subset $B' \in \mathcal{B}$ exists such that $|B'| > 3$. Without loss of generality, let $ B = \{1, 2, 3\}$. This case is further split into $4$ subcases where $B_{v_i} = \{1, 2, 3\}$ for $i = 1, 2, 3, 4$ since the cases for $i = 5, 6, 7$ would be symmetric for $P_7$. 

\noindent   \textbf{\textit{Case 2.1:}}
    $B_{v_1} = \{1, 2, 3\}$ or  $B_{v_2} = \{1, 2, 3\}$. First, assume  $B_{v_1} = \{1, 2, 3\}$. Since $\{v_1, v_2\} \in E(P_7)$, then $B_{v_2} \nsubseteq B_{v_1}$ which implies that there is a row $r$ with $A[r, v_1] = 0$ and $A[r, v_2] = 1$. Furthermore, for each $l \in \{v_3, v_4, v_5, v_6, v_7\}$ there exists a row $r'$ with $A[r', v_1] = 0, A[r', v_2] = 0,$ and $A[r', l] = 1$. Since the column $A^{v_1}$ has two $0$'s, such row $r'$ is exactly one row with $A[r', v_1] = 0, A[r', v_2] = 0,$ and $A[r', l] = 1$ for all $l \in \{v_3, v_4, v_5, v_6, v_7\}$; see Figure~\ref{fig:subcase1}. Similarly, if $B_{v_2} = \{1, 2, 3\}$, following the same argument, we conclude that there exists a row $r$ with $A[r, v_2] = 0$ and $A[r, v_1] = 1$, as well as exactly one row $r'$ with $A[r', v_1] = 0, A[r', v_2] = 0,$ and $A[r', l] = 1$ for all $l \in \{v_3, v_4, v_5, v_6, v_7\}$; see Figure~\ref{fig:subcase2}. Now in both cases, consider the submatrix $A'$ obtained by restricting $A$ to the columns $A^{v_3}, A^{v_4}, A^{v_5}, A^{v_6},$ and $A^{v_7}$ and to the rows $1, 2, 3,$ and $r$. $A'$ must be a $4 \times 5$ $\overline{P_5}$-disjunct matrix, which is a contradiction by Lemma~\ref{base4-max4}.

    \begin{figure}[h!]
    \centering
    \begin{subfigure}[b]{0.3\textwidth}
        \centering
        
        \includegraphics[width=\textwidth]{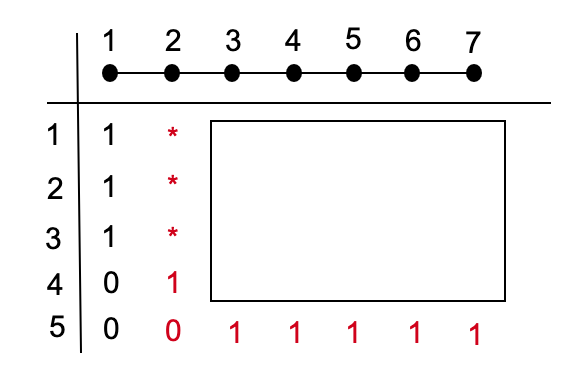}
        \caption{The subset $\{1, 2, 3\}$ corresponds to vertex $1$}
        \label{fig:subcase1}
    \end{subfigure}
    \hspace{0.02\textwidth}
    \begin{subfigure}[b]{0.3\textwidth}
        \centering
        \includegraphics[width=\textwidth]{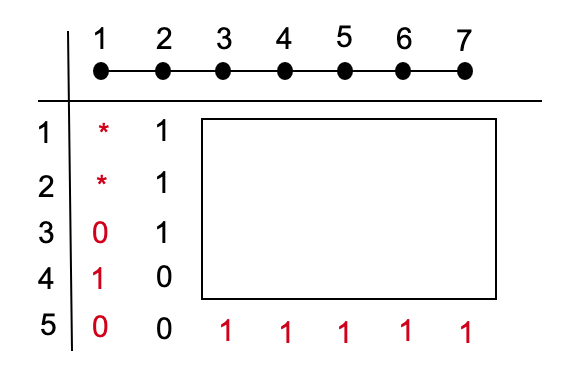}
        \caption{The subset $\{1, 2, 3\}$ corresponds to vertex $2$}
        \label{fig:subcase2}
    \end{subfigure}
    
    \caption{Case 2.1 of Lemma~\ref{max-5-6}}
    \label{fig:main1}
\end{figure}

   \noindent   \textbf{\textit{Case 2.2:}} $B_{v_3} = \{1, 2, 3\}$. Since $\{v_2, v_3\} \in E(P_7)$; by the same argument as in Case 2.1, there exist a row $r \in \{4, 5\}$ with $A[r, v_3] = 0$ and $A[r, v_2] = 1$, and a row $r' \in \{4, 5\} \setminus \{r\}$ with $A[r', v_2] = 0, A[r', v_3] = 0,$ and $A[r', l] = 1$ for all $l \in \{v_1, v_4, v_5, v_6, v_7\}$. Such entries are marked in red in Figures~\ref{fig:subcase3a} and \ref{fig:subcase3b} where $r = 4$ and $r' = 5$. Since $A$ is $\overline{P_7}$-disjunct and  $\{v_3, v_4\} \in E(P_7)$, the row $r$ is the only one such that $A[r, v_3] = 0, A[r, v_4] = 0,$ and $A[r, l'] = 1$ for all $l' \in \{v_1, v_2, v_5, v_6, v_7\}$. These entries are marked in blue in Figures~\ref{fig:subcase3a} and \ref{fig:subcase3b}. Since $\{v_1, v_2\} \in E(P_7)$, for each $l'' \in \{v_3, v_4, v_5, v_6, v_7\}$, there exists a row $m \in \{1, 2, 3\}$ with $A[m, v_1] = A[m, v_2] = 0$, $A[m, l''] = 1$. We split this into two possible cases:

   \noindent \textit{Case 2.2.1:} Suppose there exists exactly one row $m$ such that $A[m, v_1] = A[m, v_2] = 0$, $A[m, l''] = 1$ for all $l'' \in \{v_3, v_4, v_5, v_6, v_7\}$ (See the entries in green in Figure~\ref{fig:subcase3a} where $m = 3$). The submatrix obtained by restricting the columns to $A^{v_5}$, $A^{v_6}$, and $A^{v_7}$ and to the rows $m'' \in \{1, 2, 3\} \setminus \{m\}$ must form a $2 \times 3$ $\overline{P_3}$-disjunct matrix, which is a contradiction by Proposition~\ref{small-cycles-234}.

        \noindent    \textit{Case 2.2.2:} Assume there are two such rows $m_1, m_2 \in \{1, 2, 3\}$ within which the condition $A[m_i, v_1] = A[m_i, v_2] = 0$, $A[m_i, l''] = 1$ is satisfied for $l'' \in \{v_3, v_4, v_5, v_6, v_7\}$ and $i \in \{1, 2\}$. Then, the only remaining row $m_3$ must satisfy the condition $A[m_3, v_1] = 0$ and $A[m_3, v_2] = 1$ (See the green entries in Figure~\ref{fig:subcase3b}). However, $\{v_6, v_7\} \in E(P_7)$ and $B_{v_1} \subseteq B_{v_6} \cup B_{v_7}$, which is a contradiction.

    \begin{figure}[h!]
    \centering
    \begin{subfigure}[b]{0.3\textwidth}
        \centering
        \includegraphics[width=\textwidth]{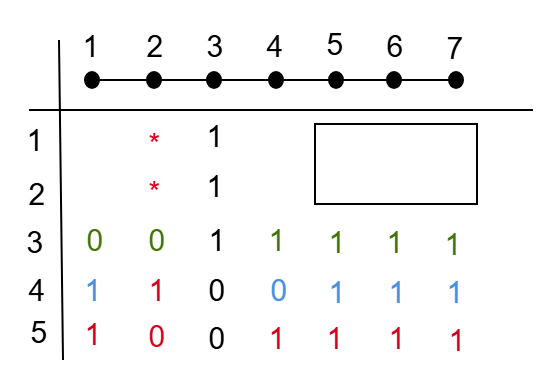}
        \caption{Case 2.2.1 of Lemma~\ref{max-5-6}}
        \label{fig:subcase3a}
    \end{subfigure}
    \begin{subfigure}[b]{0.3\textwidth}
        \centering
        \includegraphics[width=\textwidth]{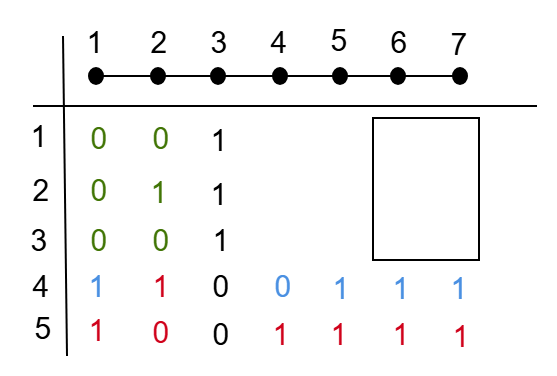}
        \caption{Case 2.2.2 of Lemma~\ref{max-5-6}}
        \label{fig:subcase3b}
    \end{subfigure}
    
    \caption{Case 2.2 of Lemma~\ref{max-5-6}}
    \label{fig:main}
\end{figure}

\noindent   \textbf{\textit{Case 2.3:}} Let $B_{v_4} = \{1, 2, 3\}$. 
Since $\{v_3, v_4\} \in E(P_7)$, there exist a row $r \in \{4, 5\}$ with $A[r, v_4] = 0$ and $A[r, v_3] = 1$, and a row $r' \in \{4, 5\} \setminus \{r\}$ with $A[r', v_3] = 0, A[r', v_4] = 0,$ and $A[r', l] = 1$ for all $l \in \{v_1, v_2, v_5, v_6, v_7\}$ (notice the red entries in Figure~\ref{fig:subcase4} where $r = 4$ and $r' = 5$). 
Furthermore, since $\{v_4, v_5\} \in E(P_7)$, the row $r$ is the one satisfying $A[r, v_4] = 0, A[r, v_5] = 0,$ and $A[r, l'] = 1$ for all $l' \in \{v_1, v_2, v_3, v_6, v_7\}$. These entries are marked in blue.
Since $\{v_6, v_7\} \in E(P_7)$, there exist rows $r_1, r_2 \in \{1, 2, 3\}$ such that $A[r_1, v_6] = 0, A[r_1, v_7] = 1$ and $A[r_2, v_6] = 1, A[r_2, v_7] = 0$; as well as exactly one row $r_3 \in \{1, 2, 3\} \setminus \{r_1, r_2\}$ with $A[r_3, v_6] = A[r_3, v_7] = 0$, and $A[r_3, l''] = 1$ for all $l'' \in \{v_1, v_2, v_3, v_4, v_5\}$ (these entries are specified in purple). 
The row $r_1$ must satisfy the condition $A[r_1, v_5] = A[r_1, v_6] = 0$ and $A[r_1, l'''] = 1$ for all $l''' \in \{v_1, v_2, v_3, v_4, v_7\}$ since $\{v_5, v_6\} \in E(P_7)$ (notice the green entries).
Since $|B_{v_1}| \geq 4 > 3$, this is a contradiction to our initial assumption that no subset in $\mathcal{B}$ other than $B_{v_4}$ has size $> 3$. 

\begin{figure}
    \centering
    \includegraphics[width=0.3\textwidth]{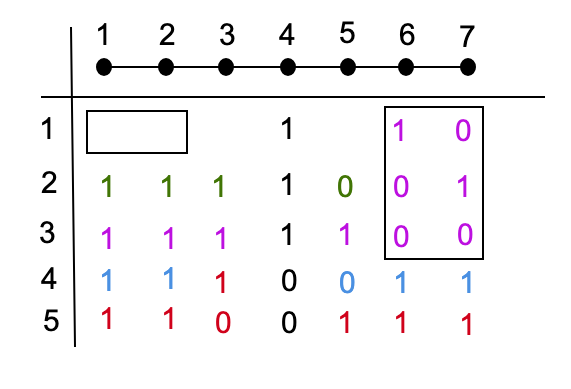}
    \caption{Case $2.3$ of Lemma~\ref{max-5-6}}
    \label{fig:subcase4}
\end{figure}

\noindent \underline{\textbf{Case 3:}} $\mathcal{B}$ has a set $B$ with $|B| = 1$ and no other subset $B' \in \mathcal{B}$ exists such that $|B'| \geq 3$. Without loss of generality, let $B = \{1\}$ and $B = B_{v_i}$ for some $i \in [1, 7]$. 
Since $\overline{P_7}$-disjunct is also $1$-disjunct, $1 \notin B_{v_j}$ for all $j \in [1, 7] \setminus \{i\}$. Now, consider the submatrix $A'$ formed by the rows $2, 3, 4, 5$ of $A$ and the columns $\{A^{v_j} : j \in [1, 7] \setminus \{i\}\}$. Then the $4 \times 6$ matrix $A'$ is also $1$-disjunct, which implies that $\{B_{v_j} : j \in [1, 7] \setminus \{i\}\} = \binom{[2, 3, 4, 5]}{2}$. 
It is easy to verify by Remark~\ref{G^A} that for any $B_j, B_k \in \mathcal{B} \setminus \{B_i\}$, there exists $B_l \in \mathcal{B} \setminus \{B_i\}$ with $B_l \notin \{B_j, B_k\}$ such that $B_l \subseteq B_j \cup B_k$, which is a contradiction.

\noindent \underline{\textbf{Case 4:}} $|B| = 2$ for all $B \in \mathcal{B}$. Let $A$ be the $5 \times 7$ $\overline{P_7}$-disjunct matrix obtained by restricting the incidence matrix of a $1$-CFF$(5, 10)$ to the first $7$ columns. In other words, three sets must be removed from the $1$-CFF$(5, 10)$. The structure of $\mathcal{F}$ is then divided into four subcases, according to the type of sets removed, as provided in Figure~\ref{fig:four}.

\begin{figure}[htbp]
    \centering

    \begin{subfigure}{0.23\textwidth}
        \centering
        \includegraphics[width=0.5\linewidth]{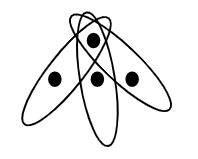}
        \caption{Case 4.1}
        \label{fig:4.1}
    \end{subfigure}
    \hfill
    \begin{subfigure}{0.23\textwidth}
        \centering
        \includegraphics[width=0.5\linewidth]{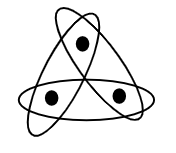}
        \caption{Case 4.2}
        \label{fig:4.2}
    \end{subfigure}
    \hfill
    \begin{subfigure}{0.23\textwidth}
        \centering
        \includegraphics[width=0.5\linewidth]{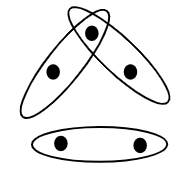}
        \caption{Case 4.3}
        \label{fig:4.3}
    \end{subfigure}
    \hfill
    \begin{subfigure}{0.23\textwidth}
        \centering
        \includegraphics[width=0.8\linewidth]{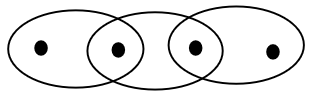}
        \caption{Case 4.4}
        \label{fig:4.4}
    \end{subfigure}

    \caption{Type of the three subsets removed from $1$-CFF$(5, 10)$ as discussed in Case 4 of Lemma~\ref{max-5-6}}
    \label{fig:four}
\end{figure}

\noindent   \textbf{\textit{Case 4.1:}} Without loss of generality, the subsets removed are $\{1,2\}$, $\{1,3\}$, and $\{1,4\}$. The subsets in $\mathcal{F}$ are $\{1,5\}$, $\{2,3\}$, $\{2,4\}$, $\{2,5\}$, $\{3,4\}$, $\{3,5\}$, and $\{4,5\}$. Let $B_{v_1} = \{2, 3\}$. For any $B \in \mathcal{B} \setminus \{B_{v_1}\}$, if $B_{v_2} = B$; then it is easy to see that there exists $B' \in \mathcal{B}$ such that $B' \subseteq B_{v_1} \cup B_{v_2}$, which is a contradiction by Remark~\ref{G^A}. 

    \noindent   \textbf{\textit{Case 4.2:}} Without loss of generality, the subsets removed are $\{1,2\}$, $\{2,3\}$, and $\{1,3\}$. The remaining subsets are $\{1,4\}$, $\{1,5\}$, $\{2,4\}$, $\{2,5\}$, $\{3,4\}$, $\{3,5\}$, and $\{4,5\}$. Let $B_{v_1} = \{4, 5\}$. For any $B \in \mathcal{B} \setminus \{B_{v_1}\}$ with $B_{v_2} = B$; there exists $B' \in \mathcal{B}$ such that $B' \subseteq B_{v_1} \cup B_{v_2}$, which is a contradiction by Remark~\ref{G^A}. 

    \noindent   \textbf{\textit{Case 4.3:}} Without loss of generality, the subsets removed are $\{1,2\}$, $\{1,3\}$, and $\{4,5\}$. The remaining subsets are $\{1,4\}$, $\{1,5\}$, $\{2,3\}$, $\{2,4\}$, $\{2,5\}$, $\{3,4\}$, and $\{3,5\}$. Let $B_{v_1} = \{2, 3\}$. For any choice of $B \in \mathcal{B} \setminus \{B_{v_1}\}$ for $B_{v_2}$, there exists a subset $B' \in \mathcal{B}$ covered by $B_{v_1}$ and $B$. This is a contradiction by Remark~\ref{G^A}. 

    \noindent   \textbf{\textit{Case 4.4:}} Without loss of generality, the subsets removed are $\{1,2\}$, $\{2,3\}$, and $\{3,4\}$. The remaining subsets are $\{1,3\}$, $\{1,4\}$, $\{1,5\}$, $\{2,4\}$, $\{2,5\}$, $\{3,5\}$, and $\{4,5\}$. This case, however, is a little different from the above cases. Notice that that for every $B \in \mathcal{B}$, there exists a $B' \in \mathcal{B}$ such that there exists a subset $\{x, y\} \subseteq B \cup B'$ with $\{x, y\} \notin \mathcal{B}$. We list all pairs $(B, B')$ below:
    
\begin{center}
       \noindent
\textbullet\ $(\{1, 3\}, \{1, 4\})$ \;
\textbullet\ $(\{1, 4\}, \{2, 4\})$ \;
\textbullet\ $(\{1, 5\}, \{2, 5\})$ \;
\textbullet\ $(\{2,5\}, \{3, 5\})$ \;
\textbullet\ $(\{3,5\}, \{4, 5\})$
\end{center}

Hence, Figure~\ref{Subcase4.4B} shows the maximal graph $H = (V, E)$ such that $\mathcal{F}$ is $H$-CFF with each vertex $v_i$ corresponding to a subset $B_{v_i}$ with $i \in [1, 7]$ in $\mathcal{F}$. 
This is a contradiction since $H \subseteq P_7$ and $H \neq P_7$. 

\begin{figure}
    \centering
    \includegraphics[width=0.5\linewidth]{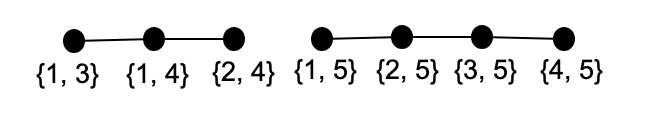}
    \caption{\, Maximal graph $H$ in Case $4.4$ of Lemma~\ref{max-5-6}}
    \label{Subcase4.4B}
\end{figure}

Hence, in all the cases considered, we show that such $A$ does not exist. Hence, $n < 7$. 
\end{proof}

Lemmas~\ref{base4-max4} and \ref{max-5-6} are important results used to obtain the exact values of $t(P_n)$ and $t(C_n)$ for some small values of $n$, as shown in the following theorem.

\begin{theorem}
\label{small-cycles-678}
\begin{enumerate}
\item $t(P_5) = t(C_5) = t(P_6) = t(C_6) = 5$.
\item $t(P_7) = t(C_7) = t(P_8) = t(C_8) = t(P_9) = t(C_9) = t(P_{10}) = 6$.
\end{enumerate}
\end{theorem}

\begin{proof}

We provide a proof for each item.
\begin{enumerate}
    \item By Proposition~\ref{subgraph-CFF}, Corollary~\ref{trivial-bound}, and $t(2, 5) = 5$ (see Table~\ref{tab:t2n-small}), we have that $4 \leq t(P_5) \leq t(C_5) \leq 5$. However, in Lemma~\ref{base4-max4} we proved that  with a ground set of size $4$, a CFF on a path has at most $4$ columns. Hence, $t(P_5) = t(C_5) = 5$. By Proposition~\ref{subgraph-CFF} and Theorem~\ref{maximal-path-cycle-thm}, we have $ t(P_5) \leq t(P_6) \leq t(C_6) \leq 5$. Since $t(P_5) = 5$, we have $t(P_6) = t(C_6) = 5$. 

    \item By Proposition~\ref{subgraph-CFF}, Theorem~\ref{maximal-path-cycle-thm}, and the above, we have $5 = t(P_6) \leq t(P_7) \leq t(C_7) \leq 6$. Since by Lemma~\ref{max-5-6}, we have that with a ground set of size $5$, a CFF on a path has at most $6$ columns. Thus, $t(P_7) = t(C_7) = 6$. By Proposition~\ref{subgraph-CFF} and Theorem~\ref{maximal-path-cycle-thm}, we have $t(P_7) \leq t(P_8) \leq t(C_8) \leq 6$. Since $t(P_7) = 6$, thus, $t(P_8) = t(C_8) = 6$. Again, by Proposition~\ref{subgraph-CFF} and Theorem~\ref{maximal-path-cycle-thm}, $6 = t(P_8) \leq t(P_9) \leq t(C_9) \leq 6$. Hence, $t(P_9) = t(C_9) = 6$. 

Let $V(P_{10}) = \{v_i : i \in [1 ,10]\}$ and $E(P_{10}) = \left\{ \{v_i, v_{i+1}\} : i \in [1,9]\} \right\}$. With the ground set $[1, 6]$, let $\mathcal{F} = \{B_{v_i} : i \in [1, 10]\}$ be the set system where 
\[
\begin{array}{lllll}
B_{v_1} = \{1, 2, 3\} & B_{v_3} = \{1, 2, 5\} & B_{v_5} = \{1, 3, 5\} & B_{v_7} = \{2, 3, 5\} & B_{v_9} = \{2, 4, 6\} \\
B_{v_2} = \{1, 2, 4\} & B_{v_4} = \{1, 5, 6\} & B_{v_6} = \{3, 4, 5\} & B_{v_8} = \{2, 3, 6\} & B_{v_{10}} = \{4, 5, 6\} \\
\end{array}
\]
It is easy to check that $\mathcal{F}$ is a $P_{10}$-CFF.  Since $6 = t(P_9) \leq t(P_{10}) \leq 6$, we have $t(P_{10}) = 6$.
\end{enumerate}
 \end{proof}

\begin{op}
\label{op-pathCFF-from-Sperner}
Just as in the case of the $P_{10}$-CFF built with $[1,6]$ by carefully eliminating subsets from its largest Sperner family, it would be interesting to investigate whether this technique can be generalized for large $n$, to see whether this can improve the bound presented in Theorem~\ref{maximal-path-cycle-thm}. 
\end{op}

\begin{corollary} Let $Q_n$ be the Hamming graph $H(n, 2)$. Then, $t(Q_2) = 4$ and $t(Q_3) = 6$.
\end{corollary}

\begin{proof}
   Since $H(n, 2) = \underbrace{K_2 \; \square \; \cdots \; \square \; K_2}_{n \text{ times }}$ and it is Hamiltonian, by Propositions~\ref{subgraph-CFF} and \ref{Gray-maximal-hamming}, we have $t(C_{2^n}) \leq t(Q_n) \leq 2n$. For $n = 2$ and $3$, since $t(C_4) = 4$ (by Proposition~\ref{small-cycles-234}) and $t(C_8) = 6$ (by Proposition~\ref{small-cycles-678}), we have the desired results. 
\end{proof}

We now determine $t(W_n)$ for small values of $n$. 
 \begin{corollary} 
 \label{wheel-small}
 \begin{enumerate}
    \item $t(W_5) = 5$. 
     \item $t(W_6) = t(W_7) = 6$.
     \item $t(W_8) = t(W_9) = t(W_{10}) = 7$.
 \end{enumerate}
\end{corollary}

\begin{proof}
In $W_n$, we label the universal vertex by $0$, and label the other vertices forming a $C_{n-1}$ as $[1, n-1]$.
    \begin{enumerate}
        \item By Corollary~\ref{wheel}, $\max{\{t(C_4), t(C_5), t(1, 4) + 1\}} \leq t(W_5) \leq t(C_4) + 1$. By Equation~\ref{max-sperner-family}, $t(1, 4) = 4$, by Proposition~\ref{small-cycles-234}, $t(C_4) = 4$ and by Theorem~\ref{small-cycles-678}, $t(C_5) = 5$. Thus, $t(W_5) = 5$. 
        
        \item By Corollary~\ref{wheel}, $ t(C_6) \leq t(W_6) \leq t(C_5) + 1$. By Theorem~\ref{small-cycles-678}, $t(C_5) = 5$ and $t(C_6) = 5$. Thus, we obtain $5 \leq t(W_6) \leq 6$. Since $t(1, 5) = 4$ (by Equation~\ref{max-sperner-family}) and $t(C_5) = t(1, 5) + 1$, by Corollary~\ref{star-gen-cons-special-case} we have $t(W_6) = 6$. By Corollary~\ref{wheel}, $t(C_7) \leq t(W_7) \leq t(C_6) + 1$. By Theorem~\ref{small-cycles-678}, $6 \leq t(W_7) \leq 6$. Thus, $t(W_7) = 6$. 

        \item Since $t(1, 7) = 5$ and by Theorem~\ref{small-cycles-678}, $t(C_7) = 6 = t(C_8)$; by Corollary~\ref{wheel} we have $6 \leq t(W_8) \leq 7$. Since $t(C_7) = t(1, 7) + 1$, by Corollary~\ref{star-gen-cons-special-case} we have $t(W_8) = 7$. Since $t(C_8) = 6$ and $t(C_9) = 6$ (by Theorem~\ref{small-cycles-678}) and $ t(1, 8) = 5$; by Corollary~\ref{wheel} we obtain $6 \leq t(W_9) \leq 7$. Since  $t(C_8) = t(1, 8) + 1$, by Corollary~\ref{star-gen-cons-special-case} we have $t(W_9) = 7$. Since $P_{10} \subseteq W_{10}$, by Proposition~\ref{subgraph-CFF} $t(P_{10}) \leq t(W_{10})$. Since by Theorem~\ref{small-cycles-678}, $t(P_{10}) = 6$, we have $t(W_{10}) \geq 6$. By Corollary~\ref{wheel}, $t(W_{10}) \leq t(C_9) + 1$ and by Theorem~\ref{small-cycles-678}, $t(C_9) = 6$. Thus, we obtain $6 \leq t(W_{10}) \leq 7$. Since $t(1,9) = 5$ and $t(C_9) = t(1, 9) + 1$, by Corollary~\ref{star-gen-cons-special-case} we have $t(W_{10}) = 7$.
    \end{enumerate}
\end{proof}

\section{Conclusion and Further directions}
\label{future}

This paper continues the work done in \cite{IM1,thaismourastructureaware} and is the first to systematically study CFFs based on a graph structure. We provided several bounds on $t(G)$, along with families of graphs achieving the extreme bounds, as well as tight and non-trivial asymptotic bounds for CFFs on several families of graphs. 

Along with some open problems mentioned earlier (Open Problems~\ref{op-E2m-existence}, \ref{op-E2m-categorize}, and \ref{op-pathCFF-from-Sperner}), we ask the following questions that are also worth investigating:

\begin{op}
    Since $t(K_n) = t(2, n)$, it is also worth noting that some dense graphs $G \subseteq K_n$ may have $t(G) = t(2, n)$. Our question is: what is the maximum number of edges that can be deleted from $K_n$ such that the resultant graph $G$ still has $t(G) = t(2, n)$? A related question is: what is the minimum number of edges that need to be deleted from $K_n$ such that the resultant graph $G$ has $t(G) < t(2, n)$?
\end{op}

\begin{op}
     Does there exist a family of sparse graphs $G$ (possibly highly structured) such that $t(G) = t(2, |V(G)|)$?
\end{op}

In addition, finding exact values of or bounds on $t(G)$ for several other classes of graphs would be another interesting direction of research.

\bibliography{References}
\bibliographystyle{plain}

\appendix

\section{Proof of Propositions~\ref{cyclic-even-m1} and \ref{loc-rec-fixed-proof} }

\subsection{Proof of Proposition~\ref{cyclic-even-m1}}
\label{cyclic-even-m1-proof}
\begin{proof}

From Construction~\ref{mixed-radix-gray-reflected-code-recursive}, it is easy to see that the first codeword is the all-zero tuple, and if $m_1$ is even, the last codeword is an $n$-tuple starting with $m_1 - 1$, followed by all zeros, since we append the reverse order of $R_{n-1}^{(m_2, \ldots, m_n)}$ to $m_1 - 1$. Thus, $R_{n}^{(m_1, m_2, \ldots, m_n)}$ is cyclic.

For the other direction, assume $R_{n}^{(m_1, m_2, \ldots, m_n)}$ is cyclic. Then, the first and last codewords also have Hamming distance $1$. Hence, by Construction~\ref{mixed-radix-gray-reflected-code-recursive}, the last codeword must be the $n$-tuple starting with $m_{1}-1$ followed by all $0$ and this is possible only if $m_1$ is even.
\end{proof}

\subsection{Proof of Proposition~\ref{loc-rec-fixed-proof}}
\label{loc-rec-fixed-proof-exp}
\begin{proof}
The proof is by induction on $n$. The result can be easily verified for $n = 1$.

\textbf{Induction Hypothesis:} Assume that $M_{n-1}^{q}$ is a cyclic modular Gray code for some integer $n \geq 2$. We will prove that $M_{n}^{q}$ is a cyclic modular Gray code.

It is clear that $M_{n}^{q}$ contains all the $q^n$ $n$-tuples. 
 For the order of the codewords listed in $D_i \cdot (q - i \mod q)^{+}$ for all $0 \leq i \leq q^{n-1}-1$, it is clear that the rightmost digit of the codewords changes following the general transitions $0 \rightarrow 1 \rightarrow 2 \rightarrow \cdots \rightarrow q - 1 \rightarrow 0 \rightarrow 1 \rightarrow 2 \rightarrow \cdots$.

Now we consider two consecutive codewords, where one of them is the last codeword of $D_i \cdot l^+$ and the other is the first codeword of $D_{i+1} \cdot m^+$, with $0 \leq i \leq q^{n-1} - 2$, $l = q - i \mod q$, and $m = q - i - 1 \mod q$. Then the last codeword of $D_i \cdot l^+$ is $D_i \cdot (2q - i - 1 \mod q)$, which is $D_i \cdot (-i - 1 \mod q)$, and the first codeword of $D_{i+1} \cdot m^+$ is $D_{i+1} \cdot (q - i - 1 \mod q)$, which is $D_{i+1} \cdot (-i - 1 \mod q)$. Since, by the Induction Hypothesis, $D_i$ and $D_{i+1}$ are two consecutive codewords in $M_{n-1}^{q}$, which is cyclic and modular, $D_i \cdot (-i - 1 \mod q)$ and $D_{i+1} \cdot (-i - 1 \mod q)$ are also codewords following the general transition of a cyclic modular Gray code.

The above claim also holds for the first and last codewords in $M_{n}^{q}$, which are $0^n$ and $(q - 1)0^{n - 1}$. Therefore, $M_{n}^{q}$ is a cyclic modular Gray code.
\end{proof}

\end{document}